\documentclass[12pt,oneside,final]{amsart}
\usepackage{stmaryrd,bm}
\title{Inverse problem for the geometric Navier-Stokes equations}
\author[Y.Kian, L. Oksanen, Z. Zhao]{Yavar Kian$^1$, Lauri Oksanen$^2$, Ziyao Zhao$^3$}

\IfFileExists{tweakslo.sty}{\usepackage{tweakslo}}{
\usepackage{amssymb,thmtools,mathtools,todonotes,amsmath}
\declaretheorem{theorem,definition,lemma,proposition,corollary,remark}}
\providecommand{\HOX}[1]{\todo[noline]{x1}}
\providecommand{\TODO}[1]{\todo[inline]{x1}}
\def\p{\partial}
\def\o2{\overline{O_2}}

\def\RR{\mathbb R}

\def\C{\mathbb C}

\def\M{M}
\def\S{\mathcal S}

\def\P{\mathcal P}
\def\cU{\mathcal U}
\def\cN{\mathcal N}
\def\cB{\mathcal B}
\def\cV{\mathcal V}
\def\cT{\mathcal T}
\def\cC{\mathcal C}

\def\dist{\text{dist}}

\def\bu{\mathbf{u}}
\def\bv{\mathbf{v}}
\def\bw{\mathbf{w}}
\def\bf{\mathbf{f}}
\def\bg{\mathbf{g}}
\def\bh{\mathbf{h}}
\def\tbu{\widetilde{\mathbf{u}}}

\def\grad{\text{grad}}
\def\diam{\text{diam}}
\def\inter{\text{int}}

\def\Re{\text{Re}}

\DeclareMathOperator{\supp}{supp}
\DeclarePairedDelimiter{\norm}{\lVert}{\rVert}
\DeclarePairedDelimiter{\bignorm}{\bigg\lVert}{\bigg\rVert}
\DeclarePairedDelimiter{\abs}{\lvert}{\rvert}

\DeclarePairedDelimiter{\inner}{\langle}{\rangle}

\newcommand{\bphi}{\bm{\phi}}
\newcommand{\bPhi}{\bm{\Phi}}
\newcommand{\tbPhi}{\widetilde{\bm{\Phi}}}
\newcommand{\tbPsi}{\widetilde{\bm{\Psi}}}
\newcommand{\bPsi}{\bm{\Psi}}
\newcommand{\bpsi}{\bm{\psi}}

\begin{document}\maketitle
\begin{abstract}
    We consider the inverse problem of determining a compact Riemannian manifold with boundary from fixed time observations of the solution, restricted to a small subset in space, for the Navier-Stokes system with a local source on the manifold. Our approach is based on a reduction to an inverse problem for an auxiliary hyperbolic Stokes system, via  linearization and spectral techniques. We solve the resulting inverse problem by a new generalization of the Boundary Control method.

    {

}
    
    
\end{abstract}

\renewcommand{\thefootnote}{\fnsymbol{footnote}}
\footnotetext{\hspace*{-5mm} 
\begin{tabular}{@{}r@{}p{16cm}@{}} 
$~^1$
& Univ Rouen Normandie, CNRS, Laboratoire de Math\'{e}matiques Rapha\"{e}l Salem, UMR 6085, F-76000 Rouen, France.  E-mail: \href{mailto:yavar.kian@univ-rouen.fr}{\texttt{yavar.kian@univ-rouen.fr}}\\
$~^2$ 
& Department of Mathematics and Statistics, University of Helsinki,
P.O 68, 00014, University of Helsinki. E-mail: \href{lauri.oksanen@helsinki.fi}{\texttt{lauri.oksanen@helsinki.fi}}.\\
$~^3$ 
& Department of Mathematics and Statistics, University of Helsinki,
P.O 68, 00014, University of Helsinki. E-mail: \href{ziyao.zhao@helsinki.fi}{\texttt{ziyao.zhao@helsinki.fi}}.\end{tabular}}
\section{Introduction}
The Navier–Stokes equations are among the most fundamental mathematical models in fluid dynamics. 

In this work, we consider the geometric Navier--Stokes equation on a Riemannian manifold $(\M, g)$, see  \eqref{eq:NS1} below. When $\M$ is a Euclidean domain and $g = c dx^2$, with $c > 0$ a constant, then \eqref{eq:NS1} can be written as 
the classical Navier-Stokes equations on $\Omega$ with no-slip boundary condition 
\begin{equation}
\begin{cases}
    \p_t \bu + \bu \nabla \bu - \mu \Delta \bu + \nabla p = \bf,\text{ in }[0,T]\times\M\\
    \nabla\cdot \bu = 0,\text{ in }[0,T]\times\M\\
    \bu=0,\text{ on }[0,T]\times \p\M,\\
    \bu(0,\cdot)=\bu_0.
\end{cases}
\end{equation}
Here $\bu: [0,T]\times\Omega\to \mathbb{R}^n$ is the velocity field of the fluid, $p: [0,T]\times\Omega\to \RR$ is a scalar function representing the pressure, and 
$\mu = c^{-1}$ is the coefficient of kinematic viscosity.

The generalization on a Riemannian manifold enables the study of fluid dynamics in complex geometric settings, thereby opening new perspectives for both theoretical analysis and practical applications \cite{VlKh,EMa,MaRaSh,Sh}. The formulation of the Navier–Stokes equations on manifolds is motivated by several important areas, including: 
1) geophysical fluid dynamics \cite{SaT,Val}, where the use of Riemannian geometry allows one to model the curvature of the Earth’s surface;
2) General Relativity and cosmology \cite{FMR}, modeling relativistic fluids, such as those present in neutron stars, black hole accretion disks, and the early universe;
3)~fluid dynamics on biological membranes and microfluidic devices~\cite{STONE}.

In all these contexts, the geometric structure of the manifold plays a central role in characterizing the physical properties of fluids and their surrounding media (e.g., density and viscosity). Motivated by these considerations, we investigate the inverse problem of identifying a Riemannian manifold from measurements of the fluid velocity field, associated with a source-to-final value map for the Navier–Stokes equations.

Inverse problems and, more broadly, identification problems related to fluid motion remain poorly understood. Only a limited number of studies have addressed this class of inverse problems, despite their strong physical motivation and mathematical significance. Among the few existing contributions devoted to inverse problems for the Stokes and Navier–Stokes equations, we may cite works concerning the determination of viscosity from boundary or internal measurements \cite{FDJN,HLiWa,ImYa,LaUWa,LiWa,Liu},  the identification of objects immersed in a fluid \cite{CaCo1,BCZ,CaCo2} as well as inverse source problems \cite{BGKN,CIPY,ILY}. To the best of our knowledge, no prior studies have addressed inverse problems for the Navier–Stokes equations beyond the Euclidean setting.


  \subsection{Notations and main results} 
  Let $\M$ be a n-dimensional, $n\geq2$, smooth connected Riemannian manifold with  boundary and $\omega\subset \M$ be a nonempty open subset.
  We denote the space consisting of smooth sections of exterior $k$-form bundles on $\M$ and $\omega$ by $\Omega^k(\M)$ and $\Omega^k(\omega)$, respectively. 
  Let
  \begin{equation}
      \ast:\ \Omega^k (\M)\to \Omega^{n-k}(\M)
  \end{equation}
  be the Hodge star operator (see e.g. \cite[section 14.1a]{frankel} for details). Then we define the codifferential operator $d^*: \Omega^k(\M)\to \Omega^{k-1}(\M)$ as
  \begin{equation}
  \label{eq:def_codifferential}
    d^*=(-1)^{n(k+1)+1}\ast d\ast.
  \end{equation}
  We denote the Hodge Laplacian $-dd^*-d^*d$ by $\Delta_H$.
   Using the metric $g$, for every 1-form $u=u_idx_i$ on $\M$ we associate the vector field $u_*=g^{ij}u_j\partial_{x_i}\in T\M$. In the same way, for every vector field $v=v_i\partial_{x_i}\in TM$ we associate the 1-form $v^*=g_{ij}v_jdx_i$. With these notations, we write $R(w)=(\nabla_{w_*}w_*)^*$,
where $\nabla$ denotes the Levi-Civita connection on $(\M,g)$.  Several works have been devoted to the study  of the following geometric Navier-Stokes system 
  \begin{equation}
    \label{eq:NS1}
    \begin{cases}
      \p_t\bu-\Delta_H\bu+R(\bu) + d p=\bf,\ \text{in } [0,T]\times \M,\\
      d^* \bu=0,\ \text{in } [0,T]\times \M,\\
      \bu|_{[0,T]\times\p\M}=0,\\
      \bu=0,\ p=0,\ \text{in }\{0\}\times\M.
    \end{cases}
  \end{equation}
  Without being exhaustive, one can refer to \cite{EMa,AA91,CP92,TW93,CRT99} for more details.
  
  We denote by $\mathbb B_r$ the set defined by
  $$\mathbb B_r:=\{\bg \in \cC_\alpha(M):\ \supp(\bg)\subset (0,T]\times\omega,\  \|\bg\|_{ \cC_\alpha(M)}\leq r\},$$
where we refer  to Section \ref{sec:Nonlinear_NS} for the  definition of the space $\cC_\alpha(M)$ and $\norm{\cdot}_{ \cC_\alpha(M)}$. We prove in Theorem \ref{thm:NS} that 
there exists $\delta>0$, depending on $(\M,g)$ and $T$, such that for all $\bf\in \mathbb B_\delta$, problem \eqref{eq:NS1} admits a solution $(\bu,p)=(\bu_\bf,p_\bf)$ with $\bu\in \mathcal C_\alpha$. The solution is unique in the sense that $\bu$ is unique and $p$ is unique up to a map depending only on $t\in[0,T]$. Then, we can define the local source-to-final value map
\begin{align}
\label{eq:def_N}
    \mathcal N:\mathbb B_\delta&\to L^2\Omega^1(\omega),\\
    \mathcal N(\bf)&=\bu_\bf(T,\cdot)|_\omega,
  \end{align}
  with $\bu_\bf\in \cC_\alpha(M)$  solving \eqref{eq:NS1}.
  Our main result can be stated as follows.
  \begin{theorem}
  \label{thm:main}
      Let $(\M_1,g_1)$ and $(\M_2,g_2)$ be two  n-dimensional smooth connected Riemannian manifolds with  boundary. Let $\omega_1\subset \M_1$ and $\omega_2\subset\M_2$ be open and connected. For $j=1,2$, let $\mathcal N_j$ be the local source-to-final value map
      \eqref{eq:def_N} with $(\M,g)=(\M_j,g_j)$
      and $\omega=\omega_j$.
      In addition, suppose that there is a diffeomorphism $\bPhi:\omega_1\to \omega_2$ such that 
      \begin{equation}\label{NN}
          \cN_1\bPhi^\ast=\bPhi^\ast \cN_2.
      \end{equation}
       Then $(\M_1,g_1)$ and $(\M_2,g_2)$ are isometric.
  \end{theorem}
  \begin{remark}
      By making $\omega_1$ and $\omega_2$ smaller, connectedness can be assumed without loss of generality.
  \end{remark}
\subsection{Comparision to previous literature}

To the best of our knowledge, Theorem \ref{thm:main} is the first result establishing the recovery of the general geometric structure of a manifold from local velocity measurements, at the final time $t=T$, associated with solutions of the Navier-Stokes equations. By contrast, the existing literature is confined to the identification of the viscosity coefficient from measurements related to the stationary Navier-Stokes equations \cite{HLiWa,LaUWa,LiWa}.

Our proof of Theorem \ref{thm:main} is based on application of the Boundary Control (BC) method, which remains one of the most powerful tools for solving inverse coefficient problems for partial differential equations (PDEs) in the time domain. The BC method was first introduced by Belishev in \cite{Beli} and thereafter further developed for solving inverse coefficient determination problems related to the wave equation in Euclidean domains \cite{ABI92,B87,B90,BK89,BK87,BKu89,BK91}. Later, this method was extended to manifold reconstruction \cite{BeKu}. It has been shown that the BC method can be applied to other classes of scalar PDEs \cite{FGKU,HLYZ,KaKuLaMa,KLLY,KOSY}.

The method appears to be less powerful in the context of systems of PDEs, however, inverse problems for the Maxwell \cite{BeIs,KuLaSo} and Dirac \cite{LK09} systems has been solved using it. Further, in \cite{KOP18} the BC method was generalized for abstract wave equations associated to connection Laplacians on vector bundles (for example, a linearized equation satisfied by the Higgs field on the pre-quantum level). To the best of our knowledge, this work constitutes the first application of the BC method to the Stokes system. The approach consists in first establishing a rigorous correspondence between the data for the Stokes system and those for an auxiliary hyperbolic Stokes problem through spectral data, as developed in Section 5. The BC method is then implemented for this auxiliary hyperbolic Stokes problem in Section 7.

It is important to emphasize that the extension of the BC method to this setting cannot be derived from existing results in the literature. The principal obstruction arises from the divergence-free constraint inherent to the Stokes system, which introduces substantial technical and structural difficulties at every stage of the method. In Section 7, we provide a detailed exposition of the adaptation of the BC method to this framework, and we show how the difficulties encountered at each step can be resolved by means of carefully tailored geometric and analytic arguments.


In inverse problems concerning recovery of geometry on manifolds with boundary, measurements are typically taken on the boundary and modeled by the Dirichlet-to-Neumann map or the associated Cauchy data \cite{AKKLT04,KK98,Oksanen2014,KaKuLaMa}. By contrast, inverse problems based on the local source-to-solution map have been predominantly studied on closed manifolds. In particular, it is shown in \cite{Saksala2018,HLYZ,SS25} that the local source-to-solution map for various linear evolution equations determines the underlying manifold up to an isometry. A related result for the wave equation, where the source and observations are supported in disjoint sets, is established in \cite{LNOY24}. 
In this work, we consider the local source-to-solution map on manifolds with boundary.

The approach adopted in Sections 4 and 5 to address the nonlinearity relies on a linearization method that has proved effective in the analysis of numerous inverse problems for nonlinear PDEs. This method, originally introduced in \cite{I93,IS94} in the context of parabolic and elliptic partial differential equations, consists in establishing that the first-order linearization of the data coincides with the data associated with an appropriate linearized equation. In contrast with the case of classical scalar partial differential equations (see, for instance, \cite{L25} for a comprehensive review), the application of the linearization method to the Navier–Stokes system involves both the velocity field $\bu$ and the pressure $p$. This necessitates a suitable representation of the pressure $p$, which is introduced in Theorem \ref{thm:NS}, together with the corresponding linearization result stated in Proposition \ref{linearization}. To the best of our knowledge, the only existing works employing techniques closely related to this linearization approach for the Navier–Stokes system are \cite{LaUWa,LiWa}, where the analysis is restricted to the Euclidean setting and to stationary Navier–Stokes equations in dimensions two and three. By contrast, our analysis is carried out in H\"older spaces, which allows for an extension to higher-dimensional settings.

Following the linearization step, in Section 5 we make use of analytic and representation properties of solutions to the linear Stokes system in order to establish a connection between the local source-to-final-value map and a source-to-solution map associated with an auxiliary hyperbolic system. The inverse problem for this auxiliary hyperbolic system is subsequently analyzed by means of the BC method. This data transfer technique by mean of spectral data has previously been employed in the study of inverse coefficient problems for various scalar evolution partial differential equations \cite{CK,HLYZ,KaKuLaMa,KLLY,KOSY}. As far as we know, the present work constitutes the first application of this methodology to Stokes systems.

\subsection{Outline of the paper}

The outline of the paper is as follows. In Section \ref{sec:elliptic}, we establish regularity results for the stationary Stokes system on manifolds and analyze properties of the associated spectral data, which prepare the terrain for the study of the inverse problem. Section \ref{sec:Nonlinear_NS} is devoted to the existence, uniqueness and appropriate regularity of the solution to the geometric Navier-Stokes equation \eqref{eq:NS1}. In Section \ref{sec:parabolic}, we consider the non-stationary Stokes equation, which arises as the linearization of \eqref{eq:NS1}, and show that the data $\cN$ determines the corresponding restricted source-to-solution operator $\P$, defined as \eqref{eq:def_P}, for the non-stationary Stokes equation. Moreover, we show that $\P$ is equivalent to the localized spectral projections of the stationary Stokes system, which in turn allows the data to be converted into the restricted source-to-solution operator for an auxiliary hyperbolic Stokes system \eqref{eq:problem_hypo}. Section \ref{sec:hyperbolic} analyzes \eqref{eq:problem_hypo} and develops several technical tools, including conditional finite speed of propagation and approximate controllability. The new feature of these tools is that they hold only on a finite time interval determined by the underlying geometry. Combining these techniques with an adapted BC, our main result Theorem \ref{thm:main} is proved in Section \ref{sec:mian_proof}.

\section{Preliminaries}
Let $M$ be a compact $n-$dimensional Riemannian manifold with boundary and $\{U_\alpha,\varphi_\alpha\}_{\alpha\in A}$ be a finite smooth atlas for $M$. For any non-negative integer $m$, we denote by $C^m\Omega^k(M)$ the space of $m-$times continuously differentiable $k-$form. 

Let $\chi_\alpha: T^\ast M|_{U_\alpha} \to U_\alpha\times \RR^n$ be the local trivialization of covector field associated with the coordinates $(U_\alpha,\varphi_\alpha)$. That is, for $\bu|_{U_\alpha} = u_i dx^i$, $\chi_\alpha (\bu) = (u_1,u_2,\cdots, u_n)\in \RR^n$.
For $s\in (0,1)$, we define the space of 1-forms that are $m-$times continuously differentiable and whose $m-th$ derivatives are bounded and H\"older continuous with exponent $s$ as 
\begin{align}
    C^{m,s}\Omega^1(M): = &\{\bf\in C^m\Omega^1(M)\mid  \p_x^{\beta}(\chi_\alpha\circ\rho_\alpha\bf\circ\varphi_\alpha^{-1})\in C^{s}(\RR^n),\\ 
    & \text{for all }\alpha\in A,\ \beta\in(\mathbb N\cup\{0\})^n,\ \abs{\beta}=m\},
\end{align}
where $\{\rho_\alpha\}_{\alpha\in A}$ is a subordinated partition of unity for open cover $\{U_\alpha\}_{\alpha\in A}$.
And we define the corresponding H\"older norm as
\begin{equation}
    \norm{\bf}_{C^{m,s}\Omega^1(M)} : = \sum_{\alpha\in A} \norm{\chi_\alpha\circ\rho_\alpha\bf\circ\varphi_\alpha^{-1}}_{C^{m,s}(\RR^n)}.
\end{equation}
Similarly, we define $C^{k,r}([0,T];C^{m,s}\Omega^1(M))$ as the space of 1-forms that is $C^{k,r}$ in time and $C^{m,s}$ in space equipped with the corresponding H\"older norm (see e.g. \cite[Appendices B.5]{Ruelle89}).

Recall that the Riemannian metric naturally induces bundle metric on all tensor bundle over $M$. We write $(\cdot,\cdot)_{\Omega^k}$ the induced metric on $k-$forms (see e.g.\cite[Eq.(14.1)]{Theodore11}). Equipped with this metric, we can define the $L^2$ norm on $\Omega^k(M)$ as 
\begin{equation}
    \norm{\bv}_{L^2\Omega^k(M)}:=\left(\int_M \sqrt{\det(g_{ij})}(\bv,\bv)_{\Omega^k} dx\right)^{1/2}.
\end{equation}

Then, $L^2\Omega^k(M)$ is defined as the completion of $\Omega^1(M)$ with respect to this norm.
According to \cite[Lemma 3.3.2]{jost}, the corresponding $L^2$ inner product $\inner{\cdot,\cdot}$ on $\Omega^k(\M)$ has the form
  \begin{equation}
    \inner{\bu,\bv}=\int_\M \bu \wedge \ast \bv,\ \bu,\bv \in \Omega^k(\M).
  \end{equation}
It is well known (see e.g. \cite[Lemma 3.3.4]{jost}, \cite[Proposition 2.1.2]{Schwarz}) that the adjoint of the exterior derivative on $C_0^\infty\Omega^k(M)$ coincides with the codifferential operator $d^\ast$ defined in \eqref{eq:def_codifferential}.
  For $\bu\in \Omega^1(\M)$ with $\bu=u_i d x^i$, we have the expression $\Delta_H \bu=\widetilde{u}_i dx^i$ where
  \begin{equation}
    \widetilde{u}_i=g^{kl}\frac{\p u_i}{\p x^k \p x^l}-g^{kl}\Gamma_{kl}^j\frac{\p u_i}{\p x^j}-\frac{\p g^{kl}}{\p x^i}\left(\Gamma_{kl}^j u_j-\frac{\p u_k}{\p x^l}\right)-\frac{\p \Gamma_{kl}^j}{\p x^i}g^{kl}u_j.
  \end{equation}
We note that $\Delta_H$ coincides with the usual Laplacian when $M$ is a domain in $\RR^n$.

 Next we define the Sobolev spaces. Let $\nabla$ be the Levi-Civita connection on $M$ and $E_\alpha:=(E_{1,\alpha},\cdots, E_{n,\alpha})$ a local $g-$orthonormal frame on $U_\alpha$. The Sobolev norm is defined as 
  \begin{equation}
      \norm{\bu}_{H^m\Omega^k(M)}^2:=\sum_{0\leq j\leq m}\sum_{\abs{\beta}=j}\sum_{\alpha\in A}\int_M  \rho_\alpha\sqrt{\det(g_{ij})} (\nabla^\beta_{E_{\alpha}}\bu,\nabla^\beta_{E_{\alpha}}\bu)_{\Omega^k}dx,
  \end{equation}
  where $\nabla_{E_\alpha}^\beta=\nabla_{E_{1,\alpha}}^{\beta_1}\nabla_{E_{2,\alpha}}^{\beta_2}\cdots\nabla_{E_{n,\alpha}}^{\beta_n}$ for $\beta=(\beta_1,\cdots, \beta_n)$. And $H^m\Omega^k(M)$ is defined as the completion of $\Omega^k(M)$ with respect to this norm.
  
  We will need the following spaces
  \begin{align}
    W&:=\left\{\bu\in C^\infty_0\Omega^1(\M) \mid d^* \bu=0 \right\},\\
    V&:=\{\bv\in H^1_0\Omega^1(\M)\mid d^*\bv=0\},\\
    H&:=\text{the closure of }W\text{ in }L^2\Omega^1(\M).
  \end{align}
  Actually, by \cite[Chap.17, Lemma 6.1]{Taylor_PDE_3}, $V$ is the closure of $W$ in $H^1_0\Omega^1(\M)$.
   \begin{remark}
      For any open $\omega \subset \M$, there is $\bu\in V$ such that $\supp(\bu)\subset \omega$. Indeed, for any $\bv\in C_0^\infty\Omega^2(\omega)$, we have $d^\ast \bv\in V$ and $\supp(d^\ast \bv)\subset \omega$.
  \end{remark}
  Based on de-Rham theory, we have the following lemma.
  \begin{lemma}
    \label{lm:1}
    Let $\bf\in H^{-1}\Omega^1(\M)$, then $\bf=dp$ for some $p\in L^2(\M)$ if and only if 
    \begin{equation}
        \inner{\bf,\bv}_{H^{-1}\Omega^1(\M),H^{1}_0\Omega^1(\M)}=0,\ \forall \bv\in W.
    \end{equation}
    
  \end{lemma}
  \begin{proof}
    Recall that for $d^*: \Omega^1(\M)\to \Omega^0(\M)$, we have $d^*=-\ast d\ast$. Moreover, due to \cite[Lemma 3.3.1]{jost}, $\ast\ast:\Omega^{m}(\M)\to \Omega^m(\M)$ satisfying $\ast\ast=(-1)^{m(n-m)}$. Then $d^*\bw=\ast d\ast\bw=0$ if and only if $d\ast \bw=0$. Hence $\phi\in \Omega^{n-1}(\M)$ is a closed smooth form with compact support if and only if there exists a $\bw\in W$ such that $\phi=\ast \bw$. We define a current (see definition in \cite[p.34]{deRham})
    \begin{equation}
        \bf[\varphi]:=
        \begin{cases}
            \int_\M f\wedge \varphi,\ \varphi\in C_0^\infty \Omega^{n-1}(\M),\\
            0,\ \ \varphi\in C_0^\infty \Omega^{k}(\M),\ k\neq n-1.
        \end{cases}
    \end{equation}
    Therefore, 
    \begin{equation}
      \inner{\bf, \bw}=\int_\M \bf\wedge \ast \bw =0,\ \forall \bw\in W
    \end{equation}
    implies that for any closed $\phi\in C_0^\infty\Omega^{n-1}(\M)$, we have
    \begin{equation}
      \bf[\phi]=\int_\M \bf\wedge \phi=0.
    \end{equation}
    According to \cite[Theorem $17^\prime$]{deRham}, $\bf[\cdot]$ is homologous to zero, that is, in view of definition in \cite[p.79, p.45]{deRham} there exists a current $\widetilde{p}[\cdot]$ on $n$-form such that for any smooth $n-1$-form $\psi$ with compact support,
    \begin{equation}
      \bf[\psi]=\widetilde{p}[d\psi].
    \end{equation}
    Notice that $n$ forms are isomorphic to functions on $\M$ via the Hodge dual operator $\ast$. As a current is a continuous linear functional on the test function space $C_0^\infty(\M)$ (see e.g. \cite[p.40]{deRham}). Thus, we can find a distribution $p$, such that
    \begin{equation}
      \widetilde{p}[\varphi]=\begin{cases}
            \int_\M p\wedge \varphi,\ \varphi\in C_0^\infty \Omega^{n}(\M),\\
            0,\ \ \varphi\in C_0^\infty \Omega^{k}(\M),\ k\neq n.
        \end{cases}
    \end{equation}
    Thus for any  $\varphi\in C_0^\infty\Omega^1(M)$, there holds
    \begin{align}
      \inner{\bf,\varphi}&=\int_M \bf\wedge \ast\varphi=\int_\M p\wedge d \ast \varphi=(-1)^{n-1}\int_\M p\wedge \ast\ast d\ast \varphi\\
      &=(-1)^{n}\inner{p,d^* \varphi}.  
    \end{align}
    Since $\varphi$ is arbitrary, we can conclude that $\bf=(-1)^n dp$ in the sense of distribution.\par
    Next we show that $p\in L^2(\M)$ provided that $\bf\in H^{-1}\Omega^1(\M)$. Since $\M$ is compact, we can find a finite open cover $\{U_i\}_{i=1}^N$ such that each $U_i$ is diffeomorphic to $\RR^n$. According to J.L. Lions' Lemma (see e.g. \cite[Proposition 1.2]{temam}), there exists a function $p_i\in L^2(U_i)$ such that $dp_i=dp|_{U_i}=\bf|_{U_i}$ and $\norm{p_i}_{L^2(U_1)}\lesssim \norm{\bf}_{H^{-1}\Omega^1(M)}$. Then we have
    \begin{equation}
      \norm{p}_{L^2(U_i)}\lesssim \norm{\bf}_{H^{-1}\Omega^1(M)}+C_i,
    \end{equation}
    where $C_i$ is a constant. The summation of the above inequalities over $i=1,\cdots,N$ gives $p\in L^2(\M)$.
  \end{proof}
  \section{Stationary Stokes system}
  \label{sec:elliptic}
  In this section, we consider the stationary Stokes system with source $\bf\in \Omega^1(M)$ on a Riemannian manifold $\M$ : 
  \begin{equation}
    \label{eq:stokes}
    \begin{cases}
      \Delta_H \bu+dp=\bf,\ \textrm{in }M\\
      d^* \bu=0,\ \textrm{in } M,\\
      \bu|_{\p\M}=0.
    \end{cases}
  \end{equation}
  First we turn to the variational form of \eqref{eq:stokes} via the linear operator 
  \begin{align}
    L: L^2\Omega^1&(\M)\to V\\
    L(\bf)&:= \bu,
  \end{align}
  where $\bu\in V$ satisfies
  \begin{equation}
    \label{eq:eigen_variational}
    \inner{d \bu,d \bw}=\inner{\bf,\bw},\ \forall \bw\in V.
  \end{equation}
  Thanks to \cite[Lemma 2.4.10]{Schwarz}, for $\bv\in V$, there holds
  \begin{equation}
    \norm{\bv}_{H^1\Omega^1(\M)}\lesssim \norm{d\bv}_{L^2\Omega^2(\M)}.
  \end{equation}
  Therefore, the inner product
  \begin{equation}
    \bv,\bw\in V\times V\to \inner{d \bv,d \bw}
  \end{equation}
  is coercive. Then according to Lax-Milgram theorem, \eqref{eq:eigen_variational} admits a unique solution in $V$. Hence $L$ is well-defined.
  We can now show the existence and uniqueness for the solution of \eqref{eq:stokes}.
  \begin{lemma}
    \label{lm:stokes_well_pose}
    Suppose $\bf\in L^2\Omega^1(\M)$, then \eqref{eq:stokes} admits a unique solution $\bu\in V$ and corresponding $p\in L^2(\M)$. Furthermore, there holds
    \begin{equation}
      \label{eq:stokes_existence}
      \norm{\bu}_{H^1\Omega^1(\M)}\lesssim \norm{\bf}_{L^2\Omega^1(\M)}.
    \end{equation}
  \end{lemma}
  \begin{proof}
   Write $\bu=L(\bf)$. Since $d^*\bu\equiv0$, 
   $-\Delta_H\bu=d^*d\bu$, and we have
  \begin{equation}
    \inner{-\Delta_H\bu-\bf,\bw}_{V',V}=\inner{d^*d\bu-\bf,\bw}_{V',V}=0,\ \forall \bw\in V.
  \end{equation}
  According to Lemma \ref{lm:1}, there exists $p\in L^2(\M)$ such that $dp=\bf-\Delta_H\bu$.
  Furthermore, taking $\bw=\bu$ in \eqref{eq:eigen_variational}, we have
  \begin{align}
    \norm{\bu}_{H^1\Omega^1(\M)}^2&\lesssim \inner{d\bu,d\bu}=\inner{\bf,\bu}\leq \norm{\bf}_{L^2\Omega^1(\M)}\norm{\bu}_{L^2\Omega^1(\M)}\\
    &\leq \frac{1}{2\varepsilon}\norm{\bf}_{L^2\Omega^1(\M)}^2+\frac{\varepsilon}{2}\norm{\bu}_{L^2\Omega^1(\M)}^2.
  \end{align}
  Let $\varepsilon$ be small enough, then \eqref{eq:stokes_existence} follows immediately.
  \end{proof}
  We recall the following regularity result.
  \begin{lemma}
    \label{lm:stokes_regularity} (\cite{Taylor_PDE_3}, Appendix 17.A)
    Suppose $\bf\in L^2\Omega^1(\M)$, $\bu\in V$ and $p\in L^2(\M)$ solves \eqref{eq:stokes}, then $\bu\in H^2\Omega^1(\M)$ and $p\in H^1(\M)$.
  \end{lemma}

  Denote the kernel of $L$ by $N$, then $\bf\in N$ if and only if 
  \begin{equation}
    \inner{\bf, \bw}=0,\ \forall \bw\in V.
  \end{equation}
  Due to Lemma \ref{lm:1} if $\bf\in N$, there exists $p\in L^2(\M)$ such that $dp=\bf\in L^2\Omega^1(\M)$. On the other hand if $p\in H^1\Omega^1(\M)$, then $dp\in N$.
  Therefore,
  \begin{equation}
    N=\{dp\mid p\in H^1\Omega^1(\M)\}.
  \end{equation}

  For $\bf,\bh\in L^2\Omega^1(\M)$, we write $\bv^\bf=L(\bf)$ and $\bv^\bh=L(\bh)$. There holds
  \begin{equation}
    \inner{L\bf,\bh}=\inner{\bv^\bf,\bh}=\inner{d\bv^\bf,d\bv^\bh}=\inner{\bf,\bv^\bh}=\inner{\bf, L\bh},
  \end{equation}
  thus $L$ is self-adjoint. Lemma \ref{lm:stokes_well_pose} shows that $L$ is bounded. Furthermore, since $R(L)\subset V\subset H^1_0\Omega^1(\M)$ and $H^1_0\Omega^1(\M)$ is compactly embedded in $L^2\Omega^1(\M)$ \cite[Theorem 1.3.6]{Schwarz}, $L$ is a bounded compact self-adjoint operator on $L^2\Omega^1(\M)$. Let $\{\frac{1}{\lambda_j}\}_{j=1}^\infty$ be the decreasing sequence of eigenvalues of $L$.
  According to Riesz-Schauder theorem (see e.g. \cite[Theorem 9.10]{Hislop}) and $N\neq \emptyset$, $0$ is the only accumulation point of $\{\frac{1}{\lambda_j}\}_{j=1}^\infty$ and thus the sequence $\lambda_j\to \infty$ as $j\to\infty$. Denote the multiplicity of $\frac{1}{\lambda_j},\ j\geq 1$, by $m_j$. Let $\{\bphi_{jk}\},\ 1\leq k\leq m_j$, be the set of orthonormal eigenfunctions corresponding to $\lambda_j$, then we have $\bphi_{jk}\in V$ and
  \begin{equation}
    \inner{-\Delta_H\bphi_{jk},\bv}=\inner{d\bphi_{jk},d\bv}=\lambda_j\inner{\bphi_{jk},\bv},\ \forall \bv\in V.
  \end{equation}
  According to Lemma \ref{lm:1}, there exists $p_{jk}\in L^2(\M)$ satisfying
  \begin{equation}
    -\Delta_H \bphi_{jk}+dp_{jk}=\lambda_j \bphi_{jk}.
  \end{equation}
  By Lemma \ref{lm:stokes_regularity}, we have $p_{j,k}\in H^1(\M)$ and $\bphi_{jk}\in H^2\Omega^1(\M)$.\par
  We close this section by a linear independence result for eigenfunctions corresponding to the same eigenvalue.
  \begin{lemma}
    \label{lm:1_1}
    Let $\omega$ be a nonempty open subset of $\M$. For any fixed $j\geq 1$, the functions $\{\bphi_{jk}|_\omega\}_{1\leq k\leq m_j}$ are linearly independent.
  \end{lemma}
  \begin{proof}
    Suppose there exists $\{c_k\}_{1\leq k\leq m_j}\in \RR^{m_j}$ such that 
    \begin{equation}
      \sum_{i=k}^{m_j}c_i \bphi_{jk}=0 \text{ in }\omega.
    \end{equation}
    Write $\bm{\psi}=\sum_{i=k}^{m_j}c_i \bphi_{jk}$ and $q=\sum_{i=k}^{m_j}c_i p_{jk}$, then $\bm{\psi}\in V$ and it satisfies
    \begin{equation}
      d^* d \bm{\psi}+d q =-\Delta_H\bm{\psi}+dq= \lambda_j \bm{\psi}.
    \end{equation} 
  Recalling that 
  $d^* d^*=\ast d \ast\ast d\ast=(-1)^{n-1}\ast d d \ast=0$, we observe that
    \begin{equation}
      -\Delta_g q =  d^* d q = d^* d^* d \bm{\psi}+\lambda_j d^* \bm{\psi}=0.
    \end{equation}
    Moreover $\bm{\psi}|_\omega=0$ gives that $d q|_\omega=0$. According to the unique continuation for harmonic functions, there holds $d q=0$ in $\Omega$. Thus we have
    \begin{equation}
      \begin{cases}
        -d^* d \bm{\psi}=\lambda_j \bm{\psi},\\
        d^* \bm{\psi}=0,\ \bm{\psi}|_\omega=0.
      \end{cases}
    \end{equation}
    Leveraging the unique continuation (see e.g.\cite[Theorem 3.5.2]{Isakov}, \cite{EINT02}), we can obtain that $\bm{\psi}=0$. Hence $\{\bphi_{jk}|_\omega\}_{1\leq k\leq m_j}$ are linearly independent.
  \end{proof}

\section{Nonstationnary Navier-Stokes system and linearization}
  \label{sec:Nonlinear_NS}
  
  
For fixed $\alpha\in(0,1)$,
we set $\cC_\alpha(M):=C^{1,\frac{\alpha}{2}}([0,T]; C^{2,\alpha}\Omega^1(\M))$ to simplify notation and we consider this space with its usual norm $\norm{\cdot}_{\cC_\alpha(M)}$(see e.g. \cite{Sol1,Sol2}). 

Then we consider the initial boundary value problem associated with the geometric Navier-Stokes equation \eqref{eq:NS1} with a fixed $\bf\in \cC_\alpha(M)$.
 Since $\bu=0$ on $[0,T]\times\p\M$, we deduce that $R(\bu)=0$ on $[0,T]\times\p\M$. Moreover, using the fact that $d^*\bu\equiv0$, we deduce that $d^*\Delta_H\bu=-d^*d^*d\bu\equiv0$. Following \cite[pp.13]{Sol1}, we can  represent the pressure $p$ as a solution of the following boundary value problem
\begin{equation}
    \label{eq:p}
    \begin{cases}
      -\Delta_g p=d^*\bf-d^*R(\bu),&\ \textrm{in }[0,T]\times\M,\\
      \p_\nu p|_{(0,T)\times\p\M}=(\bf+\Delta_H\bu)\nu,&\ \textrm{on }[0,T]\times\partial\M,
    \end{cases}
  \end{equation}
  where $\nu$ denotes the outward unit normal vector to $\p\M$.
Assuming that $\bf\in \mathbb B_r$, with $r>0$ sufficiently small, we can prove the well-posedness of \eqref{eq:NS1} as follows.
  \begin{theorem}
    \label{thm:NS}
There exists $\delta>0$, depending on $(\M,g)$ and $T$, such that for all $\bf\in \mathbb B_\delta$, problem \eqref{eq:NS1} admits a unique solution $(\bu,p)=(\bu_\bf,p_\bf)$, $p$ being unique up to a map depending only on $t\in[0,T]$, with $u_\bf\in \mathcal C_\alpha$, $dp_\bf\in C^{\frac{\alpha}{2}}([0,T];C^{\alpha}\Omega^1(\M))$ and with $p_\bf$ a solution of \eqref{eq:p}. Moreover, the map $\mathbb B_\delta\ni\bf\mapsto (u_\bf,dp_\bf)\in \cC_\alpha(M)\times L^2([0,T]; L^2\Omega^1(\M))$ is $C^1$.
  \end{theorem}  \begin{proof}
 \textbf{ Existence of solutions}:
  Let us consider the following spaces
  $$\mathcal I:=\{v\in  \cC_\alpha(M):\ d^*v=0,\ v|_{\{0\}\times \M}=\partial_tv|_{\{0\}\times \M}=0,\ v|_{[0,T]\times\p\M}=0\},$$
  $$\mathcal J:=\{\bg\in \cC_\alpha(M):\ \bg|_{\{0\}\times \M}=0\}.$$
  We define the map $\mathcal G$  on $\mathcal J\times\mathcal I$ by
$$\mathcal G(\bf,\bu)=\p_t\bu-\Delta_H\bu+R(\bu) + d p_{\bf,\bu}-\bf,\quad (\bf,\bu)\in\mathcal J\times\mathcal I,$$
  where $p_{\bf,\bu}$ is the  solution of \eqref{eq:p} up to a map depending only on $t\in[0,T]$. 
  We can decompose $d p_{\bf,\bu}$ into two terms $d p_{\bf,\bu}=dv_{\bf,\bu}+dw_\bu$, where $v_{\bf,\bu}$ and $w_\bu$ solve the following boundary value problems
  \begin{equation}
    \label{eq:v1}
    \begin{cases}
      -\Delta_g v_{\bf,\bu}=d^*\bf-d^*R(\bu),&\ \textrm{in }[0,T]\times\M,\\
      \p_\nu v_{\bf,\bu}|_{(0,T)\times\p\M}=\bf\nu,&\ \textrm{on }[0,T]\times\partial\M,
    \end{cases}
  \end{equation}
  \begin{equation}
    \label{eq:v2}
    \begin{cases}
      -\Delta_g w_\bu=0,&\ \textrm{in }[0,T]\times\M,\\
      \p_\nu w_\bu|_{(0,T)\times\p\M}=\Delta_H \bu\nu,&\ \textrm{on }[0,T]\times\partial\M.
    \end{cases}
  \end{equation}
Let $\bu\in\mathcal I$ and $\bf\in\mathcal J$. Recalling that $\bf,\bu\in \cC_\alpha(M)$, we deduce that $d^*\bf-d^*R(\bu)\in C^{\frac{\alpha}{2}}([0,T]; C^\alpha(\M))$ and $\bf\nu\in C^{1,\frac{\alpha}{2}}([0,T];C^{2,\alpha}\partial\M)$. Then,  
by elliptic regularity (see e.g. \cite[Theorem 3.17]{Gio}, \cite[Lemma 2.4.2.2]{Gr}, \cite[Theorem 6.30]{GT} and the remarks at the beginning of page 128 of \cite{GT}), we obtain $v_{\bf,\bu}\in C^{\frac{\alpha}{2}}([0,T];C^1(\M))\cap C([0,T]; C^{2,\alpha}(\M))$ which implies that $dv_{\bf,\bu}\in C^{\frac{\alpha}{2}}([0,T];C^\alpha\Omega^1(\M))$. Let us consider $N$ a Green function for the Neumann problem on $\M$, and observe that $d w_\bu$ can be defined by
\begin{equation}\label{rep}\begin{aligned}
d w_\bu(x)&=d_x\left(\int_\M N(x,y)d^*\Delta_H\bu(y) dV_g(y)-\int_{\partial \M} N(x,y)\Delta_H\bu\nu(y) d\sigma_g(y)\right)\\
&=-d_x\left(\int_\M\left\langle d_y N(x,y),\Delta_H \bu(y) \right\rangle dV_g(y)\right).\end{aligned}\end{equation}
Therefore, combining the arguments of \cite[Proposition 2.4]{Sol1} and \cite[Proposition 2.6]{Sol1} with the fact that $\Delta_H\bu \in C^{\frac{\alpha}{2}}([0,T];C^{\alpha}\Omega^1(\M))$, we can show that $dw_\bu\in C^{\frac{\alpha}{2}}([0,T];C^{\alpha}\Omega^1(\M))$. Note that the results of  \cite[Proposition 2.4]{Sol1} and \cite[Proposition 2.6]{Sol1} have been stated in an open set of $\mathbb R^n$, but the proofs remain valid on a smooth Riemannian manifold by using classical properties of Green functions on manifolds  (see e.g. \cite[Theorem 4.17]{Th}).  Thus, the map $\mathcal G$ takes values in
$$\mathcal K:=\{\bg\in C^{\frac{\alpha}{2}}([0,T];C^{\alpha}\Omega^1(\M)):\ d^*\bg=0,\ \bg\nu|_{[0,T]\times\partial\M}=0\}$$ 
and one can easily check that $\mathcal G$ is $C^1$ from $\mathcal J\times\mathcal I$ to $\mathcal K$. Moreover, we have
$$\partial_\bu\mathcal G(0,0)\bv=\p_t\bv-\Delta_H\bv + d w_\bv,\quad \bv\in\mathcal I,$$
where $w_\bv$ is the solution, up to a map depending only on $t\in[0,T]$, of \eqref{eq:v2} for $\bu=\bv$. Applying again \cite[Proposition 2.6]{Sol1} we deduce that the map $\bv\mapsto dw_\bv$ is continuous from $\mathcal I$ to $C^{\frac{\alpha}{2}}([0,T];C^{\alpha}\Omega^1(\M))$. Combining this with \cite[Theorem 1.1]{Sol1} (see also \cite{Sol2}), we deduce that $\partial_\bu\mathcal G(0,0)$ is an isomorphism from $\mathcal I$ to $\mathcal K$. Observe again that the result of  \cite[Theorem 1.1]{Sol1} has been stated in an open set of $\mathbb R^n$ but, repeating the arguments for \cite[Proposition 2.4]{Sol1} and \cite[Proposition 2.6]{Sol1}, the proof can be extended to a smooth Riemannian manifold. Note also that, since we consider solutions with homogeneous Dirichlet boundary conditions, the assumptions involving the fundamental solution of the Laplacian in \cite[Theorem 1.1]{Sol1} will not be required here.

Now, the implicit function theorem   implies that there exists $\delta>0$  depending on  $(\M,g)$ and $T$  and a $\mathcal C^1$ map $\psi$ from $ B_{\delta}$ to $\mathcal I$, such that $\psi(0)=0$ and, for all $\bf\in  B_{\delta}$, we have 
$\mathcal G(\bf,\psi(\bf))=(0,0)$. For all $\bf\in \mathbb B_\delta$, fixing $u_\bf=\psi(\bf)$ and $p_\bf=p_{\bf,\bu_\bf}$ we obtain a solution of \eqref{eq:NS1}. Finally, using the fact that the map $\bu\mapsto dp_\bu$ is $C^1$ from $\mathcal I$ to $L^2([0,T];L^2\Omega^1(\M))$ we deduce that the map $\mathbb B_\delta\ni\bf\mapsto (u_\bf,dp_\bf)\in \cC_\alpha(M)\times L^2([0,T];L^2\Omega^1(\M))$ is $C^1$.\\

\textbf{Uniqueness of solutions} Fix $\bf\in\mathbb B_\delta$ and
let us consider $(\bu^j,p^j)$, $j=1,2$, solving \eqref{eq:NS1}  with $\bu^j\in \cC_\alpha(M)$ and  $p^j\in L^2([0,T];H^1(\M))$ a solution of \eqref{eq:p} with $u=\bu^j$. Fix $\bu=\bu^1-\bu^2$, $p=p^1-p^2$, and notice that $(\bu,p)$ satisfies the following conditions
\begin{equation}
    \label{eq:NS2}
    \begin{cases}
      \p_t\bu-\Delta_H\bu+\left(\nabla_{\bu^1_*}\bu_*+\nabla_{\bu_*}\bu_*^2\right)^*+ d p=0,\ \text{in } [0,T]\times \M,\\
      d^* \bu=0,\ \text{in } [0,T]\times \M,\\
      \bu|_{[0,T]\times\p\M}=0,\\
      \bu=0,\ p=0,\ \text{in }\{0\}\times\M.
    \end{cases}
  \end{equation}
Taking the scalar product with $\bu$ and integrating by parts, for all $t\in[0,T]$ we obtain
$$\begin{aligned}&\frac{1}{2}\partial_t\|\bu(t)\|^2+\|d\bu(t)\|^2+\left\langle\left(\nabla_{\bu^1_*}\bu_*+\nabla_{\bu_*}\bu_*^2\right)^*(t),\bu(t)\right\rangle\\
&=\left\langle \p_t\bu(t)-\Delta_H\bu(t)+\left(\nabla_{\bu^1_*}\bu_*+\nabla_{\bu_*}\bu_*^2\right)^*(t)+ d p(t),\bu(t)\right\rangle=0.\end{aligned}$$
Moreover, using the fact that 
  $\bu^j\in \cC_\alpha(M)$, $j=1,2$, we can find $C>0$ depending only on $\|\bu^j\|_{\cC_\alpha(M)}$, $j=1,2$ and $(\M,g)$ such that, for all $t\in[0,T]$, we have
$$\begin{aligned}\left|\left\langle\left(\nabla_{\bu^1_*}\bu_*+\nabla_{\bu_*}\bu_*^2\right)^*(t),\bu(t)\right\rangle\right|&\leq C(\|d\bu(t)\|\|\bu(t)\|+\|\bu(t)\|^2)\\
&\leq \frac 1 2\|d\bu(t)\|^2+(2C^2+C)\|\bu(t)\|^2.\end{aligned}$$
It follows that
$$\partial_t\|\bu(t)\|^2\leq (4C^2+2C)\|\bu(t)\|^2,\quad t\in[0,T],$$
  and, recalling that $\bu\in \cC_\alpha(M)$, the Gronwall's inequality implies that $\|\bu(t)\|=0$, $t\in[0,T]$, which implies that $u\equiv0$. It follows that $u^1=u^2$ and $dp\equiv0$. Using the fact that $\M$ is connected, we deduce that the map $p$ is independent of the variable $x\in\M$ and $p^1$ coincides with $p^2$ up to a function depending only on the variable $t\in[0,T]$.
  \end{proof}
  
  \section{Recovery of local spectral projections for the Stokes system}
  \label{sec:parabolic}
The goal of this section is to prove that the local source-to-final value map \eqref{eq:def_N} for the Navier-Stokes system determines uniquely the local spectral projections of the Stokes operator. This proof will be decomposed into two steps. First, we recover the local source-to-final value map of the Stokes system by a linearization method. Then, we recover the local spectral projections of the Stokes operator by applying analytical properties and representation of solutions of Stokes systems.

  In this section, we consider two compact smooth manifolds $(\M_1,g_1)$ and $(\M_2,g_2)$ and nonempty open subsets $\omega_i\subset \M_i$ for $i=1,2$.\par
 
  For $i=1,2$, let $(\lambda_{(i),j},\{\bphi_{(i),jk}\}_{k=1}^{m_{(i),j}})_{j=1}^\infty$  be the spectral data on $(\M_i,g_i)$. Then from now on we denote by $\{P_{(i),j}:\ j\in\mathbb N\}$ the \textit{local spectral projections} of the Stokes operator, with Dirichlet boundary conditions, on $(\M_i,g_i)$ restricted to the set $\omega_i$, i.e.
  \begin{align}
    P_{(i),j}:& L^2\Omega^1(\omega_i)\to L^2\Omega^1(\omega_i),\\
    P_{(i),j}(\bf)&=\sum_{k=1}^{m_{(i),j}}\inner{\bf,\bphi_{(i),jk}|_{\omega_i}}\bphi_{(i),jk}|_{\omega_i}. 
  \end{align}
  
  According to Lemma \ref{lm:1_1}, $P_{(i),j}\neq 0$ for all $j\geq 1$ and $i=1,2$.
  The goal of this section is to show that condition \eqref{NN}  implies that \begin{equation}\label{pr}(\lambda_{(1),j},P_{(1),j}\bPhi^\ast)_{j=1}^\infty=(\lambda_{(2),j},\bPhi^\ast P_{(2),j})_{j=1}^\infty.\end{equation}

\subsection{Recovery of the local source-to-final value  operator of the Stokes system}

 We will show that  the local source-to-final value  operator  for the linear Stokes system  can be determined from the local source-to-final value  operator of the Navier-Stokes system. We start by recalling the following well-posedness result for solutions of the following initial boundary value problem \begin{equation}
    \label{eq:problem_para}
    \begin{cases}
      (\p_t-\Delta_H)\bu + d p=\bf,\ \text{in } [0,T]\times \M_i,\\
      d^* \bu=0,\ \text{in } [0,T]\times \M_i,\\
      \bu|_{[0,T]\times\p\M_i}=0,\\
      \bu=0,\ p=0,\ \text{in }(-\infty,0)\times\M_i.
    \end{cases}
  \end{equation}

  \begin{lemma}
    \label{lm:fourier_1}
    Suppose $\bf\in C_0^\infty((0,T);C_0^\infty \Omega^1(\omega))$, then \eqref{eq:problem_para} has a unique solution $\bu\in C^1([0,T],H)\cap H^1([0,T],V)$. Moreover, $\bu$ can be written in the following Fourier expansion:
    \begin{equation}
      \bu = \sum_{j=1}^\infty \sum_{k=1}^{m_j}\alpha_{jk}(t)\bphi_{jk} \text{ in }L^2\Omega^1(\M),
    \end{equation}
    where $\alpha_{jk}(t)=\int_0^t\inner{\bphi_{jk},\bf(\tau)}e^{-\lambda_j (t-\tau)}d\tau.$
  \end{lemma}
  \begin{proof}
    Thanks to the classical theory of semigroups (see e.g. \cite[Chapter 3 Theorem 1.1]{Lions1}), there exists a unique solution $\bu\in C^1([0,T],H)\cap H^1([0,T],V)$ of \eqref{eq:problem_para}. According to \cite[Theorem]{Hislop}, $\{\bphi_{jk}\}_{j=1}^\infty$ forms a complete orthogonal basis in $N^\bot$. Write $\alpha_{jk}(t)=\inner{\bu(t),\bphi_{jk}}$, then the Fourier series
    \begin{equation}
      \sum_{j=1}^\infty \sum_{k=1}^{m_j}\alpha_{jk}(t)\bphi_{jk}
    \end{equation}  
    converges to $\bu(t)$ in $L^2\Omega^1(\M)$ for $0\leq t\leq T$. Notice that $\bphi_{jk}\in V$ and consequently $\inner{\bphi_{jk},dp}=\inner{d^*\bphi_{jk},p}=0$. Thus we have
    \begin{equation}
      \begin{cases}
        \inner{\p_t \bu,\bphi_{jk}}+\inner{d \bu, d \bphi_{jk}}=\inner{\bf, \bphi_{jk}},\\
        \inner{\bu(0),\bphi_{jk}}=0.
      \end{cases}
    \end{equation}
    Since $\inner{d\bphi_{jk},d\bw}=\lambda_j\inner{\bphi_{jk},\bw}$ for all $\bw\in V$, then $\alpha_{jk}(t)$ solves
    \begin{equation}
      \begin{cases}
        \frac{d}{dt}\alpha_{jk}(t)+\lambda_j\alpha_{jk}(t)=\inner{\bf, \bphi_{jk}},\\
        \alpha_{jk}(0)=0.
      \end{cases}
    \end{equation} 
    Hence we can obtain 
    \begin{equation}
      \alpha_{jk}(t)=\int_0^t \inner{\bf(\tau),\bphi_{jk}}e^{-\lambda_j(t-\tau)}\,d\tau.
    \end{equation}
  \end{proof}

In light of Lemma \ref{lm:fourier_1}, for $i=1,2$, we define  the local source-to-final value  operator $\P_i$ for the linear Stokes system by
  \begin{align}
  \label{eq:def_P}
    \P_i:C_0^\infty((0,T);&C_0^\infty\Omega^1(\omega_i))\to L^2\Omega^1(\omega_i),\\
    \P_i(\bf)&=\bu_i(T,\cdot)|_{\omega_i},
  \end{align}
where  $\bu_i$ solves the problem \eqref{eq:problem_para} with source $\bf$.

  In view of Theorem \ref{thm:NS}, for $i=1,2$, we can define 
 the local source-to-final value map  $\mathcal N_i$, given by \eqref{eq:def_N} with $(\M,g)=(\M_i,g_i)$ and $\omega=\omega_i$, and proceed to its linearization as follows.
\begin{proposition}\label{linearization} 
For $i=1,2$ and $\bf\in C^\infty_0([0,T];C^\infty_0\Omega^1(\omega_i))$, we fix $\delta_{i,\bf}=\frac{\delta}{1+\|\bf\|_{\mathcal C_\alpha(M_i)}}$ with $\delta$ introduced in Theorem \ref{thm:NS}. Then the map
$$(-\delta_{i,\bf},\delta_{i,\bf})\ni s\mapsto \mathcal N_i(s\bf)\in L^2\Omega^1(\omega_i)$$
is $C^1$ and we have
\begin{equation}\label{lin1}
    \partial_s\mathcal N_i(s\bf)|_{s=0}=\mathcal P_i(\bf).
\end{equation}

\end{proposition}
\begin{proof} Let $i=1,2$ and $\bf\in C^\infty_0([0,T];C^\infty_0\Omega^1(\omega_i))$. In view of Theorem \ref{thm:NS}, for all $s\in(-\delta_{i,\bf},\delta_{i,\bf})$, the problem 
\begin{equation}
    \label{eq:NS3}
    \begin{cases}
      \p_t\bu_s-\Delta_H\bu_s+R(\bu_s) + d p_s=s\bf,\ \text{in } [0,T]\times \M_i,\\
      d^* \bu_s=0,\ \text{in } [0,T]\times \M_i,\\
      \bu_s|_{[0,T]\times\p\M_i}=0,\\
      \bu_s=0,\ p_s=0,\ \text{in }\{0\}\times\M_i,
    \end{cases}
  \end{equation}
admits a unique solution $(\bu_s,p_s)$, up to a  function depending only on the time variable for the uniqueness of $p_s$,   with $\bu_s=\bu_{s\bf}=\psi(s\bf)$ and $p_s=p_{s\bf}$ solving
\begin{equation}
    \label{eq:p}
    \begin{cases}
      -\Delta_g p_s=sd^*\bf-d^*R(\bu_s),\ \text{in } [0,T]\times \M_i,\\
      \p_\nu p_s|_{(0,T)\times\p\M_i}=\Delta_H\bu_s\nu.
    \end{cases}
  \end{equation}
  In view of Theorem \ref{thm:NS}, it is clear that the map
  $$(-\delta_{i,\bf},\delta_{i,\bf})\ni s\mapsto (\bu_s,dp_s)\in \cC_\alpha(M_i)\times L^2([0,T];L^2\Omega^1(\M_i))$$
is $C^1$. Moreover, exploiting the fact that $\bu_s|_{s=0}\equiv0$, we obtain that $\partial_sR(\bu_s)|_{s=0}\equiv0$.
Therefore, fixing $\bv=\partial_s\bu_s|_{s=0}$ and $q=\partial_sp_s|_{s=0}$, we observe that $(\bv,q)$ solves the problems
\begin{equation}
    \label{eq:Nlin}
    \begin{cases}
      \p_t\bv-\Delta_H\bv + d q=\bf,\ \text{in } [0,T]\times \M_i,\\
      d^* \bv=0,\ \text{in } [0,T]\times \M_i,\\
      \bv|_{[0,T]\times\p\M_i}=0,\\
      \bv=0,\ q=0,\ \text{in }\{0\}\times\M_i,
    \end{cases}
  \end{equation}
\begin{equation}
    \label{eq:q}
    \begin{cases}
      -\Delta_g q=d^*\bf+d^*\Delta_H\bv,\\
      \p_\nu q|_{(0,T)\times\p\M_i}=\Delta_H\bv\nu.
    \end{cases}
  \end{equation}
Finally, recalling that $\mathcal N_i(s\bf)=\bu_s(T_0)|_{\omega_i}$, $s\in(-\delta_{i,\bf},\delta_{i,\bf})$, and $\mathcal P_i(\bf)=\bv(T_0)|_{\omega_i}$, we obtain the desired result.

\end{proof}

In view of Proposition \ref{linearization} and the fact that $\bPhi$ is a spatial transformation and is independent of the parameter $s$, we can derive from \eqref{NN} the following identity
    \begin{equation}\label{pp}
      \mathcal{P}_1\bPhi^\ast \bf = \p_s \cN_1(s\bPhi^\ast\bf)|_{s=0} =\bPhi^\ast\p_s\cN_2(s\bf)|_{s=0}=\bPhi^\ast \mathcal{P}_2\bf.
    \end{equation}

\subsection{Determination of the local spectral projections from the Stokes system}
 We consider the unique determination of the local spectral projections of the  Stokes operator from the local source-to-final value  operator  for the linear Stokes system. More precisely, we prove that \eqref{pp} implies the identity \eqref{pr}. This will be achieved by combining Fourier series representation with time  analiticity of solution of \eqref{eq:problem_para}. We consider first the following intermediate result.
  \begin{lemma}
    \label{lm:fourier_2}
    Suppose $\bf\in C_0^\infty((0,T);C_0^\infty \Omega^1(\omega))$ and $T_0\in (0,T]$, then 
    \begin{equation}
      \bu(T_0,\cdot)=\int_0^{T_0}\sum_{j=1}^\infty \sum_{k=1}^{m_j} \inner{\bf(\tau),\bphi_{jk}}e^{-\lambda_j (T_0-\tau)} \bphi_{jk}\, d\tau\text{ in }L^2\Omega^1(\M).
    \end{equation}
  \end{lemma}
  \begin{proof}
    By Lemma \ref{lm:fourier_1}, we have
    \begin{equation}
      \bu(T_0)=\sum_{j=1}^\infty\sum_{k=1}^{m_j}\bphi_{jk}\int_0^{T_0} \inner{\bf(\tau),\bphi_{jk}}e^{-\lambda_j(T_0-\tau)}\,d\tau.
    \end{equation}
    It only remains to show that one can commute the summations with the integral. Observe that for every $\tau\in (0,T_0]$ there holds
    \begin{align}
      &\bignorm{\sum_{j=1}^\infty\sum_{k=1}^{m_j}\inner{\bf(\tau),\bphi_{jk}}\bphi_{jk}e^{-\lambda_j(T_0-\tau)} }_{L^2\Omega^1(\M)}=\left(\sum_{j=1}^\infty\sum_{k=1}^{m_j}\abs{\inner{\bf(\tau),\bphi_{jk}}}^2e^{-2\lambda_j(T_0-\tau)}  \right)^{1/2}\\
      &\leq \left(\sum_{j=1}^\infty\sum_{k=1}^{m_j}\abs{\inner{\bf(\tau),\bphi_{jk}}}^2\right)^{1/2}\leq \norm{\bf(\tau)}_{L^2\Omega^1(\M)}.
    \end{align}
    The claim follows from the Lebesgue dominated convergence theorem for the Bochner integral (see e.g. \cite[Chapter IV Theorem 9.6]{Bochner}).
  \end{proof}

  Armed with Lemma \ref{lm:fourier_2}, we prove that \eqref{pp} implies the identity \eqref{pr} as follows.
  \begin{proposition}
    \label{prop:spectral_equivalence}
    Suppose that there exists a diffeomorphism $\bPhi: \omega_1\to\omega_2$ that satisfies 
  \begin{equation}
    \bPhi^\ast(\P_2(\bf))=\P_1(\bPhi^\ast \bf),\ \text{for all }\bf\in C_0^\infty((0,T);C_0^\infty\Omega^1.(\omega_2))
  \end{equation}
    Then for all $j\geq 1$, there holds $\lambda_{(2),j}=\lambda_{(1),j}$ and $\bPhi^\ast P_{(2),j}=P_{(1),j}\bPhi^\ast$.
  \end{proposition}
  \begin{proof}
    According to Lemma \ref{lm:fourier_2}, there holds
    \begin{align}
      \bu_i^\bf(T_0)|_{\omega_i}&=\int_0^{T_0}\sum_{j=1}^\infty \sum_{k=1}^{m_{(i),j}} \inner{\bf(\tau),\bphi_{(i),jk}}e^{-\lambda_{(i),j} (T_0-\tau)} \bphi_{(i),jk}|_{\omega_i}\, d\tau\\
      &=\int_0^{T_0}\sum_{j=1}^\infty e^{-\lambda_{(i),j}(T_0-\tau)}P_{(i),j}(\bf(\tau))\, d\tau.
    \end{align}
    Then $\bPhi^\ast(\P_2(\bf))=\P_1(\bPhi^\ast \bf)$ reads that
    \begin{align}
      &\bPhi^*\left(\int_0^{T_0}\sum_{j=1}^\infty e^{-\lambda_{(2),j}(T_0-\tau)}P_{(2),j}(\bf(\tau))d\tau\right)-\int_0^{T_0}\sum_{j=1}^\infty e^{-\lambda_{(1),j}(T_0-\tau)}P_{(1),j}\bPhi^*((\bf(\tau)))d\tau\\
      &=\int_0^{T_0}\sum_{j=1}^\infty e^{-\lambda_{(2),j}(T_0-\tau)}\bPhi^\ast(P_{(2),j}(\bf(\tau)))-e^{-\lambda_{(1),j}(T_0-\tau)}P_{(1),j}(\bPhi^\ast(\bf(\tau)))d\tau=0 .
    \end{align}
    For any $s\in (0,T_0)$ and small enough $\varepsilon_s>0$ such that $(s-\varepsilon_s,s+\varepsilon_s)\subset (0,T_0)$, we may take 
    $\bf(t,x)=\bf_1(t)\bf_2(x)$, where $\bf_1\in C^\infty_0(s-\varepsilon_s,s+\varepsilon_s)$ and $\bf_2\in C_0^\infty\Omega^1(\omega)$.
    Therefore, we have
    \begin{equation}
      \int_{s-\varepsilon_s}^{s+\varepsilon_s}\sum_{j=1}^\infty \left(e^{-\lambda_{(2),j}(T_0-\tau)}\bPhi^*(P_{(2),j}(\bf_2))-e^{-\lambda_{(1),j}(T_0-\tau)}P_{(1),j}(\bPhi^*(\bf_2))\right)\bf_1(\tau)\, d\tau=0,
    \end{equation}
    and thus 
    \begin{equation}
        \sum_{j=1}^\infty e^{-\lambda_{(2),j}(T_0-s)}\bPhi^\ast P_{(2),j}=\sum_{j=1}^\infty e^{-\lambda_{(1),j}(T_0-s)}P_{(1),j}\bPhi^\ast.
    \end{equation}
    As $s\in (0,T_0)$ is arbitrary, it follows that
    \begin{equation}
      \sum_{j=1}^\infty e^{-\lambda_{(2),j}(T_0-\tau)}\bPhi^\ast P_{(2),j}=\sum_{j=1}^\infty e^{-\lambda_{(1),j}(T_0-\tau)}P_{(1),j}\bPhi^\ast
    \end{equation}
    for all $0<\tau<T_0$.
    By Lemma \ref{lm:1_1}, $\bPhi^\ast P_{(2),j}\neq 0$ and $P_{(1),j}\bPhi^\ast\neq 0$ for $j\geq 1$. Write
    \begin{equation}
      S(\tau)=\sum_{j=1}^\infty e^{-\lambda_{(2),j}(T_0-\tau)}\bPhi^\ast P_{(2),j}-e^{-\lambda_{(1),j}(T_0-\tau)}P_{(1),j}\bPhi^\ast.
    \end{equation}
    Next we show that $S(\tau)$ is holomorphic on $E:=\{z\in \C\mid \Re(z)<T_0\}$. To this end, let $K\subset E$ be a compact set and $\delta>0$ such that $\Re(z)\leq T_0-\delta$ for all $z\in K$. Since $\lambda_{(i),j}\to \infty$ as $j\to \infty$ for both $i=1,2$, for any $\epsilon>0$, one can find a large enough integer $M$ such that $e^{-\delta\lambda_{i,M}}<\frac{\epsilon}{2}$ for both $i=1,2$. Then for $\tau\in K$ and $\bpsi\in L^2\Omega^1(\omega_2)$ it holds that
    \begin{align}
      &\bignorm{\sum_{j=M}^\infty e^{-\lambda_{(2),j}(T_0-\tau)}\bPhi^\ast(P_{(2),j}(\bpsi))-e^{-\lambda_{(1),j}(T_0-\tau)}P_{(1),j}(\bPhi^\ast(\bpsi))}_{L^2\Omega^1(\omega_1)}\\
      &\leq \sum_{j=M}^\infty \left(e^{-\lambda_{(2),j}(T_0-\Re(\tau))}\norm{\bPhi^\ast(P_{(2),j}(\bpsi))}_{L^2\Omega^1(\omega)}+e^{-\lambda_{(1),j}(T_0-\Re(\tau))}\norm{P_{(1),j}(\bPhi^\ast(\bpsi))}_{L^2\Omega^1(\omega_1)}\right)\\
      &\leq \sum_{j=M}^\infty \left(e^{-\delta\lambda_{(2),j}}\norm{\bPhi^\ast(P_{(2),j}(\bpsi))}_{L^2\Omega^1(\omega_1)}+e^{-\delta\lambda_{(1),j}}\norm{P_{(1),j}(\bPhi^\ast(\bpsi))}_{L^2\Omega^1(\omega_1)}\right)\\
      &\leq \frac{\epsilon}{2}\sum_{j=1}^\infty \left(\norm{\bPhi^\ast(P_{(2),j}(\bpsi))}_{L^2\Omega^1(\omega_1)}+\norm{P_{(1),j}(\bPhi^\ast(\bpsi))}_{L^2\Omega^1(\omega_1)}\right)\leq \epsilon\norm{\bpsi}_{L^2\Omega^1(\omega_1)}.
    \end{align}
    Hence $S(\tau)$ converges uniformly in operator norm for $\tau\in K$. Since $K$ is an arbitrary compact subset of $E$, it follows that $S(\tau)$ is holomorphic on $E$.\par
    Due to the identity principle for holomorphic functions, $S(\tau)=0$ on $(0,T_0)$ implies that $S(\tau)=0$ on $E$. Without loss of generality, we assume $\lambda_{(2),1}\leq \lambda_{(1),1}$. To get a contradiction, suppose $\lambda_{(2),1}<\lambda_{(1),1}$. Notice that
    \begin{align}
      &e^{\lambda_{(2),1}(T_0-\tau)}S(\tau)=\bPhi^\ast P_{(2),1}\\
      &+\sum_{j=1}^\infty e^{-(\lambda_{(2),j+1}-\lambda_{(2),1})(T_0-\tau)}\bPhi^\ast P_{(2),j+1}-e^{-(\lambda_{(1),j}-\lambda_{(2),1})(T_0-\tau)}P_{(1),j}\bPhi^\ast.
    \end{align}
    For any $\bpsi\in L^2\Omega^1(\omega)$ there holds
    \begin{align}
      &\bignorm{\sum_{j=1}^\infty e^{-(\lambda_{(2),j+1}-\lambda_{(2),1})(T_0-\tau)}\bPhi^\ast (P_{(2),j+1}(\bpsi))-e^{-(\lambda_{(1),j}-\lambda_{(2),1})(T_0-\tau)}P_{(1),j}(\bPhi^\ast(\bpsi))}_{L^2\Omega^1(\omega_1)}\\
      &\leq 2\max\{e^{-(\lambda_{(2),2}-\lambda_{(2),1})(T_0-\tau)},e^{-(\lambda_{(1),1}-\lambda_{(2),1})(T_0-\tau)}\}\norm{\bpsi}_{L^2\Omega^1(\omega_2)}.
    \end{align}
    Thus the series
    \begin{equation}
      \sum_{j=1}^\infty e^{-(\lambda_{(2),j+1}-\lambda_{(2),1})(T_0-\tau)}\bPhi^\ast P_{(2),j+1}-e^{-(\lambda_{(1),j}-\lambda_{(2),1})(T_0-\tau)}P_{(1),j}\bPhi^\ast
    \end{equation}
    converges to $0$ in operator sense when $\tau\to -\infty$. Therefore, we have
    \begin{align}
      P_{(2),1}=\lim_{\tau\to -\infty} e^{\lambda_{(2),1}(T_0-\tau)}S(\tau)=0,
    \end{align}
    which is to a contradiction. Hence $\lambda_{(1),1}=\lambda_{(2),1}$, and consequently, $\bPhi^\ast P_{(2),1}=P_{(1),1}\bPhi^\ast$. Furthermore, repeating this argument for $j>1$, we can obtain that $\lambda_{(2),j}=\lambda_{(1),j}$ and $\bPhi^\ast P_{(2),j}=P_{(1),j}\bPhi^\ast$.
  \end{proof}

  \section{Auxiliary hyperbolic Stokes problem}
  \label{sec:hyperbolic}
    In this section, we consider the following equation
    \begin{equation}
      \label{eq:thm_1}
      \begin{cases}
        (\p_t^2-\Delta_H)\bu + d p=\bf,\ \text{in } [0,T]\times \M,\\
        d^* \bu=0,\ \text{in } [0,T]\times \M,\\
        \bu|_{[0,T]\times\p\M}=0,\\
        \bu(0,\cdot)=\bu_0,\ \p_t\bu(0,\cdot)=\bu_1.
      \end{cases}
    \end{equation}
  Firstly, we show the existence of the solution for the weak formulation, that is
  \begin{equation}
    \label{eq:weak}
    \begin{cases}
      \inner{(\p^2_t -\Delta_H)\bu,\bv}=\inner{\bf,\bv},\ \forall \bv\in V,\ t\in (0,T)\\
      \bu(0)=\bu_0,\ \p_t \bu(0)=\bu_1.
    \end{cases}
  \end{equation}
  \begin{theorem}
    \label{thm:weak}
    Suppose $\bu_0\in V$ and $\bu_1\in H$, then there exists a unique $\bu\in C([0,T];V)\cap C^1([0,T];H)\cap H^2([0,T];V^\prime)$ solving \eqref{eq:weak}. 
  \end{theorem}
  \begin{proof}
    We use the Faedo-Galerkin method. Since $V$ is separable, there exists a countable basis $\{\bw_i\}_{i=1}^\infty$. Then we construct a sequence $\{\bu_m\}_{m=1}^\infty$ as
    \begin{equation}
      \bu_m=\sum_{i=1}^m a_{im}(t)\bw_i
    \end{equation}
    satisfying 
    \begin{equation}
      \label{eq:um}
      \inner{\p_t^2 \bu_m,\bw_j}+\inner{d \bu_m,d\bw_j}=\inner{\bf,\bw_j},\ 1\leq j\leq m.
    \end{equation}
    The coefficients $a_{im}$ can be found by solving the following ordinary differential equation deduced from \eqref{eq:um},
    \begin{equation}
      \label{eq:ode_a}
      \begin{cases}
        \sum_{i=1}^m\inner{\bw_i,\bw_j}\p_t^2 a_{im}+\sum_{i=1}^m \inner{d \bw_i,d\bw_j}a_{im}=\inner{\bf,\bw_j},\ 1\leq j\leq m,\\
        a_{im}(0)=u_{0,m}^i,\  a_{im}^\prime(0)=u_{1,m}^i,
      \end{cases}
    \end{equation}
    where $u_{0,m}^i$ and $u_{1,m}^i$ are chosen so that as $m\to\infty$, $\sum_{i=1}^m u_{0,m}^i \bw_i$ converges to $\bu_0$ in $V$ and $\sum_{i=1}^m u_{1,m}^i \bw_i$ converges to $\bu_1$ in $H$, respectively.
    
    We note that the existence and uniqueness of the solution for \eqref{eq:ode_a} is guaranteed by $(\inner{\bw_i,\bw_j})_{i,j=1}^m$ being an invertible matrix as $\{\bw_i\}_{i=1}^m$ are linearly independent.\par
    Next we define the energy 
    \begin{equation}
      E[\bv](t):=\frac{1}{2}\left(\inner{d\bv(t),d\bv(t)}+\norm{\p_t \bv(t)}^2+\norm{\bv(t)}^2\right).
    \end{equation}
    We prove the following a priori estimate for $\{\bu_m\}_{m=1}^\infty$
    \begin{equation}
      \label{eq:energy_estimate}
      E[\bu_m](t)\lesssim \norm{\bu_0}_V^2+\norm{\bu_1}^2+\int_0^t \norm{\bf(\tau)}^2\ d\tau.
    \end{equation}
    To this purpose, taking the time derivative of $E[\bu_m](t)$, we have
    \begin{equation}
      \frac{d E[\bu_m](t)}{d t}=\inner{\p^2_t \bu_m,\p_t \bu_m}+\inner{d \bu_m,\p_t d \bu_m}+\inner{\bu_m,\p_t \bu_m}.
    \end{equation}
    We multiply \eqref{eq:um} by $a_{jm}^\prime(t)$ and sum them together, we obtain
    \begin{equation}
      \inner{\p^2_t \bu_m,\p_t \bu_m}+\inner{d \bu_m,\p_t d \bu_m}=\inner{\bf,\p_t \bu_m}.
    \end{equation}
    Therefore,
    \begin{align}
      \frac{d E[\bu_m](t)}{d t}&=\inner{\bf,\p_t \bu_m}+\inner{\bu_m,\p_t \bu_m}\leq \frac{1}{2}\left(\norm{\bf}^2+2\norm{\p_t \bu_m}^2+\norm{\bu_m}^2 \right)\\
      &\leq \frac{1}{2}\norm{\bf}^2+E[\bu_m](t).
    \end{align}
    Integrating above inequality yields
    \begin{equation}
      E[\bu_m](t)\leq E[\bu_m](0)+\frac{1}{2}\int_0^t\norm{\bf(\tau)}^2+E[\bu_m](\tau)\ d\tau.
    \end{equation}
    Applying Gr\"onwall's inequality, we obtain \eqref{eq:energy_estimate}.\par
    In particular, $\{\bu_m\}_{m=1}^\infty$ remains in a bounded set of $L^\infty([0,T];V)$ and $\{\p_t \bu_m\}_{m=1}^\infty$ remains in a bounded set of $L^\infty([0,T];H)$.
    Consequently, according to Alaoglu's theorem (see, e.g. \cite[Theorem 3.16]{Brezis}), we can extract a weakly-$\ast$ convergent subsequence of $\{\bu_m\}_{m=1}^\infty$, denoted by $\{\bu_{m_1}\}_{m_1=1}^\infty$ which has a limit $\widetilde{\bu}\in L^\infty([0,T];V)\cap W^{1,\infty}([0,T];H)$. Next we show that $\widetilde{\bu}$ satisfies \eqref{eq:weak}.
    
    Consider a set of scalar functions
    \begin{equation}
      C_T^1:=\{\varphi\in C^1(0,T)\mid \varphi(T)=\varphi^\prime(T)=0\},
    \end{equation}
    and a function
    \begin{equation}
      \label{eq:form_psi}
      \psi=\sum_{j=1}^M \varphi_j \bw_j,\ \varphi_j\in C^1_T,
    \end{equation}
    where $M$ is a fixed positive number. Then for $m>M$, \eqref{eq:um} and integration by parts gives
    \begin{equation}
      \label{eq:thm1_1}
      \int_0^T -\inner{\p_t \bu_m(t), \psi^\prime(t)}+\inner{d \bu_m(t),d\psi(t)}dt=\int_0^T \inner{\bf(t),\psi(t)}dt+\inner{\p_t \bu_m(0),\psi(0)}.
    \end{equation}
    Notice that the functions with the form \eqref{eq:form_psi} are dense in the space
    \begin{equation}
      U:=\{\phi\mid\phi\in C^1([0,T];V),\ \phi(T)=\phi^\prime(T)=0\}.
    \end{equation}
    Since $\tbu$ is a limit of $\{\bu_m\}_{m=1}^\infty$ in weak$-\ast$ topology and $U\subset L^1([0,T];V)$, for any $\phi\in U$, by taking the limit $m_1\to\infty$ in \eqref{eq:thm1_1} and using the density, we have
    \begin{equation}
      \label{eq:thm1_2}
      \int_0^T -\inner{\p_t \tbu(t), \phi^\prime(t)}+\inner{d \tbu(t),d\phi(t)}dt=\int_0^T \inner{\bf(t),\phi(t)}dt+\inner{\bu_1,\phi(0)}.
    \end{equation}
    In particular, for any $\bv\in V$, $\varphi\in C^\infty_0(0,T)$, we observe that $\varphi(t)\bv(x)\in U$. Thus \eqref{eq:thm1_2} gives 
    \begin{equation}
      \int_0^T \left(\inner{\p^2_t \tbu,\bv}+\inner{d \tbu,d \bv}-\inner{\bf,\bv}\right) \varphi(t)\ dt=0.
    \end{equation}
    Since $\varphi\in C_0^\infty(0,T)$ is arbitrary, we have 
    \begin{equation}
      \inner{(\p^2_t-\Delta_H)\tbu,\bv}=\inner{\bf,\bv}
    \end{equation}
    for all $\bv\in V$ and $t\in(0,T)$. Thus $\tbu$ satisfies the first equation in \eqref{eq:weak}. Furthermore, we obtain that $\p_t^2 \tbu\in L^1([0,T];V^\prime)$. Therefore, we have $\tbu\in W^{2,1}([0,T],V^\prime)$ and thus $\p_t \tbu \in C([0,T],V^\prime)$ by \cite[Corollary 7.20]{Leoni17}.
    
    Let us now turn to initial conditions in \eqref{eq:weak}. In view of continuity of $\p_t\tbu$, for any $\phi\in U$, we can rewrite \eqref{eq:thm1_2} as
    \begin{equation}
    \label{eq:thm1_3}
        \int_0^T -\inner{\p_t \tbu(t), \phi^\prime(t)}+\inner{d \tbu(t),d\phi(t)}dt=\int_0^T \inner{\bf(t),\phi(t)}dt+\inner{\p_t \tbu(0),\phi(0)}.
    \end{equation}
    It follows that $\p_t \tbu(0)=\bu_1$ in $V^\prime$ as $\phi(0)\in V$ is arbitrary.
    
    Since $L^\infty([0,T];H)\subset L^2([0,T];H)$, weak-$\ast$ convergence in $L^\infty([0,T],H)$ implies weak convergence in $L^2([0,T];H)$.
    Therefore, $\bu_{m_1}$ converges to $\widetilde{\bu}$ weakly in $H^1((0,T);H)$. By the Sobolev embedding theorem, $\bu_{m_1}(0)$ converges to $\widetilde{\bu}(0)$ weakly in $H$. Recall that $\bu_{m_1}(0)$ converges to $\bu_0$ in $V$, then we have $\widetilde{\bu}(0)=\bu_0$.
    
    Notice that the energy estimate \eqref{eq:energy_estimate} holds for 
    $$\bu\in L^2([0,T];V)\cap H^1 ([0,T];H)\cap H^2([0,T];V^\prime)$$
    whenever $\bu$ solves \eqref{eq:weak}. Hence the uniqueness of the solution follows immediately.\par
    Finally, we show the continuity of the solution. We have obtained that $\tbu\in H^1([0,T];H)\subset C([0,T];H)$ and $\p_t \tbu\in H^1([0,T];V^\prime)\subset C([0,T];V^\prime)$. Moreover, we have $\tbu\in C([0,T];H)\cap L^\infty([0,T];V)$ and $\p_t \tbu\in C([0,T];V^\prime)\cap L^\infty([0,T];H)$. Notice that $V\subset H$ and $H\subset V^\prime$ are dense embeddings. Then according to \cite[Lemma 8.1]{Lions1}, there holds $\tbu\in C([0,T];V)$ and $\p_t \tbu\in C([0,T];H)$.
    \end{proof}
    \begin{theorem}
      \label{thm:regularity}
      Suppose $\bu_0\in V$, $d^\ast d\bu_0\in H$ and $\bu_1\in V$, then there exists a pair $(\bu,p)$ solving \eqref{eq:thm_1}, and $\bu\in C^1([0,T];V)\cap C^2([0,T];H)$, $p\in L^2([0,T];H^1(\M))$.
    \end{theorem}
    \begin{proof}
      According to Theorem \ref{thm:weak}, \eqref{eq:weak} admits a unique solution $\bu$. Consider $\Psi=\p_t \bu$, then $\Psi$ solves
      \begin{equation}
        \begin{cases}
          \inner{(\p^2_t -\Delta_H)\Psi,\bv}=\inner{\p_t \bf,\bv},\ \forall \bv\in V,\ t\in (0,T)\\
          \Psi(0)=\bu_1,\ \p_t \Psi(0)=\p_t^2 \bu(0)=\bf(0)+\Delta_H \bu(0).
        \end{cases}
      \end{equation}
      Since $\bf(0)=0$ and $\Delta_H \bu(0)=(d^*d+d d^*)\bu_0\in H$, we can obtain $\Psi\in C([0,T];V)\cap C^1([0,T];H)$ by applying Theorem \ref{thm:weak} to above equation. Hence we have $\bu\in C^1([0,T];V)\cap C^2([0,T];H)$ as $\Psi=\p_t \bu$. Observe that 
      \begin{equation}
        \inner{(\p^2_t -\Delta_H)\bu-\bf,\bv}=0,\ \forall \bv\in V
      \end{equation}
      and $(\p^2_t -\Delta_H)\bu-\bf\in H^{-1}(T^*\M)$. By Lemma \ref{lm:1}, we have
      \begin{equation}
        (\p^2_t -\Delta_H)\bu-\bf=d p
      \end{equation}
      for some $p\in L^2(\M)$. Applying the local regularity of Stokes problem \cite[Theorem IV.6.1]{boyer} on
      \begin{equation}
        -\Delta_H \bu+d p=\bf-\p_t^2 \bu,
      \end{equation}
      we have
      \begin{equation}
        \norm{\bu}_{H^2(T\M)}+\norm{p}_{H^1(\M)}\lesssim \left(\norm{\bf}_{L^2(T\M)}+\norm{\p_t^2\bu}_{L^2(T\M)}+\norm{\bu}_{H^1(T\M)} \right).
      \end{equation} 
    \end{proof}

  We define the operator 
  \begin{align}
  \label{eq:def_Lambda}
    \Lambda_\omega : C_0^\infty((0,T);C_0^\infty \Omega^1(\omega))&\to C^\infty((0,T);C^\infty \Omega^1(\omega)),\\
    \Lambda_\omega(\bf)&=\bu|_{\RR\times\omega},
  \end{align}
  where $\bu$ satisfies
  \begin{equation}
    \label{eq:problem_hypo}
    \begin{cases}
      (\p_t^2-\Delta_H)\bu + d p=\bf,\ \text{in } \RR^+\times \M,\\
      d^* \bu=0,\ \text{in } \RR^+\times \M,\\
      \bu|_{[0,T]\times\p\M}=0,\\
      \bu=0,\ p=0,\ \text{in }\RR^-\times\M.
    \end{cases}
  \end{equation}
  We show that $\Lambda_\omega$ can be determined by the spectral data $(\lambda_j,P_j)_{j=1}^\infty$.
  
  Write $\beta_{jk}=\inner{\bu,\bphi_{jk}}$. Since $\bu$ satisfies \eqref{eq:problem_hypo} and $\bphi_{jk}\in V$, we have
  \begin{equation}
    \begin{cases}
      \frac{d^2}{dt^2}\beta_{jk}(t)+\lambda_j\beta_{jk}(t)=\inner{\bf(t),\bphi_{jk}},\\
      \beta_{jk}(0)=0,\ \frac{d}{dt}\beta_{jk}(0)=0.
    \end{cases}
  \end{equation}
  Solving this equation, we have
  \begin{equation}
    \beta_{jk}(t)=\int_0^t \frac{\sin(\sqrt{\lambda_j}(t-\tau))}{\sqrt{\lambda_j}}\inner{\bf(\tau),\bphi_{jk}}d\tau.
  \end{equation}
  And $\bu$ has the expression
  \begin{align}
    \label{eq:expansion_wave}
    \bu(t)|_\omega&=\sum_{j=1}^\infty\sum_{k=1}^{m_j}\beta_{jk}(t)\bphi_{jk}|_\omega=\sum_{j=1}^\infty\sum_{k=1}^{m_j}\int_0^t \frac{\sin(\sqrt{\lambda_j}(t-\tau))}{\sqrt{\lambda_j}}\inner{\bf(\tau),\bphi_{jk}}\bphi_{jk}|_\omega d\tau\\
    &=\sum_{j=1}^\infty \int_0^t \frac{\sin(\sqrt{\lambda_j}(t-\tau))}{\sqrt{\lambda_j}}P_j(\bf(\tau))d\tau.
  \end{align} 
  \begin{proposition}\label{p1}
    Suppose that there exists a diffeomorphism $\bPhi: \omega_1\to\omega_2$ \textcolor{blue}satisfying \eqref{NN}. 
    Then we have 
    \begin{equation}\label{ll}
      \bPhi^\ast (\Lambda_{2,\omega_2}(\bf))=\Lambda_{1,\omega_1}(\bPhi^\ast(\bf)),\ \text{for all }\bf\in C_0^\infty((0,T);C_0^\infty\Omega^1(\omega_2)).
    \end{equation}
  \end{proposition}
  \begin{proof} Let us first observe that, in view of \eqref{pp} and
  Proposition \ref{prop:spectral_equivalence}, condition   \eqref{pr} is fulfilled.
    Thus, leverage the expression \eqref{eq:expansion_wave}, for all $t\in (0,T)$ and $\bf\in C_0^\infty((0,T);C_0^\infty\Omega^1(\omega_2))$, there holds
    \begin{align}
      \bPhi^\ast (\Lambda_{2,\omega_2}(\bf))(t,\cdot)&=\sum_{j=1}^\infty \int_0^t \frac{\sin(\sqrt{\lambda_j}(t-\tau))}{\sqrt{\lambda_j}}\bPhi^\ast (P_{(2),j}(\bf(\tau)))d\tau\\
      &=\sum_{j=1}^\infty \int_0^t \frac{\sin(\sqrt{\lambda_j}(t-\tau))}{\sqrt{\lambda_j}} P_{(2),j}(\bPhi^\ast(\bf(\tau)))d\tau\\
      &=\Lambda_{1,\omega_1}(\bPhi^\ast(\bf))(t,\cdot).
    \end{align}
  \end{proof}
  For two sets $U_1,U_2\subset M$, we define their distance as 
  \begin{equation}
      \dist(U_1,U_2):=\min_{x\in U_1,y\in U_2}{d_g(x,y)}.
  \end{equation}
  And we define the influence domain as sets 
  \begin{equation}
  \label{def:influence_domain}
    M(\omega,r):=\{x\in \M\mid \dist(\{x\},\omega)< r\},\ r> 0.
  \end{equation}
  We have the following unique continuation result.
  \begin{theorem}
    \label{thm:UC}
    Let $T>0$ and $U\subset \M$ be open. Suppose that $\bu\in H^2((0,2T);V)$ and $p\in L^2((0,2T);H^1(\M))$ satisfy $\p_t^2 \bu -\Delta_H \bu+dp =0$ and $\bu=0$ in $(0,2T)\times U$. Then $\bu(T,x)=0$ whenever $x\in M(U,T)^{int}$.
  \end{theorem}
  \begin{proof}
    Without loss of generality, we may assume that $U$ is connected. Notice that $d p = -\p^2_t \bu+ \Delta_H \bu=0$ in $(0,2T)\times U$, then $p(t,\cdot)|_U$ is a constant function whenever $t\in(0,2T)$. Moreover,  we have
    \begin{equation}
      -\Delta_g p=d^* d p=-d^*(\p_t^2 \bu - d^* d \bu)=0
    \end{equation}
    in $(0,2T)\times\M$. According to the unique continuation for harmonic functions, we conclude that $d p=0$ in $(0,2T)\times\M$. Thus $\bu$ satisfies $(\p^2_t-\Delta_H) \bu = 0$. According to the unique continuation for hyperbolic equations \cite{eller}, $\bu|_{(0,2T)\times U}=0$ implies that $\bu(T,x)=0$ whenever $x\in M(U,T)^{int}$.
  \end{proof}
  For any open not empty  subset $U\subset \M$, we define 
  \begin{equation}
    \label{eq:define_T}
    T(U):= \max_{x\in \p \M} d(x,U).
  \end{equation}
  
 Next we show that $\bu$ possess the property of finite speed of propagation. 

  \begin{theorem}
    \label{thm:fsop}
    Let $U\subset\M$ be an open connected subset such that $\M\setminus U$ is also connected.
    Let $\bf\in C^\infty_0((0,\infty);C^\infty_0\Omega^1(U))$. 
    Suppose $(\bu^\bf,p)$ solves \eqref{eq:problem_hypo}, then $\supp(\bu^\bf(s,\cdot))\subset M(U,s)$ for $0< s < T(U)$. 
  \end{theorem}
  \begin{proof}
    First, we observe that $(\p_t^2-\Delta_H)\bu^\bf+d p=\bf$ in $(0,+\infty)\times \M$. Consequently, we have $-\Delta_g p=d^*\bf$ and 
      $$(\p_t^2-\Delta_H)(\p_t^2\bu^\bf+d p)=(\p_t^2-\Delta_H)(\Delta_H\bu^\bf+\bf)=\Delta_H \bf+(\p_t^2-\Delta_H)\bf=\p_t^2\bf.$$
   In particular, recalling that $\supp(\bf)\subset (0,+\infty)\times U$, we obtain  
     \begin{equation}
    \label{coco1}\Delta_g p(t,x)=0,\quad (t,x)\in(0,+\infty)\times(\M\setminus U),
  \end{equation}
    \begin{equation}
    \label{coco2}(\p_t^2-\Delta_H)(\p_t^2\bu^\bf+d p)(t,x)=0,\quad (t,x)\in(0,+\infty)\times(\M\setminus U).  \end{equation}
      Let us define the sets 
    \begin{equation}
K:=\{(t,x)\in[0,T(U)]\times\M\mid x\in \M\setminus M(U,t)\},
    \end{equation}
    and, using the fact that  $s< T(U)$, 
    \begin{equation}
      \Upsilon(s):=\p\M\setminus M(U,s). 
    \end{equation}
    Since $s< T(U)$, $\Upsilon(s)$ is a nonempty open subset of $\p\M$ and $\{s\}\times \Upsilon(s)\subset K$.
    Due to the finite speed of propagation for wave equation, we have 
    \begin{equation}
      \bu^\bf(t,x)=d \left(\int_0^t\int_0^{\tau_1} p(\tau_2,x)\ d\tau_2 d\tau_1\right),\quad (t,x)\in K.
    \end{equation}
    To simplify notation, we denote $\int_0^t\int_0^{\tau_1} p(\tau_2,x)\ d\tau_2 d\tau_1$ by $q(t,x)$. Then, the boundary condition in \eqref{eq:problem_hypo}, gives
    \begin{equation}
      d q(s,\cdot)|_{\Upsilon(s)}=\bu^\bf(s,\cdot)|_{\Upsilon(s)}=0.
    \end{equation}
    Let $\widetilde{\Upsilon}(s)$ be a connected component of $\Upsilon(s)$, then for any two distinct points $y,z\in \widetilde{\Upsilon}(s)$, and a curve $\gamma(\xi):[0,1]\to \widetilde{\Upsilon}(s)$ that joins them, there holds
    \begin{equation}
      q(s,y)-q(s,z)=\int_{0}^1 dq(s,\gamma^\prime(\xi))\ d\xi=0.
    \end{equation}
    This implies that $q(s,\cdot)|_{\widetilde{\Upsilon}(s)}$ is a constant. Moreover, \eqref{coco1} implies that
    \begin{equation}
     \Delta_g q(s,x)=\int_0^t\int_0^{\tau_1} \Delta_gp(\tau_2,x)\ d\tau_2 d\tau_1=0,\quad x\in \M\backslash U.
    \end{equation}
    Fix $\lambda\in\mathbb R$ such that $q(s,\cdot)|_{\widetilde{\Upsilon}(s)}\equiv \lambda$ and notice that
$\p_\nu (q-\lambda)(s,x)=dq(\nu(x))(s,x)=0=q(s,x)-\lambda$, $x\in\widetilde{\Upsilon}(s)$. Thus, by Theorem \ref{thm:regularity}, $q(s,\cdot)\in H^1(\M)$ satisfies the condition
\begin{equation}
    \begin{cases}
      \Delta_g(q(s,\cdot)-\lambda)=0,\ \text{in }  \M\backslash U,\\
      q(s,\cdot)-\lambda|_{\widetilde{\Upsilon}(s)}=\partial_\nu (q(s,\cdot)-\lambda)|_{\widetilde{\Upsilon}(s)}=0,\\
    \end{cases}
  \end{equation}
and, recalling that $\M\backslash U$ is connected and $\widetilde{\Upsilon}(s)\subset\p(\M\backslash U)$, the unique continuation for harmonic functions implies that $q(s,\cdot)-\lambda=0$ in $\M\backslash U$. Thus $q(s,\cdot)$ is  constant in $\M \setminus U$ and we have $dq(s,\cdot)=0 $ in $\M\backslash U$. Combining this with \eqref{coco2}, we get $\bu^\bf(s,x)=d q(s,x)=0$ for $x\in \M\backslash M(U,s)\subset \M\backslash U$. Therefore, we get $\supp(\bu^\bf(s,\cdot))\subset M(U,s)$.
  \end{proof}
  We define the space
  \begin{equation}
    B(r;T,\omega):=\{\bu^{\bf}(T,\cdot)\mid \bf\in C^\infty_0((T-r,T);C^\infty_0\Omega^1(\omega))\},
  \end{equation}
  where $(\bu^\bf,p_\bf)$ solves \eqref{eq:problem_hypo} and
  \begin{equation}
    D(r;\omega):=\{1_{M(\omega,r)}\bv\mid \bv\in V\}.
  \end{equation}
  Then we have the approximate controllability.
  \begin{theorem}
    \label{thm:controllability}
    Let $\omega$ be an connected open subset of $\M$ such that $\M\setminus \omega$ is also connected, then $D(r;\omega)$ is contained into the closure of $B(r;T,\omega)$, in the sense of $L^2\Omega^1(M(\omega,r))$,  whenever $r< T(\omega)$.
  \end{theorem}
  \begin{proof}
    By Theorem \ref{thm:fsop}, when $r<T(\omega)$, for any $\bf\in C^\infty_0((T-r,T);C^\infty_0\Omega^1(\omega))$,
    we have $\supp\bu^\bf(T,\cdot)\subset M(\omega,r)$. Therefore, for any $\bv\in V$, there holds
    \begin{equation}
        \inner{\bu^\bf(T,\cdot),\bv}=\inner{1_{M(\omega,r)}\bu^\bf(T,\cdot),\bv}=\inner{\bu^\bf(T,\cdot),1_{M(\omega,r)}\bv}.
    \end{equation}
    Thus we only need to show that for any $\bv\in V$, if $\bv\in B(r;T,\omega)^\bot$, then $1_{M(\omega,r)}\bv=0$. To see this, we let $\bw$ be the solution of 
    \begin{equation}
      \begin{cases}
        (\p_t^2-\Delta_H)\bw + d p_w=0,\ \text{in } [0,T]\times \M,\\
      d^* \bw=0,\ \text{in } [0,T]\times \M,\\
      \bw|_{[0,T]\times\p\M}=0,\\
      \bw|_{t=T}=0,\ \p_t \bw|_{t=T}=\bv.
      \end{cases}
    \end{equation}
    According to Theorem \ref{thm:regularity}, we have $\bw\in C^1([0,T];V)\cap C^2([0,T];H)$ and $p_w\in L^2([0,T];H^1(\M))$.
    Then an integration by parts gives
      \begin{equation}\label{tatu}
      \begin{aligned}\int_0^T&\int_\omega \bw(t,\cdot) \wedge \ast \bf(t,\cdot) dt=\int_0^T\inner{\bw,\bf}dt
      \\
      &=\int_0^T \inner{\bw,\p_t^2 \bu^\bf -d^*d \bu^\bf +d p_\bf}-\inner{\bu^\bf,\p_t^2 \bw -d^*d \bw +d p_w }dt=0\end{aligned}
    \end{equation}
    for any $\bf\in C^\infty_0((T-r,T);C^\infty_0\Omega^1(\omega))$. This property implies that $\bw(t,x)=dp_1(t,x)=0$, $(t,x)\in [T-r,T]\times \omega$.

    Let $\widetilde{\bw}$ be the odd extension to $[0,2T]\times \M$. Then due to Theorem \ref{thm:UC}, we have $1_{M(\omega,r)}\bv=1_{M(\omega,r)}\p_t\widetilde{\bw}(T,\cdot) =0$. Hence $B(r;T,\omega)$ is dense in $D(r;\omega)$.
  \end{proof}
  Define the function
  \begin{equation}
    W_{\bf,\bh}:=\inner{\bu^\bf(t,\cdot),\bu^\bh(s,\cdot)}
  \end{equation}
  and the operator 
  \begin{align}
    \label{eq:def_J}
    J:L^2((0,2T)&\times \omega)\to L^2((0,T)\times\omega),\\
    J(\phi)(t,\cdot)&=\frac{1}{2}\int_t^{2T-t}\phi(t,\cdot)\, dt.
  \end{align}
  We have the following Blagovestchenskii's identity. 
  \begin{lemma}
    \label{lm:trick}
    Let $\bf,\bh\in C^\infty_0((0,T);C^\infty_0\Omega^1(\omega))$. Then
    \begin{equation}
      W_{\bf,\bh}(T,T)=\inner{\bf, J \Lambda_\omega \bh}_{L^2((0,T)\times\omega)}-\inner{\Lambda_\omega \bf, J \bh}_{L^2((0,T)\times\omega)}.
    \end{equation}
  \end{lemma}
  \begin{proof}
    Notice that
    \begin{align}
      \label{eq:lm:trick}
      (\p^2_t - \p^2_s)W(t,s)=\inner{\Delta_H \bu^\bf - d p_\bf+\bf, \bu^\bh}-\inner{\bu^\bf,\Delta_H \bu^\bh - d p_\bh+\bh}.
    \end{align}
    Since 
    \begin{align}
      \inner{\Delta_H \bu^\bf,\bu^\bh}-\inner{\bu^\bf,\Delta_H \bu^\bh}=\inner{d \bu^\bf, d\bu^\bh}-\inner{d \bu^\bf, d\bu^\bh}=0,
    \end{align}
    we have
    \begin{equation}
      \begin{cases}
        (\p_t^2-\p_s^2)W_{\bf,\bh}(t,s)=\inner{\bf, \bu^\bh}-\inner{\bu^\bf,\bh}=\inner{\bf,\Lambda_\omega \bh}-\inner{\Lambda_\omega \bf, \bh}\\
        W_{\bf,\bh}(0,s)=\p_t W_{\bf,\bh}(0,s)=0.
      \end{cases}
    \end{equation}
    The solution of one-dimensional wave equation can be written as
    \begin{align}
      W_{\bf,\bh}(T,T)&=\frac{1}{2}\int_0^T\int_\tau^{2T-\tau} \inner{\bf(\tau),(\Lambda_\omega) \bh(\sigma)}-\inner{(\Lambda_\omega \bf)(\tau), \bh(\sigma)}\,d\sigma d\tau\\
      &=\inner{\bf, J \Lambda_\omega \bh}_{L^2((0,T)\times\omega)}-\inner{\Lambda_\omega \bf, J \bh}_{L^2((0,T)\times\omega)}.
    \end{align}
  \end{proof}
\section{Proof of Theorem \ref{thm:main}}
\label{sec:mian_proof}

In this section we will complete the proof of Theorem \ref{thm:main}. For this purpose, we assume that condition \eqref{NN} is fulfilled and apply Proposition \ref{p1} to deduce that condition \eqref{ll} holds true. Therefore, it remains to prove that condition \eqref{ll} implies that $(\M_1,g_1)$ and $(\M_2,g_2)$ are isometric. We will prove this main step in the proof of Theorem \ref{thm:main} by extending the BC method to the auxiliary hyperbolic system \eqref{eq:problem_hypo}. 
  
  From now on and in all the remaining parts of this article, for $k=1,2$ and any nonempty subset $U_k$ of $\M_k$, we fix
  $$T_k(U_k):=\max_{x\in\p\M_k}(d(x,U_k)).$$
Let $U$ be a subset of $\M_1$ and $\bPhi: U\to \M_2$ be a map. We define 
  \begin{equation}
      T_{\bPhi}(U):=\min\{T_1(U), T_2(\bPhi(U))\}.
  \end{equation}
 Similarly, we denote the distance function on $(\M_k,g_k)$ by $d_{g_k}$, the diameter of sets in $\M_k$ by $\diam_k$, the distance between sets by $\dist_k$, and the influence domain of the set $U_k$ by $M_k(U_k,\cdot)$ for $k=1,2$, respectively.
  Furthermore, for $x\in\M_k$ and $\varepsilon>0$, we write the metric ball
  \begin{equation}
      B_k(x,\varepsilon):=\{y\in \M_k\mid d_{g_k}(x,y)<\varepsilon\},
    \end{equation}
  and we denote its complement $\M_k\setminus B_k(x,\varepsilon)$ by $B^c_k(x,\varepsilon)$ for $k=1,2$.
  Our first step is to recover the distance function in $\omega_1$. Observe that for any nonempty open subset $\widetilde{\omega}_1 \subset \omega_1$, there holds $\bPhi(\widetilde{\omega}_1)\subset \bPhi(\omega_2)$ and $T_{\bPhi}(\omega_1)\leq T_{\bPhi}(\widetilde{\omega}_1)$. Moreover, $\Lambda_{1,\omega_1}\bPhi^\ast=\bPhi^\ast\Lambda_{2,\omega_2}$ implies that $\Lambda_{1,\widetilde{\omega}_1}\bPhi^\ast=\bPhi^\ast\Lambda_{2,\bPhi(\widetilde{\omega}_1)}$.
  Therefore, without loss of generality, we may assume that $\omega_1$ is small enough so that it satisfies $\diam_1(\omega_1)< T_{\bPhi}(\omega_1)$.

   According to the above discussion  Theorem \ref{thm:main} follows from  the following theorem.
  \begin{theorem}
  \label{thm_main:hyperbolic}
      Suppose there is a diffeomorphism $\bPhi:\omega_1\to\omega_2$ such that 
      \begin{equation}
      \bPhi^\ast\Lambda_{2,\omega_2}=\Lambda_{1,\omega_1}\bPhi^\ast.
      \end{equation}
      Then $(\M_1,g_1)$ and $(\M_2,g_2)$ are isometric.
  \end{theorem}
   The remainder of this section is devoted to the proof of Theorem \ref{thm_main:hyperbolic} using the BC method. We extend the BC method to the auxiliary hyperbolic system \eqref{eq:problem_hypo}. This extension is not a direct application of existing results, due to the presence of the divergence-free constraint. Accordingly, each step of the BC method must be adapted. These adaptations are carried out below.
   
   \subsection{Reconstruction of metric in the accessible domain}
  \begin{proposition}
    \label{prop:dist_omega}
    Suppose there is a diffeomorphism $\bPhi:\omega_1\to\omega_2$ such that $\diam_1(\omega_1)< T_{\bPhi}(\omega_1)$ and $\bPhi^\ast\Lambda_{2,\omega_2}=\Lambda_{1,\omega_1}\bPhi^\ast$, then for any points $x,y\in \omega_1$, there holds $d_{g_1}(x,y)=d_{g_2}(\bPhi(x),\bPhi(y))$. Moreover, $\bPhi$ is an isometry.
  \end{proposition}
  \begin{proof}
    Let $\varepsilon>0$ be small enough that $B_1(x,\varepsilon), B_1(y,\varepsilon)\subset \omega_1$ and $B^c_1(x,\varepsilon)$, $B^c_1(y,\varepsilon)$ are connected. Then we write 
    \begin{equation}
      \mathcal{B}_k(p,\varepsilon):=C_0^\infty((0,\infty);C_0^\infty\Omega^1(B_k(p,\varepsilon))),\ k=1,2,
    \end{equation}
    and 
    \begin{equation}
      \bPhi_\ast (\cB_1(x,\varepsilon)):=C_0^\infty((0,\infty);C_0^\infty\Omega^1(\bPhi(B_1(x,\varepsilon)))).
    \end{equation}
     We define
    \begin{align}
      t_{1,\varepsilon} := \inf\{t>0\mid \exists \bf\in \mathcal{B}_1(x,\varepsilon)\text{ s.t. } \supp(\bu_1^\bf(t,\cdot))\cap B_1(y,\varepsilon)\neq \emptyset \}
    \end{align}
    \begin{align}
      t_{2,\varepsilon} := \inf\{t>0\mid \exists \bf\in \mathcal{B}_2(\bPhi(x),\varepsilon)\text{ s.t. } \supp(\bu_2^\bf(t,\cdot))\cap B_2(\bPhi(y),\varepsilon)\neq \emptyset \},
    \end{align}
    and 
    \begin{align}
      \widetilde{t_{2,\varepsilon}} := \inf\{t>0\mid \exists \bh\in \bPhi_\ast (\cB_1(x,\varepsilon))\text{ s.t. } \supp(\bu_2^\bh(t,\cdot))\cap \bPhi(B_1(y,\varepsilon))\neq \emptyset  \}.
    \end{align}
    As $\bPhi:\omega_1\to \omega_2$ is a diffeomorphism, we can find two positive functions $r(\varepsilon),R(\varepsilon)$ such that $\lim_{\varepsilon\to 0}r(\varepsilon)=\lim_{\varepsilon\to 0}R(\varepsilon)=0$, and
    \begin{equation}
      B_2(\bPhi(p),r(\varepsilon))\subset \bPhi(B_1(p,\varepsilon))\subset B_2(\bPhi(p),R(\varepsilon))
    \end{equation}
    for both points $p=x,y$. Hence, $t_{2,R(\varepsilon)}\leq \widetilde{t_{2,\varepsilon}}\leq t_{2,r(\varepsilon)}.$

    Since 
    \begin{equation}
      \bPhi^\ast:C_0^\infty((0,\infty);C_0^\infty\Omega^1(\bPhi(B_1(p,\varepsilon))))\to C_0^\infty((0,\infty);C_0^\infty\Omega^1(B_1(p,\varepsilon)))
    \end{equation}
    is a vector space isomorphism provided that $B_1(p,\varepsilon)\subset \omega_1$, for any $\bf\in \cB_{1}(x,\varepsilon)$, we can find $\bh\in \bPhi_\ast (\cB_1(x,\varepsilon))$ such that $\bf=\bPhi^\ast \bh$. Notice that 
    \begin{equation}
      \bPhi(\supp(\bPhi^\ast\bu_2^\bh(t,\cdot)))=\supp(\bu_2^\bh(t,\cdot)),
    \end{equation}
    thus 
    \begin{align}
      &\supp(\bu_1^\bf(t,\cdot))\cap B_1(y,\varepsilon) = \supp(\Lambda_{1,\omega_1}\bf(t,\cdot))\cap B_1(y,\varepsilon)\\
      & =\supp(\Lambda_{1,\omega_1}\bPhi^\ast \bh(t,\cdot))\cap B_1(y,\varepsilon) = \supp(\bPhi^\ast \Lambda_{2,\omega_2} \bh(t,\cdot))\cap B_1(y,\varepsilon)\\
      &=\supp(\bPhi^\ast \bu_2^\bh(t,\cdot))\cap B_1(y,\varepsilon)=\emptyset
    \end{align}
    is equivalent to
    \begin{equation}
      \supp(\bu_2^\bh(t,\cdot))\cap \bPhi(B_1(y,\varepsilon))=\emptyset.
    \end{equation}
    This indicates that $t_{1,\varepsilon}=\widetilde{t_{2,\varepsilon}}$.
    
    Notice that for $t=\diam_1(\omega_1)$, we have $M_1(B_1(x,\varepsilon),t)\cap B_1(y,\varepsilon)\neq \emptyset$. As $B_1(x,\varepsilon)\subset \omega_1$, we have $T_1(\omega_1)\leq T_1(B(x,\varepsilon))$. Then in the light of the approximate controllability (Theorem \ref{thm:controllability}), there exists $\bf\in \mathcal{B}_1(x,\varepsilon)$ such that $\supp(\bu_1^\bf(t,\cdot))\cap B_1(y,\varepsilon)\neq \emptyset$. Hence $t_{1,\varepsilon}\leq \diam_1(\omega_1)<T_{\bPhi}(\omega_1)$. For $t\in (t_{1,\varepsilon},T_{\bPhi}(\omega_1))$, according to the finite speed of propagation (Theorem \ref{thm:fsop}), there holds $\supp(\bu_1^\bf(t,\cdot))\subset M_1(B_1(x,\varepsilon),t)$. Therefore, $\dist_1(B_1(x,\varepsilon),B_1(y,\varepsilon))\leq t_{1,\varepsilon},$.\par
    Suppose that $\eta:=\dist_1(B_1(x,\varepsilon),B_1(y,\varepsilon))< t_{1,\varepsilon}\leq T_1(B(x,\varepsilon))$. Then there holds $M_1(B_1(x,\varepsilon),\eta)\cap B_1(y,\varepsilon)\neq \emptyset$. Due to the approximate controllability, this leads to a contradiction to the definition of $t_{1,\varepsilon}$. Thus we have 
    \begin{equation}
      t_{1,\varepsilon} = d_{g_1}(B_1(x,\varepsilon),B_1(y,\varepsilon)).
    \end{equation}
    Analogously, we can obtain that
    \begin{equation}
      t_{2,\varepsilon} = d_{g_2}(B_2(\bPhi(x),\varepsilon),B_2(\bPhi(y),\varepsilon)).
    \end{equation}
    Finally, we have
    \begin{equation}
      d_{g_2}(\bPhi(x),\bPhi(y))=
      \lim_{\varepsilon\to 0}t_{2,\varepsilon}=\lim_{\varepsilon\to 0}t_{2,R(\varepsilon)}\leq \lim_{\varepsilon\to 0} \widetilde{t_{2,\varepsilon}}=\lim_{\varepsilon\to 0}t_{1,\varepsilon}=d_{g_1}(x,y),
    \end{equation}
    and 
    \begin{equation}
      d_{g_2}(\bPhi(x),\bPhi(y))=
      \lim_{\varepsilon\to 0}t_{2,\varepsilon}=\lim_{\varepsilon\to 0}t_{2,r(\varepsilon)}\geq \lim_{\varepsilon\to 0} \widetilde{t_{2,\varepsilon}}=\lim_{\varepsilon\to 0}t_{1,\varepsilon}=d_{g_1}(x,y).
    \end{equation}
    Combining above two inequalities gives $d_{g_1}(x,y)=d_{g_2}(\bPhi(x),\bPhi(y))$.

    Finally, according to \cite[Theorem 5.6.15]{RG_Petersen}, $\bPhi:\omega_1\to\omega_2$ is an isometry.
  \end{proof}
  \subsection{Determination of a relation among geodesic balls}
  
  Let $\bPhi:\omega_1\to \omega_2$ be an isometry. 
  For $k=1,2$, We define the set $\mathfrak{B}_{\omega_k,\bPhi}\subset \omega_k\times \omega_k\times \omega_k\times \RR^3$ as
  \begin{align}
      \mathfrak{B}_{\omega_k,\bPhi}:=\{&(x,y,z,l_x,l_y,l_z)\mid 0<l_x<T_{\bPhi}(\{x\}),\ 0<l_y<T_{\bPhi}(\{y\}),\\
      &0<l_z<T_{\bPhi}(\{z\}),\ B_k(x,l_x)\subset B_k(y,l_y)\cup B_k(z,l_z)\}.
  \end{align}
  Then $\bPhi$ induces a push-forward $\bPhi_\ast:\mathfrak{B}_{\omega_1,\bPhi}\to \omega_2\times \omega_2\times \omega_2\times \RR^3$ as
  \begin{equation}
      \bPhi_\ast(x,y,z,l_x,l_y,l_z) = (\bPhi(x),\bPhi(y),\bPhi(z),l_x,l_y,l_z).
  \end{equation}
  Next we show the following proposition.
  \begin{proposition}
    \label{lm:6}
    Suppose there is an isometry $\bPhi:\omega_1\to\omega_2$ such that $\Lambda_{1,\omega_1}\bPhi^\ast=\bPhi^\ast\Lambda_{2,\omega_2}$.
    Then there holds
    \begin{equation}
        \bPhi_\ast \mathfrak{B}_{\omega_1,\bPhi}=\mathfrak{B}_{\omega_2,\bPhi}.
    \end{equation}
  \end{proposition}
  We will need the following lemma.
  \begin{lemma}
    \label{lm:norm_equal}
    Suppose there is an isometry $\bPhi:\omega_1\to\omega_2$ such that $\Lambda_{1,\omega_1}\bPhi^\ast=\bPhi^\ast\Lambda_{2,\omega_2}$. Let $\bf,\bh\in C_0^\infty((0,\infty);C_0^\infty\Omega^1(\omega_1))$ and $\widetilde{\bf},\widetilde{\bh}\in C_0^\infty((0,\infty);C_0^\infty\Omega^1(\omega_2))$ such that $\bf=\bPhi^\ast \widetilde{\bf}$ and $\bh=\bPhi^\ast \widetilde{\bh}$. There holds
    \begin{equation}
      \inner{\bu_1^{\bf}(T,\cdot),\bu_1^{\bh}(T,\cdot)}=\inner{\bu_2^{\widetilde{\bf}}(T,\cdot),\bu_2^{\widetilde{\bh}}(T,\cdot)}_{L^2\Omega^1(\M_2)},\ T>0.
    \end{equation}
  \end{lemma}
  \begin{proof}
    Notice that the pull-back $\bPhi^\ast$ and the operator $J$ defined in \ref{eq:def_J} commute, since they are transformations in space and time, respectively. 
    Leveraging the fact that $\bPhi$ is an isometry, Lemma \ref{lm:trick} gives that
    \begin{align}
      \inner{\bu_1^{\bf}(T,\cdot),\bu_1^{\bh}(T,\cdot)}&_{L^2\Omega^1(\M_1)}=\inner{\bu_1^{\bPhi^\ast \widetilde{\bf}}(T,\cdot),\bu_1^{\bPhi^\ast \widetilde{\bh}}(T,\cdot)}_{L^2\Omega^1(\M_1)}\\
      &=\inner{\bPhi^\ast \widetilde{\bf},J\Lambda_{1,\omega_1}\bPhi^\ast \widetilde{\bh}}_{L^2((0,T)\times\omega_1)}-\inner{\Lambda_{1,\omega_1}\bPhi^\ast \widetilde{\bf},J\bPhi^\ast \widetilde{\bh}}_{L^2((0,T)\times\omega_1)}\\
      &=\inner{\bPhi^\ast \widetilde{\bf},\bPhi^\ast J\Lambda_{2,\omega_2} \widetilde{\bh}}_{L^2((0,T)\times\omega_1)}-\inner{\bPhi^\ast\Lambda_{2,\omega_2} \widetilde{\bf},\bPhi^\ast J \widetilde{\bh}}_{L^2((0,T)\times\omega_1)}\\
      &=\inner{\widetilde{\bf}, J\Lambda_{2,\omega_2} \widetilde{\bh}}_{L^2((0,T)\times\omega_2)}-\inner{\Lambda_{2,\omega_2} \widetilde{\bf}, J \widetilde{\bh}}_{L^2((0,T)\times\omega_2)}\\
      &=\inner{\bu_2^{\widetilde{\bf}}(T,\cdot),\bu_2^{\widetilde{\bh}}(T,\cdot)}_{L^2\Omega^1(\M_2)}.
    \end{align}
  \end{proof}
  \begin{proof}[Proof of Proposition \ref{lm:6}] Firstly we show that $\bPhi_\ast \mathfrak{B}_{\omega_1,\bPhi}\subset\mathfrak{B}_{\omega_2,\bPhi}$. To this end, we only need to show that $B_1(x,l_x)\subset B_1(y,l_y)\cup B_1(z,l_z)$ implies $B_2(\bPhi(x),l_x)\subset B_2(\bPhi(y),l_y)\cup B_2(\bPhi(z),l_z)$. Let $\varepsilon>0$ be a small enough constant such that $B_1(x,\varepsilon),B_1(y,\varepsilon),B_1(z,\varepsilon)\subset \omega_1$, $B^c_1(p,\varepsilon)$ and $B^c_2(\bPhi(p),\varepsilon)$ are connected for points $p=x,y,z$. According to Proposition \ref{prop:dist_omega}, it follows that 
  \begin{equation}
      B_2(\bPhi(x),\varepsilon),B_2(\bPhi(y),\varepsilon),B_2(\bPhi(z),\varepsilon)\subset \omega_2.
  \end{equation}
  To simplify the notation, we fix $r>\varepsilon$, $T>r-\varepsilon$, and write 
  \begin{align}
    \S_{\varepsilon}^T(x,r):= C_0^\infty((T-(r-\varepsilon),T);C_0^\infty\Omega^1(B_1(x,\varepsilon))),
  \end{align}
  and 
  \begin{align}
    \bPhi_\ast(\S_{\varepsilon}^T(x,r)):= C_0^\infty((T-(r-\varepsilon),T);C_0^\infty\Omega^1(B_2(\bPhi(x),\varepsilon))).
  \end{align}
    Suppose $B_1(x,l_x)\subset B_1(y,l_y)\cup B_1(z,l_z)$.
    For any $\widetilde{\bf}\in \bPhi_\ast(\S_\varepsilon^T(x,l_x))$ and $\widetilde{\bh}\in \bPhi_\ast(\S_{\varepsilon}^T(y,l_y))\cup \bPhi_\ast(\S_{\varepsilon}^T(z,l_z))$, we have $\bf:=\bPhi^\ast\widetilde{\bf}\in \S_{\varepsilon}^T(x,l_x)$ and $\bh:=\bPhi^\ast \widetilde{\bh}\in \S_{\varepsilon}^T(y,l_y)\cup \S_{\varepsilon}^T(z,l_z)$.
    
    Notice that $T_{\bPhi}(\{p\})-\varepsilon\leq T_{1}(B_1(p,\varepsilon))$ for all $p\in \M_1$. Then due to the approximate controllability, we can find sequences $\{\boldsymbol{\xi}_i\}_{i=0}^\infty\subset \S_{\varepsilon}^T(y,l_y)$ and $\{\boldsymbol{\zeta}_i\}_{i=0}^\infty\subset \S_{\varepsilon}^T(z,l_z)$ such that $\bu_1^{\boldsymbol{\xi}_i}(T,\cdot)\to 1_{B_1(y,l_y)}\bu_1^\bf(T,\cdot)$ and $\bu_1^{\boldsymbol{\zeta}_i}(T,\cdot)\to (1_{B_1(x,l_x)}-1_{B_1(y,l_y)})\bu_1^\bf(T,\cdot)$ in $L^2\Omega^1(\M_1)$. Since $l_x<T_{\bPhi}(\{x\})$, there holds $\supp(\bu_1^\bf(T,\cdot))\subset B_1(x,l_x)$ by finite speed of propagation. Hence,
    \begin{equation}
      \inf \{\norm{\bu_1^{\bf}(T,\cdot)-\bu_1^\bh(T,\cdot)}_{L^2\Omega^1(\M_1)} \mid \bh\in \S_{\varepsilon}^T(y,l_y)\cup \S_{\varepsilon}^T(z,l_z)\}=0.
    \end{equation}
   By Lemma \ref{lm:norm_equal}, we obtain
    \begin{equation}
      \label{eq:inf_2}
      \inf \{\norm{\bu_2^{\widetilde{\bf}}(T,\cdot)-\bu_1^{\widetilde{\bh}}(T,\cdot)}_{L^2\Omega^1(\M_2)} \mid \widetilde{\bh}\in \bPhi_\ast(\S_{\varepsilon}^T(y,l_y))\cup \bPhi_\ast(\S_{\varepsilon}^T(z,l_z))\}=0.
    \end{equation}
    To get a contradiction, we assume that $B_2(\bPhi(x),l_x)\not\subset B_2(\bPhi(y),l_y)\cup B_2(\bPhi(z),l_z)$. Then the open set 
    \begin{equation}
      U=B_2(\bPhi(x),l_x)\backslash \overline{B_2(\bPhi(y),l_y)\cup B_2(\bPhi(z),l_z)}
    \end{equation}
    is non-empty. Due to approximate controllability, we can choose $\widetilde{\bf}\in \bPhi_\ast(S_\varepsilon(x,l_x))$ such that $\bu_2^{\widetilde{\bf}}(T,\cdot)|_U\neq 0$. By finite speed of propagation, there holds
    \begin{equation}
      \supp(\bu_2^{\widetilde{\bh}}(T,\cdot))\subset B_2(\bPhi(y),l_y)\cup B_2(\bPhi(z),l_z)
    \end{equation}
    for all $\widetilde{\bh}\in \bPhi_\ast(\S_{\varepsilon}^T(y,l_y))\cup \bPhi_\ast(\S_{\varepsilon}^T(z,l_z))$. This leads to a contradiction with \eqref{eq:inf_2} and thus we have $\bPhi_\ast \mathfrak{B}_{\omega_1,\bPhi}\subset\mathfrak{B}_{\omega_2,\bPhi}$.

    Notice that $\bPhi^{-1}:\omega_2\to \omega_1$ is an isometry satisfying $(\bPhi^{-1})^\ast\Lambda_{1,\omega_1}=\Lambda_{2,\omega_2}(\bPhi^{-1})^\ast$. Repeating above argument, we can obtain $(\bPhi^{-1})_\ast \mathfrak{B}_{\omega_2,\bPhi}\subset\mathfrak{B}_{\omega_1,\bPhi}$. Therefore,
    \begin{equation}
        \mathfrak{B}_{\omega_2,\bPhi}=\bPhi_\ast(\bPhi^{-1})_\ast \mathfrak{B}_{\omega_2,\bPhi}\subset \bPhi_\ast\mathfrak{B}_{\omega_1,\bPhi}\subset \mathfrak{B}_{\omega_2,\bPhi}.
    \end{equation}
    Then it follows immediately that $\bPhi_\ast \mathfrak{B}_{\omega_1,\bPhi}=\mathfrak{B}_{\omega_2,\bPhi}$.
    
  \end{proof}
  \subsection{Reconstruction of metric in the domain of influence}
  Next we recover the metric within the domain of influence for $\omega_1$ and $\omega_2$.
  For $k=1,2$, and a point $p\in \M^{\text{int}}_k$, we define
  \begin{equation}
      S_p\M_k:=\{\xi\in T_p\M_k\mid \norm{\xi}_{g_k}=1\},
  \end{equation}
  and
  \begin{equation}
    \sigma_k(p):=\min_{\xi\in S_p\M_k}\max\{s\in (0,\tau_k(p,\xi)]\mid d_{g_k}(\gamma_k(s;p,\xi),p)=s\},
  \end{equation}
  where
  \begin{equation}
    \tau_k(p,\xi):=\inf\{s>0\mid \gamma_k(s;p,\xi)\in \p\M_k\},
  \end{equation}
  and $\gamma_k(\cdot;p,\xi)$ is the geodesic on $\M_k$ satisfying $\gamma_k(0;p,\xi)=p$ and $\gamma_k^\prime(0;p,\xi)=\xi$. Note that $\gamma_k(-t;p,\xi)=\gamma_k(t;p,-\xi)$ whenever $0<t<\min\{\tau_k(p,\xi),\tau_k(p,-\xi)\}$.
  Furthermore, for $x\in \M_1^\inter$ and $\bPhi(x)\in \M_2^\inter$, we define
  \begin{equation}
    \sigma_1^T(x)=\min\{\sigma_1(x),T_{\bPhi}(\{x\})\},\ \sigma_2^T(\bPhi(x))=\min\{\sigma_2(\bPhi(x)),T_{\bPhi}(\{x\})\}.
  \end{equation}
  The bound $T_{\bPhi}(\{x\})$ is needed for Proposition \ref{lm:6} to be valid for both $(\M_1,g_1)$ and $(\M_2,g_2)$.
   For $x\in\M_k$, $k=1,2$, and $\rho\in (0,\sigma_k^T(x))$, we define
  \begin{equation}
    M^T_k(x,\rho):=\{\gamma_k(t;x,\xi)\mid \xi\in S_x\M_k, t\in (0,\sigma_k^T(x)-\rho)\}.
  \end{equation}
    And the maps
  \begin{align}
  \label{eq:def:R_1}
    R_1&: \overline{M_1^T(x,3\rho)}\to C(\overline{B_1(x,\rho)}),\\
    R_1[y](z)&:=d_{g_1}(y,z),\ y\in \overline{M_1^T(x,3\rho)},\ z\in \overline{B_1(x,\rho)}.
  \end{align} 
  and 
  \begin{align}
  \label{eq:def:R_2}
    R_2&: \overline{M_2^T(\bPhi(x),3\rho)}\to C(\overline{B_1(x,\rho)}),\\
    R_2[y](z)&:=d_{g_2}(y,\bPhi(z)),\ y\in \overline{M_2^T(\bPhi(x),3\rho)},z\in \overline{B_1(x,\rho)}.
  \end{align}
   We aim to show that $M^T_1(x,3\rho)$ is isomertric to $M^T_2(\bPhi(x),3\rho)$. More precisely, we will prove the following proposition.
   \begin{proposition}
    \label{prop:isometry}
    Let $x\in \M_1^\inter$, and $\rho>0$ such that there is an isometry $\bPhi:B_1(x,\rho)\to B_2(\bPhi(x),\rho)$ such that
    \begin{equation}
        \bPhi_\ast \mathfrak{B}_{B_1(x,\rho),\bPhi}=\mathfrak{B}_{B_2(\bPhi(x),\rho),\bPhi}
    \end{equation}
     Then
    \begin{equation}
      R_2^{-1}\circ R_1: \overline{M_1^T(x,3\rho)}\to \overline{M_2^T(\bPhi(x),3\rho)}
    \end{equation}
    is an isometry.
  \end{proposition}
  We recall the following lemma borrowed from \cite{Oksanen2014}, see also \cite[p.70]{doCarmo} for the definition of normal neighborhood.
  \begin{lemma}
    \label{lm:7}(Lemma 15, \cite{Oksanen2014})
    For $k=1,2$ and $p\in \M^\inter_k$, let $\rho>0$ be small enough so that $B_k(p,\rho)$ is contained in a normal neighborhood  of $p$ in $\M^\inter_k$. For $r>0$, the following properties are equivalent:
    \begin{enumerate}
      \item $r+\rho\leq \sigma_k(p)$
      \item for all $t\in (0,r]$, $s\in (0,t)$, and $y\in \p B_k(p,\rho)$, $B_k(y,t)\not\subset B_k(p,s+\rho)$.
    \end{enumerate}
  \end{lemma}
  Applying the above mentioned lemma, we obtain the following lemma.
  
  \begin{lemma}
    \label{lm:cut_func}
   Suppose there is an isometry $\bPhi:\omega_1\to\omega_2$ such that 
   \begin{equation}
       \bPhi_\ast \mathfrak{B}_{\omega_1,\bPhi}=\mathfrak{B}_{\omega_2,\bPhi}.
   \end{equation}
    Then for any $x\in \omega_1$, we have $\sigma_1^T(x)=\sigma_2^T(\bPhi(x))$.
  \end{lemma}
  \begin{proof}
    Without loss of generality, we assume that $T_1(\{x\})\leq T_2(\{\bPhi(x)\})$. Let $\rho$ be small enough so that $\overline{B_1(x,\rho)}\subset \omega_1$, and $B_1(x,\rho)$ and $B_2(\bPhi(x),\rho)$ are contained in a normal neighborhood of $x$ and $\bPhi(x)$, respectively. By definition we have $\sigma_1(x)\leq T_1(\{x\})=T_{\bPhi}(\{x\})$. For any $\varepsilon\in(0,\sigma_1(x)-\rho)$, we write $r_\varepsilon=\sigma_1(x)-\varepsilon-\rho$. By Lemma \ref{lm:7}, for all $t\in (0,r_\varepsilon]$, $s\in (0,t)$ and $y\in \p B_1(x,\rho)$, it holds that $B_1(y,t)\not\subset B_1(x,s+\rho)$. Notice that 
    \begin{equation}
      T_1(y)\geq T_1(\{x\})-\rho,\ T_2(\bPhi(y))\geq T_2(\{\bPhi(x)\})-\rho\geq T_1(\{x\})-\rho,
    \end{equation}
    then we have
    \begin{equation}
      T_{\bPhi}(y)\geq T_1(\{x\})-\rho\geq \sigma_1(x)-\rho > r_\varepsilon.
    \end{equation}
    According to Proposition \ref{lm:6}, there holds
    \begin{equation}
      B_2(\bPhi(y),t)\not\subset B_2(\bPhi(x),s+\rho),\ \text{for all } t\in (0,r_\varepsilon],\ s\in (0,t) \text{ and }y\in \p B_1(x,\rho).
    \end{equation}
    Again, by Lemma \ref{lm:7}, we have
    \begin{equation}
      r_\varepsilon+\rho=\sigma_1(x)-\varepsilon\leq \sigma_2(\bPhi(x)).
    \end{equation}
    Taking $\varepsilon\to 0$, we can conclude that $\sigma_1(x)\leq \sigma_2(\bPhi(x))$. If $\sigma_1(x)=T_{\bPhi}(\{x\})$, we immediately have
    \begin{equation}
      \sigma_2^T(\bPhi(x))=T_{\bPhi}(\{x\})=\sigma_1^T(x).
    \end{equation}
    Suppose $\sigma_1(x)<T_{\bPhi}(\{x\})$. Due to Lemma \ref{lm:7}, for $r\in (\sigma_1(x)-\rho,T_{\bPhi}(\{x\})-\rho)$, there exists $t\in (0,r]$, $s\in (0,t)$ and a point $y\in \p B_1(x,\rho)$ such that $B_1(y,t)\subset B_1(x,s+\rho)$. Using Proposition \ref{lm:6} again, we have $B_2(\bPhi(y),t)\subset B_2(\bPhi(x),s+\rho)$ and thus $r+\rho>\sigma_2(\bPhi(x))$. Thus $\sigma_1(x)\geq \sigma_2(\bPhi(x))$. Therefore, $\sigma_1(x)= \sigma_2(\bPhi(x))$ and consequently $\sigma_1^T(x)= \sigma_2^T(\bPhi(x))$.
  \end{proof}
  \begin{lemma}
    \label{lm:dist_influence}
    Let $x\in \M^\inter_1$, $\xi\in S_x\M_1$ and $\rho>0$. Suppose there is an isometry $\bPhi: B_1(x,\rho)\to B_2(\bPhi(x),\rho)$ such that
    \begin{equation}
        \bPhi_\ast \mathfrak{B}_{B_1(x,\rho),\bPhi}=\mathfrak{B}_{B_2(\bPhi(x),\rho),\bPhi}.
    \end{equation}
    Let $0<t\leq\sigma^T_k(x)-3\rho$ for both $k=1,2$, then there holds $d_{g_1}(z,\gamma_1(t;x,\xi))=d_{g_2}(\bPhi(z),\gamma_2(t;\bPhi(x),\bPhi_\ast\xi))$ for any $z\in B_1(x,\rho)$.
  \end{lemma}
  \begin{proof}
    By the definition of $M^T_1(x,3\rho)$, for any $p\in M^T_1(x,3\rho)$, there is a unit vector $\xi\in S_x\M_1$ and $l>0$ such that $p = \gamma_1(l;x,\xi)$ and $d_{g_1}(p,x)=l$. Fix $s\in (0,\rho)$, then we write $y=\gamma_1(s;x,\xi)$ and $r=l-s$. For any $z\in B_1(x,\rho)$, let $0<R$,  and suppose that there is a small $\varepsilon>0$ such that 
    \begin{equation}
      \label{eq:prop_2}
      B_1(y,r+\varepsilon)\subset B_1(x,l)\cup B_1(z,R).
    \end{equation}
    Since $d_{g_1}(x,p)=l$ and $d_{g_1}(y,p)=r$, we have $p\in B_1(z,R)$ and thus $d_{g_1}(z,p)<R$. Define
    \begin{equation}
      R^* = \inf_{R>0}\{\text{there exists }\varepsilon>0 \text{ such that }\text{\eqref{eq:prop_2} holds}\}.
    \end{equation}
    Then we have $d_{g_1}(z,p)\leq R^*$. To get a contradiction, assume that $d_{g_1}(z,p)\neq R^*$, then there exists $R\in (d_{g_1}(z,p),R^*)$ such that for any $\varepsilon>0$, \eqref{eq:prop_2} does not hold. That is
    \begin{equation}
      B_1(y,r+1/j)\backslash \left(B_1(x,l)\cup B_1(z,R)\right)\neq \emptyset
    \end{equation}
    for every positive integer $j$. Then we can select a sequence $\{p_j\}_{j=1}^\infty$ such that 
    \begin{equation}
      p_j\in B_1(y,r+1/j)\backslash \left(B_1(x,l)\cup B_1(z,R)\right).
    \end{equation}
    Due to the sequential compactness of the ball $\overline{B_1(y,r+1)}$, there is a subsequence of $\{p_j\}_{j=1}^\infty$ that converges to a point $\widetilde{p}\in \overline{B_1(y,r)}$. Notice that $B_1(y,r)\subset B_1(x,l)$ by the triangular inequality, hence we have $\widetilde{p}\in \p B_1(y,r)$.
    Consequently, there holds
    \begin{equation}
      d_{g_1}(x,\widetilde{p})\leq d_{g_1}(x,y)+d_{g_1}(y,\widetilde{p})\leq s+r=l.
    \end{equation}
    Notice that $d_{g_1}(x,\widetilde{p})<l$ contradicts with $d_{g_1}(x,p_j)>l$ for all $j>0$ and $p_j\to \widetilde{p}$. Therefore, $d_{g_1}(x,\widetilde{p})=l$.
    
    Next we show that $p=\widetilde{p}$ via a contradiction. Assume that there is
    a minimizing path from $x$ to $\widetilde{p}$ such that it is $\gamma_1(\cdot;x,\zeta)$ near $x$ where 
    $\zeta\in S_x\M$ and $\zeta\neq \xi$. Choose small enough $0<\delta<s$ such that $B_1(x,\delta)$ is geodesically convex, see for instance \cite[Theorem 6.17]{Lee18}, and there is no cut locus of $\widetilde{p}$ in $B_1(x,\delta)$. Then we have
    \begin{equation}
    \label{eq:lm_6.6_1}
        d_{g_1}(\gamma_1(-\delta;x,\zeta),\widetilde{p})=l+\delta.
    \end{equation}
    Moreover, since $d_{g_1}(y,\widetilde{p})=r$, there holds
    \begin{equation}
      d_{g_1}(\gamma_1(\delta;x,\xi),\widetilde{p})\leq d_{g_1}(\gamma_1(\delta;x,\xi),y) + d_{g_1}(y,\widetilde{p})=r+s-\delta=l-\delta,
    \end{equation}
    and 
    \begin{equation}
      d_{g_1}(\gamma_1(\delta;x,\xi),\widetilde{p})\geq d_{g_1}(x,\widetilde{p}) - d_{g_1}(\gamma_1(\delta;x,\xi),x)=l-\delta.
    \end{equation}
    Hence $d_{g_1}(\gamma_1(\delta;x,\xi),\widetilde{p})=l-\delta$. Notice that the concatenated curve
    $$\gamma_1([-\delta,0];x,\zeta)\cup \gamma_1([0,\delta];x,\xi)$$
    joins $\gamma_1(-\delta;x,\zeta)$ and $\gamma_1(\delta;x,\xi)$, and it is not a geodesic. Thus
    \begin{equation}
        d_{g_1}(\gamma_1(-\delta;x,\zeta),\gamma_1(\delta;x,\xi))<2\delta.        
    \end{equation}
    Consequently, we have
    \begin{equation}
      d_{g_1}(\gamma_1(-\delta;x,\zeta),\widetilde{p})\leq d_{g_1}(\gamma_1(\delta;x,\xi),\widetilde{p})+d_{g_1}(\gamma_1(-\delta;x,\zeta),\gamma_1(\delta;x,\xi))<l+\delta.
    \end{equation}
    This is a contradiction with \eqref{eq:lm_6.6_1}. Thus we can conclude that $\zeta=\xi$ and $\widetilde{p}=\gamma_1(l;x,\xi)=p$.
    Now we have $d_{g_1}(p_j,z)>R$ for all $j>0$ and $p_j\to p$, which imply that $d_{g_1}(z,p)\geq R$. But $d_{g_1}(z,p)< R$ since $p\in B_1(z,R)$, a contradiction. Thus $d_{g_1}(z,p)=R^*$.\par
    Notice that $d_{g_1}(z,p)\leq d_{g_1}(z,x)+d_{g_1}(x,p)\leq \sigma^T_1(\{x\})-2\rho\leq T_{\bPhi}(\{x\})-2\rho$ and $T_{\bPhi}(z)\geq T_{\bPhi}(\{x\})-\rho$. Therefore $R^*=d_{g_1}(z,p)<T_{\bPhi}(z)$. 
    Since $\bPhi$ is an isometry, there holds $\bPhi(y)=\bPhi(\gamma_1(s;x,\xi))=\gamma_2(s;\bPhi(x),\bPhi_\ast \xi)$.
    According to Proposition \ref{lm:6}, \eqref{eq:prop_2} is equivalent to $B_2(\bPhi(y),r+\varepsilon)\subset B_2(\bPhi(x),l)\cup B_2(\bPhi(z),R)$ when $R<T_{\bPhi}(z)$. Repeating above argument on $(\M_2,g_2)$, we have $d_{g_2}(\bPhi(z),\gamma_2(l;\bPhi(x),\bPhi_\ast\xi))=R^*=d_{g_1}(z,\gamma_1(l;x,\xi))$.
  \end{proof}
  \begin{remark}
    By the continuity of the distance function, we can extend the result in Lemma \ref{lm:dist_influence} to $z\in \overline{B_1(x,\rho)}$.
  \end{remark}
  \begin{proposition}
    \label{prop:homeomorphism}
    Let $x\in \M_1^\inter$ and $\rho>0$. Let maps $R_k$ be defined as in \eqref{eq:def:R_1} and \eqref{eq:def:R_2} for $k=1,2$, respectively. Suppose there is an isometry $\bPhi: B_1(x,\rho)\to B_2(\bPhi(x),\rho)$ such that
    \begin{equation}
        \bPhi_\ast \mathfrak{B}_{B_1(x,\rho),\bPhi}=\mathfrak{B}_{B_2(\bPhi(x),\rho),\bPhi}.
    \end{equation}
    Then $R_k$ are homeomorphisms for both $k=1,2$, and $R_1(\overline{M_1^T(x,3\rho)})=R_2(\overline{M_2^T(\bPhi(x),3\rho)})$.
  \end{proposition}
  \begin{proof}
    In the view of \cite[Theorem 4.17]{rudin_1}, to conclude that $R_1$ is a homeomorphism we only need to show that $R_1$ are continuous and one-to-one as $\overline{M_1^T(x,3\rho)}$ is compact. Let $p,q\in M_1^T(x,3\rho)$, the triangle inequality reads
    \begin{equation}
      \norm{R_1(p)-R_1(q)}_{L^\infty(\overline{B_1(x,\rho)})}=\sup_{z\in \overline{B_1(x,\rho)}}\abs{d_{g_1}(p,z)-d_{g_1}(q,z)}\leq d_{g_1}(p,q).
    \end{equation}
    This implies that $R_1$ is continuous. \par
    Suppose $R_1(p)=R_1(q)$, there exists $\xi\in S_x\M$ such that $p=\gamma_1(t;x,\xi)$. Write $y=\gamma_1(\rho;x,\xi)$. Then $\gamma_1^\prime(\rho;x,\xi)$ is normal to $\p B(x,\rho)$ at $y$. Notice that
    \begin{equation}
      d_{g_1}(p,y)=\min_{z\in \overline{B_1(x,\rho)}}d_{g_1}(p,z)=\min_{z\in \overline{B_1(x,\rho)}}d_{g_1}(q,z).
    \end{equation}
    Then $q=\gamma_1(t-\rho;y,\gamma_1^\prime(\rho;x,\xi))=p$  by the uniqueness of the closest point on $\p B_1(x,\rho)$. Hence $R_1$ is continuous and one-to-one and thus holomorphic. Analogously, we can conclude that $R_2$ is also a homeomorphism.

    By the definition of $M^T_1(x,3\rho)$, for any $p_1\in M^T_1(x,3\rho)$, there is a normalized vector $\xi\in S_x\M_1$ such that $p_1 = \gamma_1(t;x,\xi)$ with $0<t<\sigma_1^T(x)-3\rho$. By Lemma \ref{lm:cut_func}, we have $p_2=\gamma_2(t;\bPhi(x),\bPhi_\ast\xi)\in M_2^T(x,3\rho)$. According to Lemma \ref{lm:dist_influence}, we have $d_{g_1}(p_1,z)=d_{g_2}(p_2,\bPhi(z))$ for all $z\in B_1(x,\rho)$. That is, $R_1(M_1^T(x,3\rho))=R_2(M_2^T(\bPhi(x),3\rho))$. Leveraging the continuity of $R_1$ and $R_2$, it follows that $R_1(\overline{M_1^T(x,3\rho)})=R_2(\overline{M_2^T(\bPhi(x),3\rho)})$.
  \end{proof}
  To simplify the notation, from now on for $x\in \M^\inter_1$ and $\rho>0$, we denote $R_1(\overline{M_1^T(x,3\rho)})$ by $\Sigma_{x,\rho}$. According to Proposition \ref{prop:homeomorphism}, we know that $R_2(\overline{M_2^T(\bPhi(x),3\rho)})=\Sigma_{x,\rho}$.
  \begin{lemma}
    \label{lm:R_identity}
    Let $x\in \M^\inter_1$, $\rho>0$ and $r_0\in \Sigma_{x,\rho}$. Write $y_1=R_1^{-1}(r_0)$ and $y_2=R_2^{-1}(r_0)$. Suppose there is an isometry $\bPhi: B_1(x,\rho)\to B_2(\bPhi(x),\rho)$ such that
    \begin{equation}
        \bPhi_\ast \mathfrak{B}_{B_1(x,\rho),\bPhi}=\mathfrak{B}_{B_2(\bPhi(x),\rho),\bPhi}.
    \end{equation}
    Let $\xi\in S_x\M_1$ be such that $y_1=\gamma_1(r_0(x);x,\xi)$, then $y_2=\gamma_2(r_0(x);\bPhi(x),\bPhi_\ast\xi)$.
  \end{lemma}
  \begin{proof}
    Choose $\delta<\rho$ such that $B_1(x,\delta)$ and $B_2(\bPhi(x),\delta)$ are geodesically convex neighborhoods of $x$ and $\bPhi(x)$, respectively. Notice that 
    \begin{align}
    \label{eq:lm_6.9_1}
      \max_{z\in \overline{B_1(x,\delta)}}r_0(z)&=r_0(x)+\delta=\max_{z\in \overline{B_1(x,\delta)}}d_{g_1}(y_1,z)=d_{g_1}(y_1,\gamma_1(-\delta;x,\xi))\\
      &=d_{g_2}(y_2,\gamma_2(-\delta;\bPhi(x),\bPhi_\ast\xi))=\max_{z\in \overline{B_1(x,\delta)}}d_{g_2}(y_2,\bPhi(z)).
    \end{align}
    To get a contradiction, we suppose there is $\zeta\in S_{\bPhi(x)}\M_2$ such that $y_2=\gamma_2(r_0(x);\bPhi(x),\zeta)$ and $\zeta\neq \bPhi_\ast \xi$. Thus $d_{g_2}(y_2,\gamma_2(\delta;\bPhi(x),\zeta))=r_0(x)-\delta$. Since the concatenated curve $\gamma_2([-\delta,0];\bPhi(x),\bPhi_\ast\xi)\cup \gamma_2([0,\delta];\bPhi(x),\zeta)$ joins $\gamma_2(-\delta;\bPhi(x),\bPhi_\ast\xi)$ and $\gamma_2(\delta;\bPhi(x),\zeta)$ and is not a geodesic, it holds that
    \begin{equation}
      d_{g_2}(\gamma_2(-\delta;\bPhi(x),\bPhi_\ast\xi),\gamma_2(\delta;\bPhi(x),\zeta))<2\delta.
    \end{equation}
    Therefore,
    \begin{align}
      d_{g_2}(y_2,\gamma_2(-\delta;\bPhi(x),\bPhi_\ast\xi))&\leq d_{g_2}(y_2,\gamma_2(\delta;\bPhi(x),\zeta))\\
      &+d_{g_2}(\gamma_2(\delta;\bPhi(x),\zeta),\gamma_2(-\delta;\bPhi(x),\bPhi_\ast\xi))<r_0(x)+\delta.
    \end{align}
    This is a contradiction with \eqref{eq:lm_6.9_1}.
  \end{proof}
  \begin{proposition}
    \label{prop:diffeomorphism}
    Let $x\in \M^\inter_1$, $\rho>0$ and $r_0\in \Sigma_{x,\rho}$. Suppose there is an isometry $\bPhi: B_1(x,\rho)\to B_2(\bPhi(x),\rho)$ such that
    \begin{equation}
        \bPhi_\ast \mathfrak{B}_{B_1(x,\rho),\bPhi}=\mathfrak{B}_{B_2(\bPhi(x),\rho),\bPhi}.
    \end{equation}
    Then there is a smooth structure on $\Sigma_{x,\rho}$ such that $R_k$ is a diffeomorphism for both $k=1,2$. Furthermore, $\overline{M_1^T(x,3\rho)}$ and $\overline{M_2^T(\bPhi(x),3\rho)}$ are diffeomorphic and $R_2^{-1}\circ R_1|_{B_1(x,\rho)}=\bPhi$.
  \end{proposition}
  \begin{proof}
    For any $p\in \overline{M_1^T(x,3\rho)}$, we denote the function $R_1(p)\in C(\overline{B_1(x,\rho)})$ by $r_p$.
    Since there is no cut locus of $x$ in $\overline{M_1^T(x,3\rho)}$, $r_p$ is smooth at $x$. We define the map
    \begin{equation}
      \Psi_x(r_p):= r_p(x)\grad_{g_1} r_p|_x,\ p\in \overline{M_1^T(x,3\rho)}.
    \end{equation}
    We claim that 
    \begin{equation}
      \label{eq:diff_claim}
      \Psi_x\circ R_1 = \exp_x^{-1},
    \end{equation}
    where $\exp_x$ is the exponential map at $x$. Indeed, let $p\in \overline{M_1^T(x,3\rho)}$ and $\xi=-\grad_{g_1}d_{g_1}(p,\cdot)|_x$, then
    \begin{equation}
      \gamma_1(d_{g_1}(p,x);x,\xi)=\exp_x(d_{g_1}(p,x)\xi)=p.
    \end{equation}
    Define a smooth structure on $R_1(\overline{M_1^T(x,3\rho)})$ by requiring that $\Psi_x$ is smooth. In light of \eqref{eq:diff_claim}, $R_1$ is a diffeomorphism.
    
    According to Lemma \ref{lm:R_identity}, there holds
    \begin{equation}
      \label{eq:R_k_relation}
      \Psi_x^{-1}=R_1\circ \exp_x = R_2\circ \exp_{\bPhi(x)}\circ \bPhi_\ast.
    \end{equation}
    Hence, $R_2$ is a diffeomorphism with the smooth structure on $\Sigma_{x,\rho}$ such that $\Psi_x$ is smooth. Since
    \begin{equation}
      \overline{M_2^T(\bPhi(x),3\rho)}=R_2^{-1}\circ R_1 (\overline{M_1^T(x,3\rho)}),
    \end{equation}
    $\overline{M_1^T(x,3\rho)}$ and $\overline{M_2^T(\bPhi(x),3\rho)}$ are diffeomorphic. Moreover, since $\bPhi$ is an isometry associated with $B_1(x,\rho)$, for any $0\leq s<\rho$ and $\xi\in S_x\M_1$, there holds 
    \begin{equation}
        \bPhi(\gamma_1(s;x,\xi))=\gamma_2(s;\bPhi(x),\bPhi_\ast \xi).
    \end{equation}
    Then in light of \eqref{eq:R_k_relation}, we have
    \begin{align}
        R_1(\gamma_1(s;x,\xi))&=R_1\circ \exp_x (s\xi)=R_2\circ\exp_{\bPhi(x)}(s \bPhi_\ast \xi)=\\
        &=R_2(\gamma_2(s;\bPhi(x),\bPhi_\ast \xi))=R_2(\bPhi(\gamma_1(s;x,\xi))).
    \end{align}
    Therefore, we have $R_2^{-1}\circ R_1|_{B_1(x,\rho)}=\bPhi$.
  \end{proof}

  \begin{proof}[Proof of Proposition \ref{prop:isometry}]
    Define the metric tensor $\widetilde{g_k}:=(R^{-1}_k)^* g_k$ on $\Sigma_{x,\rho}$ for $k=1,2$. Then 
    \begin{equation}
        R_1:(\overline{M_1^T(x,3\rho)},g_1)\to (\Sigma_{x,\rho},\widetilde{g_1})
    \end{equation}
    and 
    \begin{equation}
        R_2:(\overline{M_2^T(\bPhi(x),3\rho)},g_2)\to (\Sigma_{x,\rho},\widetilde{g_2})
    \end{equation}
    are isometries.
    
    Let $r_0\in \Sigma_{x,\rho}^\inter$, $y_1=R_1^{-1}(r_0)$ and $y_2=R_2^{-1}(r_0)$. Since $x$ is not in the cut locus of $y_1$ and the cut locus of $y_1$ is a closed set, we can find a small enough $\eta>0$ such that there is no cut locus of $y_1$ in $B(x,\eta)$ and thus $r_0$ is smooth in $B(x,\eta)$.\par
    Let $\xi\in S_x\M_1$ such that $y_1=\gamma_1(r_0(x);x,\xi)$. Then we can find a small enough $\delta>0$ such that $B_1(\gamma_1(-\delta;x,\xi),\delta)\subset B_1(x,\eta)$. Observe that 
    \begin{equation}
      d_{g_1}(x,y_1)=\min_{z\in \p B_1(\gamma_1(-\delta;x,\xi),\delta)}d_{g_1}(z,y_1).
    \end{equation}
    Define
    \begin{equation}
      \mathcal{W}:=\{\grad_{g_1}d_{g_1}(z,\cdot)|_{y_1}\mid z\in B_1(x,\eta)\}.
    \end{equation}
    According to \cite[Lemma 2.15]{KKL}, there is an open subset $U$ on $S_{y_1}\M_1$ such that
    \begin{equation}
      U\subset \{\grad_{g_1}d_{g_1}(z,\cdot)|_{y_1}\mid z\in \p B_1(x,\delta)\}\subset \mathcal{W}\subset S_{y_1}\M_1.
    \end{equation}
    For any $z\in B_1(x,\rho)$, we define functions on $\Sigma_{x,\rho}$
    \begin{equation}
      E_{1,z}(R_1(y))=R_1[y](z)=d_{g_1}(z,y),\ y\in \overline{M_1^T(x,3\rho)},
    \end{equation}
    and 
    \begin{equation}
      E_{2,z}(R_2(y))=R_2[y](\bPhi(z))=d_{g_2}(\bPhi(z),y),\ y\in \overline{M_2^T(\bPhi(x),3\rho)}.
    \end{equation}
    Let $p_1\in \overline{M_1^T(x,3\rho)}$, we denote $R_1(p_1)$ by $\tau$ and $R^{-1}_2(\tau)$ by $p_2$. Let $\zeta\in S_x\M$ such that $p_1=\gamma_1(\tau(x);x,\zeta)$. By Lemma \ref{lm:R_identity}, we have $p_2=\gamma_2(\tau(x);\bPhi(x),\bPhi_\ast\zeta)$. According to Lemma \ref{lm:dist_influence}, for any $z\in B_1(x,\rho)$, there holds
    \begin{align}
      &E_{1,z}(\tau)=E_{1,z}(R_1(p_1))=d_{g_1}(z,p_1)=d_{g_1}(z,\gamma_1(\tau(x);x,\zeta))\\
      &=d_{g_2}(\bPhi(z),\gamma_2(\tau(x);\bPhi(x),\bPhi_\ast\zeta))=d_{g_2}(\bPhi(z),p_2)=E_{2,z}(R_2(p_2))=E_{2,z}(\tau).
    \end{align}
    Hence $E_{1,z}=E_{2,z}$ as a function on $\Sigma_{x,\rho}$. Notice that for any $z\in B_1(x,\eta)$, the function $d_{g_1}(z,\cdot)$ is smooth in a neighborhood of $y_1$. Equip $\Sigma_{x,\rho}$ with the smooth structure in Proposition \ref{prop:diffeomorphism}. Then $E_{1,z}$ is smooth in a neighborhood of $R_1(y_1)=r_0$. Moreover, $E_{2,z}$ is smooth at $r_0$ and we have the co-vector $dE_{1,z}|_{r_0}=dE_{2,z}|_{r_0}$.
    We write
    \begin{equation}
      \mathcal{W}^*:=\{d_1 E_{1,z}|_{r_0}\mid z\in B_1(x,\eta)\}.
    \end{equation}
    
    Notice that for any $z\in B_1(x,\eta)$, there holds
    \begin{equation}
        \abs{dE_{1,z}|_{r_0}}_{\widetilde{g_1}} = \abs{d\, d_{g_1}(z,y)|_{y_1}}_{g_1}=1,
    \end{equation}
    and 
    \begin{equation}
        \abs{dE_{2,z}|_{r_0}}_{\widetilde{g_2}} = \abs{d\, d_{g_2}(\bPhi(z),y)|_{y_2}}_{g_2}=1.
    \end{equation}
    Since $\mathcal{W}^\ast\subset S^\ast_{r_0}\Sigma_{x,\rho}$ is open, the set 
    \begin{equation}
      \mathbb{R}_+ \mathcal{W}^* : = \{c\ dE_{1,z}|_{r_0}\mid c\in \mathbb{R}_+,\ z\in B_1(x,\eta)\}.
    \end{equation}
    contains an open cone in $T_{r_0}^*\Sigma_{x,\rho}$. For any $u\in \mathbb{R}_+ \mathcal{W}^*$, there exist $c_u\in \mathbb{R}_+$ and $z_u\in B_1(x,\eta)$ such that 
    \begin{equation}
        u = c_u d E_{1,z_u} = c_u d E_{2,z_u}.
    \end{equation}
    Moreover, there holds
    \begin{equation}
        (u,u)_{\widetilde{g_1}}=\abs{u}_{\widetilde{g_1}}^2 = c_u^2 = \abs{u}_{\widetilde{g_2}}^2=(u,u)_{\widetilde{g_2}}.
    \end{equation}
    Let $v\in \mathbb{R}_+ \mathcal{W}^*$. For any $\xi,\zeta\in T_{r_0}^*\Sigma_{x,\rho}$ and small enough $\varepsilon,\delta>0$, we have $v+\varepsilon\xi+\delta\zeta\in \mathbb{R}_+\mathcal{W}^\ast$. Therefore,
    \begin{align}
        (\xi,\zeta)_{\widetilde{g_1}} & = \lim_{\varepsilon\to 0}\lim_{\delta\to 0}\frac{1}{2\varepsilon\cdot \delta}\left(\abs{v+\varepsilon \xi +\delta \zeta}_{\widetilde{g_1}}^2-\abs{v+\varepsilon \xi}_{\widetilde{g_1}}^2-\abs{v +\delta \zeta}_{\widetilde{g_1}}^2+ \abs{v}_{\widetilde{g_1}}^2 \right)\\
        &= \lim_{\varepsilon\to 0}\lim_{\delta\to 0}\frac{1}{2\varepsilon\cdot \delta}\left(\abs{v+\varepsilon \xi +\delta \zeta}_{\widetilde{g_2}}^2-\abs{v+\varepsilon \xi}_{\widetilde{g_2}}^2-\abs{v +\delta \zeta}_{\widetilde{g_2}}^2+ \abs{v}_{\widetilde{g_2}}^2 \right)\\
        &=(\xi,\zeta)_{\widetilde{g_2}}
    \end{align}
    Since $\xi,\zeta\in T_{r_0}^*\Sigma_{x,\rho}$ are arbitrary, we have $\widetilde{g_1}(r_0)=\widetilde{g_2}(r_0)$.

  \end{proof}
  \subsection{Reconstruction of metric on the entire manifold}
  In this subsection, we will complete the proof of Theorem \ref{thm:main}. To this end, we prepare several auxiliary lemmas.
  
    For $k=1,2$ and two open sets $U,V\subset \M_k$, we define the map 
  \begin{equation}
    L_{(k),U,V} \bf = \bu_k^{\bf}|_{(0,T)\times V},\ \bf\in C_0^\infty((0,\infty);C_0^\infty\Omega^1(U)),
  \end{equation}
  where $\bu_k^{\bf}$ is the solution of \eqref{eq:problem_hypo} on $(\M_k,g_k)$.
   \begin{lemma}
    \label{lm:extend}
    Let $U_1,U_2\in \M_1^\inter$ be connected open sets with isometries 
    \begin{equation}
        \bPhi:U_1\to\bPhi(U_1),\ \bPsi: U_2\to \bPsi(U_2).
    \end{equation}
    Suppose $\Lambda_{1,U_1}\bPhi^\ast=\bPhi^\ast \Lambda_{2,\bPhi(U_1)}$, and $U_1\cap U_2\neq \emptyset$, $\bPhi|_{U_1\cap U_2}=\bPsi|_{U_1\cap U_2}$. Then there is an isometry $\tbPhi$ on $U_1\cup U_2$ such that 
    \begin{equation}
      \Lambda_{1,U_1\cup U_2}\tbPhi^\ast=\tbPhi^\ast \Lambda_{2,\tbPhi(U_1\cup U_2)},
    \end{equation}
    and 
    \begin{equation}
        \tbPhi|_{U_1}=\bPhi,\ \tbPhi|_{U_2}=\bPsi.
    \end{equation}
  \end{lemma}
  \begin{proof}
    We define 
    \begin{equation}
      \tbPhi(x)=\begin{cases}
        \bPhi(x),\text{ if }x\in U_1\\
        \bPsi(x),\text{ if }x\in U_2
      \end{cases}.
    \end{equation}
    Since $\bPhi|_{U_1\cap U_2}=\bPsi|_{U_1\cap U_2}$, $\tbPhi$ is well-defined. And it is clear that 
    \begin{equation}
      \tbPhi: U_1\cup U_2\to \bPhi(U_1)\cup \bPsi(U_2)
    \end{equation}
    is an isometry.

    Notice that $\tbPhi|_{U_1}=\bPhi$. In particular, for any $\bf\in C_0^\infty((0,\infty);C_0^\infty\Omega^1(\tbPhi(U_1)))$, we have
    $\Lambda_{1,U_1}\tbPhi^\ast\bf=\tbPhi^\ast \Lambda_{2,\tbPhi(U_1)}\bf$. Write 
    \begin{equation}
      \widetilde{\bu}^\bf=\bu_1^{\tbPhi^\ast\bf}-\tbPhi^\ast\bu_2^\bf,
    \end{equation}
    where $(\bu^\bf_k,p_k^\bf)$ are the solutions of the system \eqref{eq:problem_hypo} in $(\M_k,g_k)$ for $k=1,2$, respectively. Then $\Lambda_{1,U_1}\tbPhi^\ast=\tbPhi^\ast\Lambda_{2,\tbPhi(U_1)}$ reads that $\widetilde{\bu}^\bf=0$ in $\RR\times U_1$. Furthermore, since $\tbPhi$ is an isometry, there holds
    \begin{equation}
      \begin{cases}
      (\p_t^2-\Delta_H)\widetilde{\bu}^\bf + d (p_1^{\tbPhi^\ast\bf}-\tbPhi^\ast p_2^\bf)=0,\ \text{in } \RR\times (U_1\cup U_2),\\
      d^* \widetilde{\bu}^\bf=0,\ \text{in } [0,T]\times U_1\cup U_2.
      \end{cases}
    \end{equation}
    Since $U_1\cap U_2\neq \emptyset$ and $U_2$ is connected, $U_1\cup U_2\subset M(U_1,r)$ large $r$. According to the unique continuation (Theorem \ref{thm:UC}), $\widetilde{\bu}^\bf=0$ in $U_1\cup U_2$. Thus we have
    \begin{equation}
      L_{(1),U_1,U_1\cup U_2}\tbPhi^\ast=\tbPhi^\ast L_{(2),\tbPhi(U_1),\tbPhi(U_1\cup U_2)}.
    \end{equation}
    Taking the adjoint of $L_{(1),U_1,U_1\cup U_2}$ and $L_{(2),\tbPhi(U_1),\tbPhi(U_1\cup U_2)}$ and conjugating them with the operator reversing the time, we have
    \begin{equation}
      \label{eq:reverse_operator}
      L_{(1),U_1\cup U_2,U_1}\tbPhi^\ast=\tbPhi^\ast L_{(2),\tbPhi(U_1\cup U_2),\tbPhi(U_1)}.
    \end{equation}
    For any $\bh\in C_0^\infty((0,\infty);C_0^\infty\Omega^1(\tbPhi(U_1\cup U_2)))$, again we write $\widetilde{\bu}^\bh=\bu_1^{\tbPhi^\ast\bh}-\tbPhi^\ast\bu_2^\bh$. Then \eqref{eq:reverse_operator} gives that $\widetilde{\bu}^\bh=0$ in $U_1\times \RR$, and there holds 
    \begin{equation}
      \begin{cases}
      (\p_t^2-\Delta_H)\widetilde{\bu}^\bh + d (p_1^{\tbPhi^\ast\bh}-\tbPhi^\ast p_2^\bh)=0,\ \text{in } \RR\times (U_1\cup U_2),\\
      d^* \widetilde{\bu}^\bh=0,\ \text{in } \RR\times U_1\cup U_2.
      \end{cases}
    \end{equation}
    Using the unique continuation again, we can conclude that $\widetilde{\bu}^\bh=0$ in $\RR\times U_1\cup U_2$. Hence $ \Lambda_{1,U_1\cup U_2}\tbPhi^\ast=\tbPhi^\ast\Lambda_{2,\tbPhi(U_1\cup U_2)}$.
  \end{proof}
    Lemma \ref{lm:extend} and Proposition \ref{prop:isometry} imply the following result.
  \begin{lemma}
    \label{lm:operator_continuation}
    Let $x\in \M_1^\inter$, $T>0$, and $\rho>0$ such that there is an isometry
    \begin{equation}
      \bPhi:B_1(x,\rho)\to B_2(\bPhi(x),\rho)
    \end{equation}
    and $\Lambda_{1,B_1(x,\rho)}\bPhi^\ast=\bPhi^\ast\Lambda_{2,B_2(\bPhi(x),\rho)}$. Then there exists an extended isometry of $\bPhi$, denoted by $\tbPhi$, such that
    \begin{equation}
      \tbPhi:\overline{M_1^T(x,3\rho)}\to \overline{M_2^T(\bPhi(x),3\rho)},
    \end{equation}
    and 
    \begin{equation}
      \Lambda_{1,M_1^T(x,3\rho)}\tbPhi^\ast=\tbPhi^\ast\Lambda_{2,M_2^T(\bPhi(x),3\rho)}.
    \end{equation}
  \end{lemma}
  For $k=1,2$ and any $x,y\in \overline{\M_k}$, there holds
  \begin{equation}
      T_k(\{y\})-d_{g_k}(x,y) \leq T_k(\{x\})\leq T_k(\{y\})+d_{g_k}(x,y).
  \end{equation}
  Thus $T_1$ and $T_2$ are continuous positive functions on $\overline{\M_1}$ and $\overline{\M_2}$, respectively. And we write
  \begin{equation}
      \cT:=\min_{x\in \overline{\M_1},y\in \overline{\M_2}}\{T_1(\{x\}),T_2(\{y\})\}>0.
  \end{equation}
  Let $\cV\subset U\subset \M_1$, we define the $k$-th $\cT$-extension of $\cV$ in $U$ as
  \begin{equation}
      E^{(k)}_\cT(\cV,U):=M_1(E^{(k-1)}(\cV),\cT)\cap U,
  \end{equation}
  where 
  \begin{equation}
       E^{(0)}_\cT(\cV,U):=\cV.
  \end{equation}
  And we say that $U$ is $\cT$-exhaustive by $\cV$ if there exists $N>0$, such that $E^{(N)}_\cT(\cV,U)=U$. We note that if $U\subset \M_1$ is connected, then $U$ is $\cT$-exhaustive by any nonempty $\cV\subset U$.
  
  Next we show the local reconstruction can be glued together.\par
  \begin{lemma}
    \label{lm:glue}
    Let $U_1,U_2\subset \M_1^\inter$ be open sets with isometries 
    \begin{equation}
        \bPhi:U_1\to\bPhi(U_1),\ \bPsi: U_2\to \bPsi(U_2).
    \end{equation}
    Suppose $\Lambda_{1,U_1}\bPhi^\ast=\bPhi^\ast \Lambda_{2,\bPhi(U_1)}$, $\Lambda_{1,U_2}\bPsi^\ast=\bPsi^\ast \Lambda_{2,\bPsi(U_2)}$ and there exists an open set $\cV\subset U_1\cap U_2$ such that $\bPhi|_\cV=\bPsi|_\cV$ and $U_1\cap U_2$ is $\cT$-exhaustive by $\cV$. Then there holds
    \begin{equation}
    \label{eq:lm_6.14}
      \bPhi|_{U_1\cap U_2}=\bPsi|_{U_1\cap U_2}.
    \end{equation}
  \end{lemma}
  \begin{proof}
  Let us show that 
  \begin{equation}
  \label{eq:lm_6.14_1}
      \bPhi|_{E^{(1)}_{\cT}(\cV,U_1\cap U_2)}=\bPsi|_{E^{(1)}_{\cT}(\cV,U_1\cap U_2)}.
  \end{equation}
  Let $y\in U_1\cap U_2\cap M(\cV,\cT)$ be arbitrary. To get a contradiction, we assume that $\bPhi(y)\neq \bPsi(y)$. Let $z\in \cV$ be such that $d_{g_1}(y,z)<\cT$. For simplicity, we denote $d_{g_1}(y,z)$ by $s$. Since $\bPhi$ is an isometry, we have $d_{g_2}(\bPhi(y),\bPhi(z))=s<\cT$. Then we can choose a small enough $0<\varepsilon<\frac{1}{2}(\cT-s)$ such that $B_2(\bPhi(y),2\varepsilon)\subset \bPhi(U_1)$, $B_2(\bPsi(y),2\varepsilon)\subset \bPsi(U_2)$, $B_2(\bPhi(y),2\varepsilon)\cap B_2(\bPsi(y),2\varepsilon)=\emptyset$, and $\M_2\setminus B_2(\bPhi(y),\varepsilon)$ and $\M_2\setminus  B_2(\bPsi(y),\varepsilon)$ are connected.
    By approximate controllability, we may take $\bf\in C_0^\infty((T-\varepsilon,T);C_0^\infty\Omega^1(B_2(\bPhi(y),\varepsilon)))$ such that $\bu_2^{\bf}(T,\cdot)\neq 0$. For any 
    \begin{align}
      \bh\in C_0^\infty((T-\varepsilon,T);C_0^\infty\Omega^1(B_2(\bPsi(y),\varepsilon))),
    \end{align}
    due to the finite speed of propagation, we have
    \begin{equation}
      \supp(\bu_2^\bf(T,\cdot))\cap \supp(\bu_2^\bh(T,\cdot))=\emptyset.
    \end{equation}
    Since
    \begin{align}
      \bPhi^\ast: C_0^\infty((T-\varepsilon,T);C_0^\infty\Omega^1(B_2(\bPhi(y),\varepsilon)))\to C_0^\infty((T-\varepsilon,T);C_0^\infty\Omega^1(B_1(y,\varepsilon))),\\
      \bPsi^\ast: C_0^\infty((T-\varepsilon,T);C_0^\infty\Omega^1(B_2(\bPsi(y),\varepsilon)))\to C_0^\infty((T-\varepsilon,T);C_0^\infty\Omega^1(B_1(y,\varepsilon)))
    \end{align}
    are vector space isomorphism, we can choose $\bPsi^\ast \bh = \bPhi^\ast \bf$. Let $\delta>0$ be such that $B_1(z,\delta)\subset \cV$. Then we have $B_2(\bPhi(y),2\varepsilon)\subset M_2(B_2(\bPhi(z),\delta),s+2\varepsilon-\delta)$. Since 
    $$T_2(B_2(\bPhi(z),\delta))\geq T_{\bPhi}(B_1(z,\delta))\geq \cT-\delta>s+2\varepsilon-\delta,$$ 
    by the approximate controllability, there exists a sequence
    $$\{\bphi_n\}_{n=1}^\infty\subset C_0^\infty((T-(s+2\varepsilon-\delta),T);C_0^\infty\Omega^1(B_2(\bPhi(z),\delta))$$ 
    such that
    $\bu_2^{\bphi_n}(T,\cdot)\to \bu_2^{\bf}(T,\cdot)$ as $n\to \infty$. Since $$B_2(\bPhi(y),2\varepsilon)\cap B_2(\bPsi(y),2\varepsilon)=\emptyset,$$
    there holds
    \begin{equation}
        \lim_{n\to \infty}\inner{\bu_2^{\bh}(T,\cdot),\bu_2^{\bphi_n}(T,\cdot)}=0,
    \end{equation}
    and 
    \begin{equation}
        \lim_{n\to \infty}\inner{\bu_2^{\bf}(T,\cdot),\bu_2^{\bphi_n}(T,\cdot)}=\norm{\bu_2^{\bf}(T,\cdot)}_{L^2\Omega^1(\M_2)}>0.
    \end{equation}
    Then for large enough $N$ we write $\bpsi=\bphi_N$, and we have
    \begin{equation}
      \label{eq:glue_contradiction}
      \inner{\bu_2^{\bf-\bh}(T,\cdot),\bu_2^{\bpsi}(T,\cdot)}>0.
    \end{equation}
     Notice that $\bPhi^\ast \bpsi=\bPsi^\ast \bpsi$ as $\bPhi|_\cV=\bPsi|_\cV$. According to Lemma \ref{lm:norm_equal} and the fact that $\bPsi^\ast \bh = \bPhi^\ast \bf$, we have
    \begin{align}
      \inner{\bu_2^{\bf-\bh}(T,\cdot),\bu_2^{\bpsi}(T,\cdot)}&=\inner{\bu_2^{\bf}(T,\cdot),\bu_2^{\bpsi}(T,\cdot)}-\inner{\bu_2^{\bh}(T,\cdot),\bu_2^{\bpsi}(T,\cdot)}\\
      &=\inner{\bu_1^{\bPhi^\ast\bf}(T,\cdot),\bu_1^{\bPhi^\ast\bpsi}(T,\cdot)}-\inner{\bu_1^{\bPsi^\ast\bh}(T,\cdot),\bu_1^{\bPsi^\ast\bpsi}(T,\cdot)}\\
      &=\inner{\bu_1^{\bPhi^\ast\bf-\bPsi^\ast\bh}(T,\cdot),\bu_1^{\bPhi^\ast\bpsi}(T,\cdot)}=0.
    \end{align}
    This is in contradiction with \eqref{eq:glue_contradiction} and we have proved that \eqref{eq:lm_6.14_1} holds as $y\in E^{(1)}_{\cT}(\cV,U_1\cap U_2)$ is arbitrary. Since $U_1\cap U_2$ is $\cT$-exhaustive by $\cV$, there exists $N>0$ such that $U_1\cap U_2=E^{(N)}_{\cT}(\cV,U_1\cap U_2)$. Replacing $E^{(k)}_{\cT}(\cV,U_1\cap U_2)$ by $E^{(k+1)}_{\cT}(\cV,U_1\cap U_2)$ iteratively in above argument for $k=0,1,\cdots,N-1$, we can obtain \eqref{eq:lm_6.14}.
  \end{proof}
 \begin{proof}[Proof of Theorem \ref{thm:main}]
  Provided $\bPhi:\omega_1\to \omega_2$ is a diffeomorphism and $\Lambda_{1,\omega_1}\bPhi^\ast=\bPhi^\ast \Lambda_{2,\omega_2}$. we have $\bPhi$ is an isometry by Proposition \ref{prop:dist_omega}.

  We define the set
  \begin{equation}
    \cN:=\left\{U\mid \omega_1\subset U\subset \M_1^\inter,
    \begin{matrix}
      \ U\text{ is open, connected and there is an isometry}\\
      \tbPhi:U\to\tbPhi(U),\text{ s.t. }\tbPhi|_{\omega_1}=\bPhi,\  \Lambda_{1,U}\tbPhi^\ast=\tbPhi^\ast\Lambda_{2,\bPhi(U)}.
    \end{matrix}
    \right\}.
  \end{equation}
  We note that $\subset$ is a partial order relation on $\cN$, and it is clear that $\cN$ is not empty as $\omega_1\in \cN$. We claim that $\cN$ is inductive. That is, every totally ordered subset in $\cN$ has an upper bound. Indeed, let $\mathcal{Q}\subset \cN$ be a totally ordered subset. We write $\mathcal{Q}=\{U_i\}_{i\in I}$. There is an isometry $\bPhi_i$ on $U_i$ such that $\bPhi_i|_{\omega_1}=\bPhi$, $\Lambda_{1,U_i}\tbPhi_i^\ast=\tbPhi_i^\ast\Lambda_{2,\bPhi(U_i)}$ for all $i\in I$ by definition. Denote $\bigcup_{i\in I}U_i$ by $U$, we can define an isometry $\bPsi$ on $U$ that 
  \begin{equation}
    \bPsi(x)=\bPhi_i(x),\text{ if }x\in U_i\text{ for some }i. 
  \end{equation}
  Firstly, we check that $\bPsi$ is well-defined. Suppose $x\in U_i$ and $x\in U_j$. Since $\mathcal{Q}$ is totally ordered, without loss of generality, we assume that $U_i\subset U_j$. According to Lemma \ref{lm:glue} and $\bPhi_i|_{\omega_1}=\bPhi_j|_{\omega_1}$, there holds $\bPhi_j|_{U_i}=\bPhi_i$. Hence, $\bPhi$ is well defined. Furthermore, we can conclude that $\bPsi$ is an isometry in a neighborhood of $x$, and consequently, $\bPsi$ is an isometry in $U$ since $x$ is arbitrarily chosen.
  Next we verify that $\Lambda_{1,U}\bPsi^\ast=\bPsi^\ast\Lambda_{2,\bPsi(U)}$. For any $\bf\in C_0^\infty((0,\infty);C_0^\infty\Omega^1(U))$, $\bigcup_{i\in I}(0,\infty)\times U_i$ is an open cover of $\supp(\bf)$. Since $\bf$ is compactly supported and $\{U_i\}_{i\in I}$ is a totally ordered subset, there exists some $J\in I$ such that $\bf\in C_0^\infty((0,\infty);C_0^\infty\Omega^1(U_J))$. For any $x\in U$, there exists some $K\in I$ such that $U_J\subset U_K$ and $x\in U_K$. Since $\bPsi|_{U_K}=\bPhi_K$ and $\Lambda_{1,U_K}\bPhi_K^\ast=\bPhi_K^\ast\Lambda_{1,\bPhi_K(U_K)}$, we have
  \begin{equation}
    \bu_1^{\bPsi^\ast \bf}(x,t)=\bPsi^\ast\bu_2^{\bf}(x,t),\ t>0.
  \end{equation}
  Therefore, we obtain that $\Lambda_{1,U}\bPsi^\ast=\bPsi^\ast\Lambda_{2,\bPsi(U)}$, and consequently, $U\in \cN$. This indicates that $\cN$ is inductive. Thanks to Zorn's Lemma (see e.g. \cite[Theorem 2.1]{Dugundji66}, \cite[Lemma 1.1]{Brezis}), there is a maximal element of $\cN$ and we denote it by $\cU$. We write $\bPsi$ for the isometry associated with $\cU$.
  
  It is clear that $\cU$ is open in $\M_1^\inter$, next we show that $\cU$ is also closed in $\M_1^\inter$. To get a contradiction, we assume that $\cU$ is not closed in $\M_1^\inter$, then there exists $y\in \M_1^\inter$ and $y\not\in \cU$ such that there is a sequence $\{x_j\}_{j=1}^\infty\subset \cU$ converges to $y$. 
  And we can find small $0<\eta<\cT$ such that $B_1(y,\eta)$ is geodesically convex and $B_1(y,\eta)\subset \M_1^\inter$. Then there exists a large enough $N$ such that $x_N\in B_1(y,\frac{\eta}{3})$. Since $B_1(x_N,\frac{2\eta}{3})\subset B_1(y,\eta)$ which is a geodesically convex set and $T_{\bPhi}(x_N)>\cT>\eta$, we have $\sigma_1^T(x_N)\geq\frac{2\eta}{3}$. Let $0<\varepsilon<\frac{\eta}{12}$ be such that $B_1(x_N,\varepsilon)\subset \cU$, then we have $\Lambda_{1,B_1(x_N,\varepsilon)}\bPsi^\ast=\bPsi^\ast\Lambda_{2,B_2(\bPsi(x_N),\varepsilon)}$.
  
  According to Lemma \ref{lm:operator_continuation}, there is an isometry 
  \begin{equation}
    \Upsilon: M_1^T(x_N,3\varepsilon)\to M_2^T(\bPsi(x_N),3\varepsilon)
  \end{equation}
  such that $\Upsilon|_{B_1(x_N,\varepsilon)}=\bPsi|_{B_1(x_N,\varepsilon)}$, and
  \begin{equation}
    \Lambda_{1,M_1^T(x_N,3\varepsilon)}\Upsilon^\ast = \Upsilon^\ast \Lambda_{2,M_2^T(\bPsi(x_N),3\varepsilon)}.
  \end{equation}
  Since 
  $$\sigma_1^T(x_N)-3\varepsilon>\frac{2\eta}{3}-\frac{\eta}{4}=\frac{5\eta}{12},$$
  we have
  $$B_1(x_N,\frac{5\eta}{12})\subset B_1(x_N,\sigma_1^T(x_N)-3\varepsilon)=M_1^T(x_N,3\varepsilon).$$
  Then we consider the restriction of $\Upsilon$
  \begin{equation}
      \overline{\Upsilon}: B_1(x_N,\frac{5\eta}{12})\to B_2(\bPsi(x_N),\frac{5\eta}{12}).
  \end{equation}
  And we notice that 
  $\cU\cap B_1(x_N,\frac{5\eta}{12})$ is trivially $\cT$-exhaustive by $B_1(x_N,\varepsilon)$ since 
  \begin{align}
      E^{(1)}_\cT(B_1(x_N,\varepsilon),\cU\cap B_1(x_N,\frac{5\eta}{12}))&=B_1(x_N,\varepsilon+\cT)\cap \cU\cap  B_1(x_N,\frac{5\eta}{12})\\
      &=\cU\cap  B_1(x_N,\frac{5\eta}{12}).
  \end{align}
  According to Lemma \ref{lm:glue}, there holds
  \begin{equation}
    \overline{\Upsilon}|_{\cU\cap B_1(x_N,\frac{5\eta}{12})}=\bPsi|_{\cU\cap B_1(x_N,\frac{5\eta}{12})}.
  \end{equation}
  By Lemma \ref{lm:extend}, there is an isometry $\widetilde{\Upsilon}$ on $\cU\cup B_1(x_N,\frac{5\eta}{12})$ such that 
  \begin{equation}
    \Lambda_{1,\cU\cup B_1(x_N,\frac{5\eta}{12})}\widetilde{\Upsilon}^\ast = \widetilde{\Upsilon}^\ast\Lambda_{2,\widetilde{\Upsilon}(\cU\cup B_1(x_N,\frac{5\eta}{12}))}.
  \end{equation}
  Since $d_{g_1}(x_N,y)<\frac{\eta}{3}<\frac{5\eta}{12}$, there holds $y\in B_1(x_N,\frac{5\eta}{12})$.
  Hence, $\cU\cup B_1(x_N,\frac{5\eta}{12})\in \cN$ but $\cU\cup B_1(x_N,\frac{5\eta}{12})\not\subset \cU$ as $y\in B_1(x_N,\frac{5\eta}{12})$ and $y\not\in \cU$. This is a contradiction to $\cU$ being a maximal element of $\cN$. Therefore, $\cU=\M_1^\inter$ and there is an isometry $\tbPhi: \M_1^\inter\to \tbPhi(\M_1^\inter)$ such that $\tbPhi|_{\omega_1}=\bPhi$ and $\Lambda_{1,\M_1^\inter}\tbPhi^\ast=\tbPhi^\ast\Lambda_{2,\tbPhi(\M_1^\inter)}$.
  Interchanging the roles of $\M_1$ and $\M_2$ and repeating the above argument, we can obtain an isometry $\tbPsi:\M_2^\inter \to \tbPsi(\M_2^\inter)$ such that $\tbPsi|_{\omega_2}=\bPhi^{-1}$ and $\Lambda_{2,\M_2^\inter}\tbPsi^\ast = \tbPsi^\ast \Lambda_{1,\tbPsi(\M_2^\inter)}$. Since $(\tbPsi^\ast)^{-1}=(\tbPsi^{-1})^\ast$, there holds
  \begin{equation}
      \Lambda_{1,\tbPsi(\M_2^\inter)} (\tbPsi^{-1})^\ast = (\tbPsi^{-1})^\ast \Lambda_{2,\M_2^\inter}.
  \end{equation}
  Hence $\tbPsi(\M_2^\inter)\in \cN$, and consequently $\tbPsi(\M_2^\inter)\subset \cU = \M_1^\inter$. Moreover, by $\tbPsi^{-1}|_{\omega_1}=\tbPhi_{\omega_1}$ and Lemma \ref{lm:glue}, we have 
  \begin{equation}
      \M_2^\inter = \tbPsi^{-1}(\tbPsi(\M_2^\inter))=\tbPhi(\tbPsi(\M_2^\inter))\subset \tbPhi(\M_1^\inter)\subset \M_2^\inter.
  \end{equation}
  It follows that $\tbPhi(\M_1^\inter)= \M_2^\inter$.
  By the smoothness of the Riemannian structure near the boundary (see e.g. \cite[Section 6.5.3]{Isozaki2010}), we have $M_1$ and $M_2$ are isometric.
  \end{proof}
  \subsection*{Acknowledgements}
 Y. Kian would like to thanks Raphael Danchin for the fruitful discussion about the forward problem. The work of Y. Kian is supported by the French National Research Agency ANR and Hong Kong RGC Joint Research Scheme for the project IdiAnoDiff (grant ANR-24-CE40-7039).
 L.Oksanen was supported by the European Research Council of the European Union, grant 101086697 (LoCal), and the Research Council of Finland, grants 359182, 347715 and 353096.
Z.Zhao was supported by the Finnish Ministry of Education and Culture’s Pilot for Doctoral Programmes (Pilot project Mathematics of Sensing, Imaging and Modelling).
Views and opinions expressed are those of the authors only and do not necessarily reflect those of the European Union or the other funding organizations.
\subsection*{Conflict of interest}
The authors have no conflicts of interest to declare that are relevant to this article.
\subsection*{Ethical statement}
This study was conducted in accordance with all relevant ethical guidelines and regulations.
\subsection*{Informed Consent}
All participants provided informed consent prior to taking part in the study.
\subsection*{Data availability statement}
Data sharing is not applicable to this article as no datasets were generated or analysed during the current study.
\bibliography{reference}

@book{GT,
    AUTHOR = {Gilbarg, David and Trudinger, Neil S.},
     TITLE = {{Elliptic Partial Differential Equations of Second Order}},
 PUBLISHER = {Springer-Verlag, Berlin-New York},
      YEAR = {1977},
     PAGES = {x+401},
      ISBN = {3-540-08007-4},
   MRCLASS = {35-02 (35J25 35J65)},
  MRNUMBER = {473443},
MRREVIEWER = {O.\ John},
}

@book{Th,
 author = {Aubin, Thierry},
 title = {Nonlinear analysis on manifolds. {Monge}-{Amp{\`e}re} equations},
 fseries = {Grundlehren der Mathematischen Wissenschaften},
 series = {Grundlehren Math. Wiss.},
 issn = {0072-7830},
 volume = {252},
 year = {1982},
 publisher = {Springer, Cham},
 language = {English},
 zbMATH = {3808476},
 Zbl = {0512.53044}
}

@article{Sol1,
	year = 2005,
	month = {nov}, 
	volume = {25},
	number = {12},
	pages = {1-40},
	author =  {Solonnikov, Vsevolod A.},
	title = {Schauder estimates for evolution generalized Stokes problem},
	journal = {PDMI-Preprint}
}

@incollection{Sol2,
 author = {Solonnikov, Vsevolod A.},
 title = {On {Schauder} estimates for the evolution generalized {Stokes} problem},
 booktitle = {Hyperbolic problems and regularity questions},
 isbn = {978-3-7643-7450-1; 978-3-7643-7451-8},
 pages = {197--205},
 year = {2007},
 publisher = {Basel: Birkh{\"a}user},
 language = {English},
 zbMATH = {5168415},
 Zbl = {1123.35005}
}

@Book{Taylor_PDE_3,
 Author = {Taylor, Michael E.},
 Title = {Partial differential equations {III}. {Nonlinear} equations},
 Edition = {3rd corrected and expanded edition},
 FSeries = {Applied Mathematical Sciences},
 Series = {Appl. Math. Sci.},
 ISSN = {0066-5452},
 Volume = {117},
 ISBN = {978-3-031-33927-1; 978-3-031-33930-1; 978-3-031-33928-8},
 Year = {2023},
 Publisher = {Cham: Springer},
 Language = {English},
 DOI = {10.1007/978-3-031-33928-8},
 zbMATH = {7756365},
 Zbl = {1527.35004}
}

@Book{Brezis,
 Author = {Brezis, Haim},
 Title = {Functional analysis, {Sobolev} spaces and partial differential equations},
 FSeries = {Universitext},
 Series = {Universitext},
 ISSN = {0172-5939},
 ISBN = {978-0-387-70913-0},
 Year = {2011},
 Publisher = {New York, NY: Springer},
 Language = {English},
 zbMATH = {5633610},
 Zbl = {1220.46002}
}

@article{BeKu,
 author = {Belishev, Michael I. and Kurylev, Yaroslav V.},
 title = {To the reconstruction of a {Riemannian} manifold via its spectral data ({BC}- method)},
 fjournal = {Communications in Partial Differential Equations},
 journal = {Commun. Partial Differ. Equations},
 issn = {0360-5302},
 volume = {17},
 number = {5-6},
 pages = {767--804},
 year = {1992},
 language = {English},
 doi = {10.1080/03605309208820863},
 zbMATH = {147418},
 Zbl = {0812.58094}
}

@article{Beli,
 author = {Belishev, M. I.},
 title = {On an approach to multidimensional inverse problems for the wave equation},
 fjournal = {Soviet Mathematics. Doklady},
 journal = {Sov. Math., Dokl.},
 issn = {0197-6788},
 volume = {36},
 number = {3},
 pages = {481--484},
 year = {1988},
 language = {English},
 zbMATH = {4080062},
 Zbl = {0661.35084}
}

@article{BCZ,
 author = {Beretta, E. and Cavaterra, C. and Ortega, J. H. and Zamorano, S.},
 title = {Size estimates of an obstacle in a stationary {Stokes} fluid},
 fjournal = {Inverse Problems},
 journal = {Inverse Probl.},
 issn = {0266-5611},
 volume = {33},
 number = {2},
 pages = {29},
 note = {Id/No 025008},
 year = {2017},
 language = {English},
 doi = {10.1088/1361-6420/33/2/025008},
 zbMATH = {6691659},
 Zbl = {1515.35202}
}

@article{ILY,
 author = {Imanuvilov, Oleg Y. and Lorenzi, Luca and Yamamoto, Masahiro},
 title = {Carleman estimate for the {Navier}-{Stokes} equations and applications},
 fjournal = {Inverse Problems},
 journal = {Inverse Probl.},
 issn = {0266-5611},
 volume = {38},
 number = {8},
 pages = {30},
 note = {Id/No 085006},
 year = {2022},
 language = {English},
 doi = {10.1088/1361-6420/ac4c33},
 zbMATH = {7559611},
 Zbl = {1492.35183}
}

@article{CIPY,
 author = {Choulli, Mourad and Imanuvilov, Oleg Yu. and Puel, Jean-Pierre and Yamamoto, Masahiro},
 title = {Inverse source problem for linearized {Navier}-{Stokes} equations with data in arbitrary sub-domain},
 fjournal = {Applicable Analysis},
 journal = {Appl. Anal.},
 issn = {0003-6811},
 volume = {92},
 number = {10},
 pages = {2127--2143},
 year = {2013},
 language = {English},
 doi = {10.1080/00036811.2012.718334},
 zbMATH = {6204296},
 Zbl = {1302.35432}
}

@misc{BGKN,
 author = {Blouza, Adel and Glangetas, L{\'e}o and Kian, Yavar and Ngo, Van-Sang},
 title = {Determination of time dependent source terms for {Stokes} systems in unbounded domains},
 year = {2024},
 howpublished = {Preprint, {arXiv}:2407.21589 [math.{AP}] (2024)},
 url = {https://arxiv.org/abs/2407.21589},
 arXiv = {arXiv:2407.21589}
}

@article{FDJN,
 author = {Fan, Jishan and di Cristo, Michele and Jiang, Yu and Nakamura, Gen},
 title = {Inverse viscosity problem for the {Navier}-{Stokes} equation},
 fjournal = {Journal of Mathematical Analysis and Applications},
 journal = {J. Math. Anal. Appl.},
 issn = {0022-247X},
 volume = {365},
 number = {2},
 pages = {750--757},
 year = {2010},
 language = {English},
 doi = {10.1016/j.jmaa.2009.12.012},
 zbMATH = {5676147},
 Zbl = {1186.35240}
}

@article{Sh,
 author = {Shkoller, Steve},
 title = {Analysis on groups of diffeomorphisms of manifolds with boundary and the averaged motion of a fluid.},
 fjournal = {Journal of Differential Geometry},
 journal = {J. Differ. Geom.},
 issn = {0022-040X},
 volume = {55},
 number = {1},
 pages = {145--191},
 year = {2000},
 language = {English},
 doi = {10.4310/jdg/1090340568},
 zbMATH = {1782651},
 Zbl = {1044.35061}
}

@article{MaRaSh,
 author = {Marsden, J. E. and Ratiu, T. S. and Shkoller, S.},
 title = {The geometry and analysis of the averaged {Euler} equations and a new diffeomorphism group},
 fjournal = {Geometric and Functional Analysis. GAFA},
 journal = {Geom. Funct. Anal.},
 issn = {1016-443X},
 volume = {10},
 number = {3},
 pages = {582--599},
 year = {2000},
 language = {English},
 doi = {10.1007/PL00001631},
 zbMATH = {1525496},
 Zbl = {0979.58004}
}

@article{CaCo1,
 author = {Alvarez, C. and Conca, C. and Friz, L. and Kavian, O. and Ortega, J. H.},
 title = {Identification of immersed obstacles via boundary measurements},
 fjournal = {Inverse Problems},
 journal = {Inverse Probl.},
 issn = {0266-5611},
 volume = {21},
 number = {5},
 pages = {1531--1552},
 year = {2005},
 language = {English},
 doi = {10.1088/0266-5611/21/5/003},
 url = {semanticscholar.org/paper/2fd72b0a314322e4c3bcc81193785e0cca1f09d0},
 zbMATH = {2218075},
 Zbl = {1088.35080}
}

@article{CaCo2,
 author = {Conca, Carlos and Malik, Muslim and Munnier, Alexandre},
 title = {Detection of a moving rigid solid in a perfect fluid},
 fjournal = {Inverse Problems},
 journal = {Inverse Probl.},
 issn = {0266-5611},
 volume = {26},
 number = {9},
 pages = {18},
 note = {Id/No 095010},
 year = {2010},
 language = {English},
 doi = {10.1088/0266-5611/26/9/095010},
 zbMATH = {5797220},
 Zbl = {1200.35320}
}

@article{Liu,
 author = {Liu, Genqian},
 title = {Determining the viscosity function from the boundary measurements for the {Stokes} and the {Navier}-{Stokes} equations},
 fjournal = {Inverse Problems},
 journal = {Inverse Probl.},
 issn = {0266-5611},
 volume = {40},
 number = {12},
 pages = {48},
 note = {Id/No 125011},
 year = {2024},
 language = {English},
 doi = {10.1088/1361-6420/ad8479},
 zbMATH = {7950667},
 Zbl = {1553.35242}
}

@article{ImYa,
 author = {Imanuvilov, O. Yu. and Yamamoto, M.},
 title = {Global uniqueness in inverse boundary value problems for the {Navier}-{Stokes} equations and {Lam{\'e}} system in two dimensions},
 fjournal = {Inverse Problems},
 journal = {Inverse Probl.},
 issn = {0266-5611},
 volume = {31},
 number = {3},
 pages = {46},
 note = {Id/No 035004},
 year = {2015},
 language = {English},
 doi = {10.1088/0266-5611/31/3/035004},
 zbMATH = {6422193},
 Zbl = {1331.35261}
}

@article{LaUWa,
 author = {Lai, Ru-Yu and Uhlmann, Gunther and Wang, Jenn-Nan},
 title = {Inverse boundary value problem for the {Stokes} and the {Navier}-{Stokes} equations in the plane},
 fjournal = {Archive for Rational Mechanics and Analysis},
 journal = {Arch. Ration. Mech. Anal.},
 issn = {0003-9527},
 volume = {215},
 number = {3},
 pages = {811--829},
 year = {2015},
 language = {English},
 doi = {10.1007/s00205-014-0794-1},
 zbMATH = {6410249},
 Zbl = {1309.35085}
}

@article{HLiWa,
 author = {Heck, Horst and Li, Xiaosheng and Wang, Jenn-Nan},
 title = {Identification of viscosity in an incompressible fluid},
 fjournal = {Indiana University Mathematics Journal},
 journal = {Indiana Univ. Math. J.},
 issn = {0022-2518},
 volume = {56},
 number = {5},
 pages = {2489--2510},
 year = {2007},
 language = {English},
 doi = {10.1512/iumj.2007.56.3037},
 url = {ntur.lib.ntu.edu.tw/bitstream/246246/216125/1/Identification of viscosity in an incompressible fluid.pdf},
 zbMATH = {5215061},
 Zbl = {1142.35103}
}

@article{LiWa,
 author = {Li, Xiaosheng and Wang, Jenn-Nan},
 title = {Determination of viscosity in the stationary {Navier}-{Stokes} equations},
 fjournal = {Journal of Differential Equations},
 journal = {J. Differ. Equations},
 issn = {0022-0396},
 volume = {242},
 number = {1},
 pages = {24--39},
 year = {2007},
 language = {English},
 doi = {10.1016/j.jde.2007.07.008},
 zbMATH = {5210062},
 Zbl = {1143.35083}
}

@book{Val,
 author = {Vallis, Geoffrey K.},
 title = {Atmospheric and oceanic fluid dynamics. {Fundamentals} and large-scale circulation},
 edition = {2nd edition},
 isbn = {978-1-107-06550-5; 978-1-107-58841-7},
 year = {2017},
 publisher = {Cambridge: Cambridge University Press},
 language = {English},
 doi = {10.1017/9781107588417},
 zbMATH = {6807261},
 Zbl = {1374.86002}
}

@book{VlKh,
 author = {Arnold, Vladimir I. and Khesin, Boris A.},
 title = {Topological methods in hydrodynamics},
 edition = {2nd edition},
 fseries = {Applied Mathematical Sciences},
 series = {Appl. Math. Sci.},
 issn = {0066-5452},
 volume = {125},
 isbn = {978-3-030-74277-5; 978-3-030-74280-5; 978-3-030-74278-2},
 year = {2021},
 publisher = {Cham: Springer},
 language = {English},
 doi = {10.1007/978-3-030-74278-2},
 zbMATH = {7349486},
 Zbl = {1475.76003}
}

@article{CK,
 author = {Canuto, B. and Kavian, O.},
 title = {Determining coefficients in a class of heat equations via boundary measurements},
 fjournal = {SIAM Journal on Mathematical Analysis},
 journal = {SIAM J. Math. Anal.},
 issn = {0036-1410},
 volume = {32},
 number = {5},
 pages = {963--986},
 year = {2001},
 language = {English},
 doi = {10.1137/S003614109936525X},
 zbMATH = {1578850},
 Zbl = {0981.35096}
}

@article{EMa,
 author = {Ebin, D. G. and Marsden, J.},
 title = {Groups of diffeomorphisms and the motion of an incompressible fluid},
 fjournal = {Annals of Mathematics. Second Series},
 journal = {Ann. Math. (2)},
 issn = {0003-486X},
 volume = {92},
 pages = {102--163},
 year = {1970},
 language = {English},
 doi = {10.2307/1970699},
 zbMATH = {3335835},
 Zbl = {0211.57401}
}

@misc{Gio,
 author = {Troianiello, Giovanni Maria},
 title = {Elliptic differential equations and obstacle problems},
 isbn = {0-306-42448-7},
 year = {1987},
 language = {English},
 howpublished = {The {University} {Series} in {Mathematics}. {New} {York}-{London}: {Plenum} {Press}. {XVI}, 370 p.; {\$} 49.50 ({US} \& {Canada}); {\$} 59.40 (outside {US} \& {Canada}) (1987).},
 zbMATH = {4069392},
 Zbl = {0655.35002}
}

@article{STONE, title={Hydrodynamics of particles embedded in aflat surfactant layer overlying a subphase of finite depth}, volume={369}, DOI={10.1017/S0022112098001980}, journal={Journal of Fluid Mechanics}, author={Stone, Howard A. and Ajdari, Armand}, year={1998}, pages={151–173}}

@article{SaT,
title = {Navier–Stokes equations on Riemannian manifolds},
journal = {Journal of Geometry and Physics},
volume = {148},
pages = {103543},
year = {2020},
issn = {0393-0440},
doi = {https://doi.org/10.1016/j.geomphys.2019.103543},
url = {https://www.sciencedirect.com/science/article/pii/S0393044019302244},
author = {Maryam Samavaki and Jukka Tuomela}
}

@article{FMR,
title = {Generalized Navier–Stokes equations and soft hairy horizons in fluid/gravity correspondence},
journal = {Nuclear Physics B},
volume = {973},
pages = {115603},
year = {2021},
issn = {0550-3213},
doi = {https://doi.org/10.1016/j.nuclphysb.2021.115603},
url = {https://www.sciencedirect.com/science/article/pii/S055032132100300X},
author = {A.J. Ferreira–Martins and R. {da Rocha}}
}

@book{Gr,
 author = {Grisvard, P.},
 title = {Elliptic problems in nonsmooth domains},
 fseries = {Monographs and Studies in Mathematics},
 series = {Monogr. Stud. Math.},
 volume = {24},
 isbn = {0-273-08647-2},
 year = {1985},
 publisher = {Pitman, Boston, MA},
 language = {English},
 zbMATH = {4138299},
 Zbl = {0695.35060}
}

@Book{temam,
 Author = {Temam, Roger},
 Title = {Navier-{Stokes} equations. {Theory} and numerical analysis. {Rev}. ed},
 FSeries = {Studies in Mathematics and its Applications},
 Series = {Stud. Math. Appl.},
 ISSN = {0168-2024},
 Volume = {2},
 Year = {1979},
 Publisher = {Elsevier, Amsterdam},
 Language = {English},
 zbMATH = {3663609},
 Zbl = {0426.35003}
}

@article{FGKU,
 author = {Feizmohammadi, Ali and Ghosh, Tuhin and Krupchyk, Katya and Uhlmann, Gunther},
 title = {Fractional anisotropic {Calder{\'o}n} problem on closed {Riemannian} manifolds},
 fjournal = {Journal of Differential Geometry},
 journal = {J. Differ. Geom.},
 issn = {0022-040X},
 volume = {131},
 number = {2},
 pages = {401--414},
 year = {2025},
 language = {English},
 doi = {10.4310/jdg/1757353909},
 url = {projecteuclid.org/journals/journal-of-differential-geometry/volume-131/issue-2/Fractional-anisotropic-Calder%c3%b3n-problem-on-closed-Riemannian-manifolds/10.4310/jdg/1757353909.full},
 zbMATH = {8097552}
}

@article{HLYZ,
 author = {Helin, Tapio and Lassas, Matti and Ylinen, Lauri and Zhang, Zhidong},
 title = {Inverse problems for heat equation and space-time fractional diffusion equation with one measurement},
 fjournal = {Journal of Differential Equations},
 journal = {J. Differ. Equations},
 issn = {0022-0396},
 volume = {269},
 number = {9},
 pages = {7498--7528},
 year = {2020},
 language = {English},
 doi = {10.1016/j.jde.2020.05.022},
 url = {hdl.handle.net/10138/344428},
 zbMATH = {7212520},
 Zbl = {1443.35168}
}

@article{KOSY,
 author = {Kian, Y. and Oksanen, L. and Soccorsi, E. and Yamamoto, M.},
 title = {Global uniqueness in an inverse problem for time fractional diffusion equations},
 fjournal = {Journal of Differential Equations},
 journal = {J. Differ. Equations},
 issn = {0022-0396},
 volume = {264},
 number = {2},
 pages = {1146--1170},
 year = {2018},
 language = {English},
 doi = {10.1016/j.jde.2017.09.032},
 zbMATH = {6807494},
 Zbl = {1376.35099}
}

@article{KaKuLaMa,
 author = {Katchalov, A. and Kurylev, Y. and Lassas, M. and Mandache, N.},
 title = {Equivalence of time-domain inverse problems and boundary spectral problems.},
 fjournal = {Inverse Problems},
 journal = {Inverse Probl.},
 issn = {0266-5611},
 volume = {20},
 number = {2},
 pages = {419--436},
 year = {2004},
 language = {English},
 doi = {10.1088/0266-5611/20/2/007},
 zbMATH = {2094176},
 Zbl = {1073.35209}
}

@article{KuLaSo,
 author = {Kurylev, Yaroslav and Lassas, Matti and Somersalo, Erkki},
 title = {Maxwell's equations with a polarization independent wave velocity: direct and inverse problems},
 fjournal = {Journal de Math{\'e}matiques Pures et Appliqu{\'e}es. Neuvi{\`e}me S{\'e}rie},
 journal = {J. Math. Pures Appl. (9)},
 issn = {0021-7824},
 volume = {86},
 number = {3},
 pages = {237--270},
 year = {2006},
 language = {English},
 doi = {10.1016/j.matpur.2006.01.008},
 zbMATH = {5183145},
 Zbl = {1134.35107}
}

@article{BeIs,
 author = {Belishev, M. I. and Isakov, V. M.},
 title = {On the uniqueness of the recovery of parameters of the {Maxwell} system from dynamical boundary data},
 fjournal = {Journal of Mathematical Sciences (New York)},
 journal = {J. Math. Sci., New York},
 issn = {1072-3374},
 volume = {122},
 number = {5},
 pages = {3459--3469},
 year = {2002},
 language = {English},
 doi = {10.1023/B:JOTH.0000034024.38243.02},
 zbMATH = {2183294},
 Zbl = {1082.35161}
}

@Misc{rudin_1,
 Author = {Rudin, Walter},
 Title = {Principles of mathematical analysis. 3rd ed},
 Year = {1976},
 Language = {English},
 HowPublished = {International {Series} in {Pure} and {Applied} {Mathematics}. {D{\"u}sseldorf} etc.: {McGraw}-{Hill} {Book} {Company}. {X}, 342 p. {DM} 47.80 (1976).},
 zbMATH = {3539473},
 Zbl = {0346.26002}
}

@article{KLLY,
 author = {Kian, Yavar and Li, Zhiyuan and Liu, Yikan and Yamamoto, Masahiro},
 title = {The uniqueness of inverse problems for a fractional equation with a single measurement},
 fjournal = {Mathematische Annalen},
 journal = {Math. Ann.},
 issn = {0025-5831},
 volume = {380},
 number = {3-4},
 pages = {1465--1495},
 year = {2021},
 language = {English},
 doi = {10.1007/s00208-020-02027-z},
 url = {hdl.handle.net/2115/82210},
 zbMATH = {7387655},
 Zbl = {1472.35455}
}

@Book{boyer,
 Author = {Boyer, Franck and Fabrie, Pierre},
 Title = {Mathematical tools for the study of the incompressible {Navier}-{Stokes} equations and related models},
 FSeries = {Applied Mathematical Sciences},
 Series = {Appl. Math. Sci.},
 ISSN = {0066-5452},
 Volume = {183},
 ISBN = {978-1-4614-5974-3; 978-1-4614-5975-0},
 Year = {2013},
 Publisher = {New York, NY: Springer},
 Language = {English},
 DOI = {10.1007/978-1-4614-5975-0},
 URL = {www.gbv.de/dms/goettingen/730899411.pdf},
 zbMATH = {6106780},
 Zbl = {1286.76005}
}

@Book{KKL,
 Author = {Katchalov, Alexander and Kurylev, Yaroslav and Lassas, Matti},
 Title = {Inverse boundary spectral problems},
 FSeries = {Chapman \& Hall/CRC Monographs and Surveys in Pure and Applied Mathematics},
 Series = {Chapman Hall/CRC Monogr. Surv. Pure Appl. Math.},
 Volume = {123},
 ISBN = {1-58488-005-8},
 Year = {2001},
 Publisher = {Boca Raton, FL: CRC Press},
 Language = {English},
 zbMATH = {1674059},
 Zbl = {1037.35098}
}

@Book{jost,
 Author = {Jost, J{\"u}rgen},
 Title = {Riemannian geometry and geometric analysis},
 Edition = {7th edition},
 FSeries = {Universitext},
 Series = {Universitext},
 ISSN = {0172-5939},
 ISBN = {978-3-319-61859-3; 978-3-319-61860-9},
 Year = {2017},
 Publisher = {Cham: Springer},
 Language = {English},
 DOI = {10.1007/978-3-319-61860-9},
 zbMATH = {6755669},
 Zbl = {1380.53001}
}

@Book{Lions1,
 Author = {Lions, J. L. and Magenes, E.},
 Title = {Non-homogeneous boundary value problems and applications. {Vol}. {I}. {Translated} from the {French} by {P}. {Kenneth}},
 FSeries = {Grundlehren der Mathematischen Wissenschaften},
 Series = {Grundlehren Math. Wiss.},
 ISSN = {0072-7830},
 Volume = {181},
 Year = {1972},
 Publisher = {Springer, Cham},
 Language = {English},
 zbMATH = {3353865},
 Zbl = {0223.35039}
}

@Book{deRham,
 Author = {de Rham, Georges},
 Title = {Differentiable manifolds. {Forms}, currents, harmonic forms. {Transl}. from the {French} by {F}. {R}. {Smith}. {Introduction} to the {English} ed. by {S}. {S}. {Chern}},
 FSeries = {Grundlehren der Mathematischen Wissenschaften},
 Series = {Grundlehren Math. Wiss.},
 ISSN = {0072-7830},
 Volume = {266},
 Year = {1984},
 Publisher = {Springer, Cham},
 Language = {English},
 zbMATH = {3848254},
 Zbl = {0534.58003}
}

@Book{frankel,
 Author = {Frankel, Theodore},
 Title = {The geometry of physics. {An} introduction},
 Edition = {3rd ed.},
 ISBN = {978-1-107-60260-1; 978-1-139-15414-7},
 Year = {2011},
 Publisher = {Cambridge: Cambridge University Press},
 Language = {English},
 zbMATH = {6007815},
 Zbl = {1250.58001}
}

@Article{eller,
 Author = {Eller, Matthias and Toundykov, Daniel},
 Title = {A global {Holmgren} theorem for multidimensional hyperbolic partial differential equations},
 FJournal = {Applicable Analysis},
 Journal = {Appl. Anal.},
 ISSN = {0003-6811},
 Volume = {91},
 Number = {1-2},
 Pages = {69--90},
 Year = {2012},
 Language = {English},
 DOI = {10.1080/00036811.2010.538685},
 zbMATH = {6017982},
 Zbl = {1235.35058}
}

@Book{Schwarz,
 Author = {Schwarz, G{\"u}nter},
 Title = {Hodge decomposition. {A} method for solving boundary value problems},
 FSeries = {Lecture Notes in Mathematics},
 Series = {Lect. Notes Math.},
 ISSN = {0075-8434},
 Volume = {1607},
 ISBN = {3-540-60016-7},
 Year = {1995},
 Publisher = {Berlin: Springer Verlag},
 Language = {English},
 DOI = {10.1007/BFb0095978},
 zbMATH = {783743},
 Zbl = {0828.58002}
}

@Book{Bochner,
 Author = {Mikusinski, Jan},
 Title = {The {Bochner} integral},
 FSeries = {Lehrb{\"u}cher und Monographien aus dem Gebiete der exakten Wissenschaften. Mathematische Reihe},
 Series = {Lehrb. Monogr. Geb. Exakten Wiss., Math. Reihe},
 Volume = {55},
 Year = {1978},
 Publisher = {Birkh{\"a}user Verlag, Basel},
 Language = {English},
 zbMATH = {3575883},
 Zbl = {0369.28010}
}

@Book{Hislop,
 Author = {Hislop, P. D. and Sigal, I. M.},
 Title = {Introduction to spectral theory. {With} applications to {Schr{\"o}dinger} operators},
 FSeries = {Applied Mathematical Sciences},
 Series = {Appl. Math. Sci.},
 ISSN = {0066-5452},
 Volume = {113},
 ISBN = {0-387-94501-6},
 Year = {1996},
 Publisher = {New York, NY: Springer-Verlag},
 Language = {English},
 zbMATH = {837307},
 Zbl = {0855.47002}
}

@article{Saksala2018,
 author = {Helin, Tapio and Lassas, Matti and Oksanen, L. and Saksala, Teemu},
 title = {Correlation based passive imaging with a white noise source},
 fjournal = {Journal de Math{\'e}matiques Pures et Appliqu{\'e}es. Neuvi{\`e}me S{\'e}rie},
 journal = {J. Math. Pures Appl. (9)},
 issn = {0021-7824},
 volume = {116},
 pages = {132--160},
 year = {2018},
 language = {English},
 doi = {10.1016/j.matpur.2018.05.001},
 zbMATH = {6904601},
 Zbl = {1401.35343}
}

@Book{doCarmo,
 Author = {do Carmo, Manfredo Perdig{\~a}o},
 Title = {Riemannian geometry. {Translated} from the {Portuguese} by {Francis} {Flaherty}},
 ISBN = {0-8176-3490-8},
 Year = {1992},
 Publisher = {Boston, MA etc.: Birkh{\"a}user},
 Language = {English},
 zbMATH = {52737},
 Zbl = {0752.53001}
}

@article{Oksanen2014,
 author = {Lassas, Matti and Oksanen, Lauri},
 title = {Inverse problem for the {Riemannian} wave equation with {Dirichlet} data and {Neumann} data on disjoint sets},
 fjournal = {Duke Mathematical Journal},
 journal = {Duke Math. J.},
 issn = {0012-7094},
 volume = {163},
 number = {6},
 pages = {1071--1103},
 year = {2014},
 language = {English},
 doi = {10.1215/00127094-2649534},
 zbMATH = {6298159},
 Zbl = {1375.35634}
}

@article{Isozaki2010,
 author = {Isozaki, Hiroshi and Kurylev, Yaroslav and Lassas, Matti},
 title = {Forward and inverse scattering on manifolds with asymptotically cylindrical ends},
 fjournal = {Journal of Functional Analysis},
 journal = {J. Funct. Anal.},
 issn = {0022-1236},
 volume = {258},
 number = {6},
 pages = {2060--2118},
 year = {2010},
 language = {English},
 doi = {10.1016/j.jfa.2009.11.009},
 url = {hdl.handle.net/2241/104845},
 zbMATH = {5676161},
 Zbl = {1204.58023}
}

@misc{Dugundji66,
 author = {Dugundji, James},
 title = {Topology},
 year = {1966},
 language = {English},
 howpublished = {Series in {Advanced} {Mathematics}. {Boston}: {Allyn} and {Bacon}, {Inc}. {XVI}, 447 p. (1966).},
 zbMATH = {3232606},
 Zbl = {0144.21501}
}

@book{Leoni17,
 author = {Leoni, Giovanni},
 title = {A first course in {Sobolev} spaces},
 edition = {2nd edition},
 fseries = {Graduate Studies in Mathematics},
 series = {Grad. Stud. Math.},
 issn = {1065-7339},
 volume = {181},
 isbn = {978-1-4704-2921-8; 978-1-4704-4226-2},
 year = {2017},
 publisher = {Providence, RI: American Mathematical Society (AMS)},
 language = {English},
 doi = {10.1090/gsm/181},
 zbMATH = {6821285},
 Zbl = {1382.46001}
}

@book{Lee18,
 author = {Lee, John M.},
 title = {Introduction to {Riemannian} manifolds},
 edition = {2nd edition},
 fseries = {Graduate Texts in Mathematics},
 series = {Grad. Texts Math.},
 issn = {0072-5285},
 volume = {176},
 isbn = {978-3-319-91754-2; 978-3-319-91755-9},
 year = {2018},
 publisher = {Cham: Springer},
 language = {English},
 doi = {10.1007/978-3-319-91755-9},
 zbMATH = {6897812},
 Zbl = {1409.53001}
}

@book{Isakov,
 author = {Isakov, Victor},
 title = {Inverse problems for partial differential equations},
 edition = {3rd edition},
 fseries = {Applied Mathematical Sciences},
 series = {Appl. Math. Sci.},
 issn = {0066-5452},
 volume = {127},
 isbn = {978-3-319-51657-8; 978-3-319-51658-5},
 year = {2017},
 publisher = {Cham: Springer},
 language = {English},
 doi = {10.1007/978-3-319-51658-5},
 zbMATH = {6684667},
 Zbl = {1366.65087}
}

@incollection{EINT02,
 author = {Eller, M. and Isakov, V. and Nakamura, G. and Tataru, D.},
 title = {Uniqueness and stability in the {Cauchy} problem for {Maxwell} and elasticity systems.},
 booktitle = {Nonlinear partial differential equations and their applications. Coll\`ege de France seminar. Vol. XIV. Lectures held at the J. L. Lions seminar on applied mathematics, Paris, France, 1997--1998},
 isbn = {0-444-51103-2},
 pages = {329--349},
 year = {2002},
 publisher = {Amsterdam: Elsevier},
 language = {English},
 zbMATH = {1824036},
 Zbl = {1038.35159}
}

@book{RG_Petersen,
 author = {Petersen, Peter},
 title = {Riemannian geometry},
 edition = {3rd edition},
 fseries = {Graduate Texts in Mathematics},
 series = {Grad. Texts Math.},
 issn = {0072-5285},
 volume = {171},
 isbn = {978-3-319-26652-7; 978-3-319-26654-1},
 year = {2016},
 publisher = {Cham: Springer},
 language = {English},
 doi = {10.1007/978-3-319-26654-1},
 zbMATH = {6520113},
 Zbl = {1417.53001}
}

@article{AKKLT04,
 author = {Anderson, Michael and Katsuda, Atsushi and Kurylev, Yaroslav and Lassas, Matti and Taylor, Michael},
 title = {Boundary regularity for the {Ricci} equation, geometric convergence, and {Gel}'fand's inverse boundary problem},
 fjournal = {Inventiones Mathematicae},
 journal = {Invent. Math.},
 issn = {0020-9910},
 volume = {158},
 number = {2},
 pages = {261--321},
 year = {2004},
 language = {English},
 doi = {10.1007/s00222-004-0371-6},
 zbMATH = {2132012},
 Zbl = {1177.35245}
}

@article{KK98,
 author = {Katchalov, A. and Kurylev, Ya.},
 title = {Multidimensional inverse problem with incomplete boundary spectral data},
 fjournal = {Communications in Partial Differential Equations},
 journal = {Commun. Partial Differ. Equations},
 issn = {0360-5302},
 volume = {23},
 number = {1-2},
 pages = {55--95},
 year = {1998},
 language = {English},
 doi = {10.1080/03605309808821337},
 zbMATH = {1143851},
 Zbl = {0904.65114}
}

@article{LNOY24,
 author = {Lassas, Matti and Nursultanov, Medet and Oksanen, Lauri and Ylinen, Lauri},
 title = {Disjoint data inverse problem on manifolds with quantum chaos bounds},
 fjournal = {SIAM Journal on Mathematical Analysis},
 journal = {SIAM J. Math. Anal.},
 issn = {0036-1410},
 volume = {56},
 number = {6},
 pages = {7748--7779},
 year = {2024},
 language = {English},
 doi = {10.1137/23M1606897},
 zbMATH = {7957213},
 Zbl = {1554.35406}
}

@misc{SS25,
 author = {Saksala, Teemu and Shedlock, Andrew},
 title = {An {Inverse} {Problem} for symmetric hyperbolic {Partial} {Differential} {Operators} on {Complete} {Riemannian} {Manifolds}},
 year = {2025},
 howpublished = {Preprint, {arXiv}:2503.14676 [math.{AP}] (2025)},
 url = {https://arxiv.org/abs/2503.14676},
 arXiv = {arXiv:2503.14676}
}

@article{I93,
 author = {Isakov, V.},
 title = {On uniqueness in inverse problems for semilinear parabolic equations},
 fjournal = {Archive for Rational Mechanics and Analysis},
 journal = {Arch. Ration. Mech. Anal.},
 issn = {0003-9527},
 volume = {124},
 number = {1},
 pages = {1--12},
 year = {1993},
 language = {English},
 doi = {10.1007/BF00392201},
 zbMATH = {474498},
 Zbl = {0804.35150}
}

@article{IS94,
 author = {Isakov, Victor and Sylvester, John},
 title = {Global uniqueness for a semilinear elliptic inverse problem},
 fjournal = {Communications on Pure and Applied Mathematics},
 journal = {Commun. Pure Appl. Math.},
 issn = {0010-3640},
 volume = {47},
 number = {10},
 pages = {1403--1410},
 year = {1994},
 language = {English},
 doi = {10.1002/cpa.3160471005},
 zbMATH = {714611},
 Zbl = {0817.35126}
}

@misc{L25,
 author = {Lassas, Matti},
 title = {Introduction to inverse problems for non-linear partial differential equations},
 year = {2025},
 howpublished = {Preprint, {arXiv}:2503.12448 [math.{AP}] (2025)},
 url = {https://arxiv.org/abs/2503.12448},
 arXiv = {arXiv:2503.12448}
}

@article{B90,
 author = {Belishev, M. I.},
 title = {Wave bases in multidimensional inverse problems},
 fjournal = {Mathematics of the USSR, Sbornik},
 journal = {Math. USSR, Sb.},
 issn = {0025-5734},
 volume = {67},
 number = {1},
 pages = {23--42},
 year = {1990},
 language = {English},
 doi = {10.1070/SM1990v067n01ABEH001185},
 zbMATH = {4144611},
 Zbl = {0698.35164}
}

@article{ABI92,
 author = {Avdonin, S. A. and Belishev, M. I. and Ivanov, S. A.},
 title = {Boundary control and a matrix inverse problem for the equation {{\(u_{tt}- u_{xx}+V(x)u=0\)}}},
 fjournal = {Mathematics of the USSR, Sbornik},
 journal = {Math. USSR, Sb.},
 issn = {0025-5734},
 volume = {72},
 number = {2},
 pages = {287--310},
 year = {1992},
 language = {English},
 doi = {10.1070/SM1992v072n02ABEH002141},
 zbMATH = {11808},
 Zbl = {0782.93054}
}

@article{BK89,
 author = {Belishev, M. I. and Kachalov, A. P.},
 title = {The methods of boundary control theory in the inverse spectral problem for an inhomogeneous string},
 fjournal = {Journal of Soviet Mathematics},
 journal = {J. Sov. Math.},
 issn = {0090-4104},
 volume = {57},
 number = {3},
 pages = {3072--3077},
 year = {1989},
 language = {English},
 doi = {10.1007/BF01098970},
 zbMATH = {54369},
 Zbl = {0741.73031}
}

@article{BK87,
 author = {Belishev, M. I. and Kurylev, Ya. V.},
 title = {Nonstationary inverse problem for the multidimensional wave equation ``in the large''},
 fjournal = {Zapiski Nauchnykh Seminarov Leningradskogo Otdeleniya Matematicheskogo Instituta Imeni V. A. Steklova},
 journal = {Zap. Nauchn. Semin. Leningr. Otd. Mat. Inst. Steklova},
 issn = {0373-2703},
 volume = {165},
 pages = {21--30},
 year = {1987},
 language = {Russian},
 url = {https://eudml.org/doc/188393},
 zbMATH = {4142675},
 Zbl = {0697.35163}
}

@article{BKu89,
 author = {Belishev, M. I. and Kurylev, Ya. V.},
 title = {The inverse spectral problem of the scattering of plane waves in a half- space with local inhomogeneity},
 fjournal = {U.S.S.R. Computational Mathematics and Mathematical Physics},
 journal = {U.S.S.R. Comput. Math. Math. Phys.},
 issn = {0041-5553},
 volume = {29},
 number = {4},
 pages = {56--64},
 year = {1989},
 language = {English},
 doi = {10.1016/0041-5553(89)90116-X},
 zbMATH = {4171554},
 Zbl = {0712.35106}
}

@article{B87,
 author = {Belishev, M. I.},
 title = {Equations of {Gel}'fand-{Levitan} type in multidimensional inverse problem for the wave equation},
 fjournal = {Zapiski Nauchnykh Seminarov Leningradskogo Otdeleniya Matematicheskogo Instituta Imeni V. A. Steklova},
 journal = {Zap. Nauchn. Semin. Leningr. Otd. Mat. Inst. Steklova},
 issn = {0373-2703},
 volume = {165},
 pages = {15--20},
 year = {1987},
 language = {Russian},
 url = {https://eudml.org/doc/188392},
 zbMATH = {4138350},
 Zbl = {0695.35111}
}

@article{BK91,
 author = {Belishev, Michael I. and Kurylev, Yaroslav V.},
 title = {Boundary control, wave field continuation and inverse problems for the wave equation},
 fjournal = {Computers \& Mathematics with Applications},
 journal = {Comput. Math. Appl.},
 issn = {0898-1221},
 volume = {22},
 number = {4-5},
 pages = {27--52},
 year = {1991},
 language = {English},
 doi = {10.1016/0898-1221(91)90130-V},
 zbMATH = {21803},
 Zbl = {0768.35077}
}

@book{Theodore11,
 author = {Frankel, Theodore},
 title = {The geometry of physics. {An} introduction},
 edition = {3rd ed.},
 isbn = {978-1-107-60260-1; 978-1-139-15414-7},
 year = {2011},
 publisher = {Cambridge: Cambridge University Press},
 language = {English},
 zbMATH = {6007815},
 Zbl = {1250.58001}
}

@article{LK09,
 author = {Kurylev, Yaroslav and Lassas, Matti},
 title = {Inverse problems and index formulae for {Dirac} operators},
 fjournal = {Advances in Mathematics},
 journal = {Adv. Math.},
 issn = {0001-8708},
 volume = {221},
 number = {1},
 pages = {170--216},
 year = {2009},
 language = {English},
 doi = {10.1016/j.aim.2008.12.001},
 zbMATH = {5550660},
 Zbl = {1163.35041}
}

@article{KOP18,
 author = {Kurylev, Yaroslav and Oksanen, Lauri and Paternain, Gabriel P.},
 title = {Inverse problems for the connection {Laplacian}},
 fjournal = {Journal of Differential Geometry},
 journal = {J. Differ. Geom.},
 issn = {0022-040X},
 volume = {110},
 number = {3},
 pages = {457--494},
 year = {2018},
 language = {English},
 doi = {10.4310/jdg/1542423627},
 url = {www.repository.cam.ac.uk/handle/1810/274613},
 zbMATH = {6982217},
 Zbl = {1415.53024}
}

@article{CRT99,
 author = {Cao, Chongsheng and Rammaha, Mohammad A. and Titi, Edriss S.},
 title = {The {Navier}-{Stokes} equations on the rotating 2-{D} sphere: {Gevrey} regularity and asymptotic degrees of freedom},
 fjournal = {ZAMP. Zeitschrift f{\"u}r angewandte Mathematik und Physik},
 journal = {Z. Angew. Math. Phys.},
 issn = {0044-2275},
 volume = {50},
 number = {3},
 pages = {341--360},
 year = {1999},
 language = {English},
 doi = {10.1007/PL00001493},
 zbMATH = {1320485},
 Zbl = {0928.35120}
}

@article{CP92,
 author = {Carverhill, Andrew and Pedit, Franz J.},
 title = {Global solutions of the {Navier}-{Stokes} equation with strong viscosity},
 fjournal = {Annals of Global Analysis and Geometry},
 journal = {Ann. Global Anal. Geom.},
 issn = {0232-704X},
 volume = {10},
 number = {3},
 pages = {255--261},
 year = {1992},
 language = {English},
 doi = {10.1007/BF00136868},
 zbMATH = {109070},
 Zbl = {0766.58011}
}

@article{TW93,
 author = {Temam, Roger and Wang, Shouhong},
 title = {Inertial forms of {Navier}-{Stokes} equations on the sphere},
 fjournal = {Journal of Functional Analysis},
 journal = {J. Funct. Anal.},
 issn = {0022-1236},
 volume = {117},
 number = {1},
 pages = {215--242},
 year = {1993},
 language = {English},
 doi = {10.1006/jfan.1993.1126},
 zbMATH = {515228},
 Zbl = {0801.35109}
}

@article{AA91,
 author = {Il'in, A. A.},
 title = {The {Navier}-{Stokes} and {Euler} equations on two-dimensional closed manifolds},
 fjournal = {Mathematics of the USSR, Sbornik},
 journal = {Math. USSR, Sb.},
 issn = {0025-5734},
 volume = {69},
 number = {2},
 pages = {559--579},
 year = {1991},
 language = {English},
 doi = {10.1070/SM1991v069n02ABEH002116},
 zbMATH = {4194328},
 Zbl = {0724.35088}
}

@book{Ruelle89,
 author = {Ruelle, David},
 title = {Elements of differentiable dynamics and bifurcation theory},
 isbn = {0-12-601710-7},
 year = {1989},
 publisher = {Boston, MA etc.: Academic Press, Inc.},
 language = {English},
 zbMATH = {42418},
 Zbl = {0684.58001}
}
\bibliographystyle{abbrv}

\end{document}